\documentclass[11pt,a4paper,twoside]{article}
\usepackage{amssymb}
\usepackage{amsfonts}
\usepackage{geometry}
\usepackage{color}
\usepackage{mathtools}
\usepackage{amsmath,amsthm,amssymb,amsfonts,enumitem}
\usepackage{CJK}
\usepackage{latexsym}
\usepackage{graphicx}
\usepackage{pgfpages}
\usepackage{mathrsfs}
\usepackage{ytableau}
\usepackage{enumerate}
\usepackage{cite}
\usepackage[hyphens]{url}
\usepackage[bookmarks=false,hidelinks]{hyperref}
\usepackage[affil-it]{authblk}
\usepackage{appendix}
\usepackage{amsmath}
\allowdisplaybreaks

\newtheorem{definition}{Definition}[section]
\newtheorem{theorem}[definition]{Theorem}
\newtheorem{lemma}[definition]{Lemma}
\newtheorem{proposition}[definition]{Proposition}
\newtheorem{corollary}[definition]{Corollary}
\newtheorem{remark}[definition]{Remark}
\newtheorem{condition}{Assumption}

\numberwithin{equation}{section}

\begin{document}
	\title{Invariant and periodic measures in classical spin systems on infinite lattices with highly degenerate noise 
    %driven stochastic differential equations
    }
	\author[a]{Tong Lu,}
    \author[b,a]{Huaizhong Zhao}
    \affil[a] {Research Centre for Mathematics and Interdisciplinary Sciences, Shandong University, Qingdao 266237, China}
    \affil[b]{Department of Mathematical Sciences, Durham University, DH1 3LE, UK}
    \affil[ ]{lutong1019@mail.sdu.edu.cn, huaizhong.zhao@durham.ac.uk}
	\date{}
	\maketitle
	
	\begin{abstract}
		In this paper, we consider the classical spin systems on unbounded lattices given by infinite-dimensional stochastic differential equations (SDEs). We assume that the stochastic forcing acts only on one particle. The other particles are not subject to stochastic forcing directly, but interact with their nearest neighbouring particles. Under the above highly degenerate noise setting, with some mild assumptions on the local interaction of each particle such as weak dissipation,
        we obtain the existence, uniqueness and the Markovian property of weak martingale solutions. We prove that the one-dimensional noise can propagate to any spin particle in the system in the sense that there exists a unique invariant/periodic measure and geometric ergodicity holds for the Markovian system when restricted to any finite volume. We then prove the finite-dimensional invariant measure and the average of lifted periodic measure are tight, and weak convergent subsequence gives an invariant and periodic measures of the infinite spin systems, respectively, in the time-homogeneous or time-periodic cases.

		\medskip
		
		\noindent
{\bf Keywords:} Spin systems;  martingale problem; highly degenerate noise; weak dissipativity; invariant measures; periodic measures; Markovian processes
	\end{abstract}
	
	\tableofcontents

    \section{Introduction}
    Spin particle systems on lattices have a long and significant research history. Here, spin is the value of the configuration at each atom of the infinite lattice and represents the state of each of the infinite number of individual particles. 
    The dynamics in the continuous case are normally modelled by stochastic differential equations (SDEs) with solutions taking values in an infinite-dimensional space, while in the discrete case, it is normally described as a Markov chain, e.g. in the infinite ferromagnetic Ising model. Its invariant measure, if it exists, gives the equilibrium of configurations and is a Gibbs state if a certain consistency condition is satisfied when conditioned each finite subsystem (see Holley and Stroock \cite{holley1981diffusions,holley1976l2,H1987}). It can be investigated as the long-time limit in the sense distributions and through weak approximations by a sequence of finite subsystems. 
    Therefore, studies of invariant measures of infinite-dimensional SDEs are of physical significance and have attracted considerable attention from mathematicians and physicists. Periodically driven spin systems have emerged as a fundamental paradigm in nonequilibrium statistical mechanics, exhibiting rich physical phenomena and important applications in Floquet engineering, time crystals, and topological phase transitions (see Abanin et al. \cite{abanindehuveexponentiallyslow2015}, Goldman and Dalibard \cite{goldmandalibard2014}), but there is little understanding of the periodic measure or Gibbs state with time periodic pattern, which, though, is highly expected given the recent work on random periodicity that we will mention below.
   In this paper, we take only the first step in this analysis to carry out the existence of periodic measures of spin systems on the infinite lattice.
    
    To investigate the long-term behaviour of spin systems, two primary approaches have been developed. One is to interpret the classical Gibbs state as an invariant measure of a stochastic Langevin equation via ``stochastic quantization method"  (see Parisi and Wu  \cite{parisi1980perturbation}, Damgaard and  H{\"u}ffel\cite{damgaard1988stochastic}, Albeverio et al.  \cite{albeverio1994stochastic,albeverio2009statistical,albeverio1991stochastic}); the other one is to describe the particle dynamics through the generator of a Markov process, thereby characterizing the long-time behaviour of the spin system via the ergodic properties of the process (see Holley and Stroock \cite{H1987}, Zegarlinski \cite{zegarlinski1996strong}, Liggett \cite{Liggett1985Interacting}, Chen \cite{chen2004from}, Xu and Zegarlinski \cite{xu2010ergodicity}).  
    In particular, Da Prato and Zabczyk \cite{da1996ergodicity} employed the theory of dissipative operators in weighted Banach spaces to study the asymptotics of the associated transition semigroups. Albeverio et al. \cite{albeverio1994stochastic} and Zegarlinski \cite{zegarlinski1996strong} developed methodologies to construct a sequence of stochastic equations defined on finite volumes. In these works, convergence and ergodicity properties are established through the application of dissipative properties and the logarithmic Sobolev inequality, respectively. Won \cite{won20212} developed an $L^2$-approximation strategy to study the Markov semi-group generated by an infinite system of elliptic diffusion processes on a lattice. Bakhtin and Chen \cite{bakhtin2022dynamic} used Galerkin approximation techniques combined with infinite-volume polymer measures to prove the existence and uniqueness of invariant measures in polymer dynamics. In a related development, Xu and Zegarlinski \cite{xu2010ergodicity} examined the existence of invariant measures for an infinite-dimensional stochastic system driven by white $\alpha$-stable noise.
    %by perturbing Ornstein-Uhlenbeck $\alpha$-stable processes. 
    These works were all conducted under the consideration of non-degenerate noise. {\v{Z}}{\'a}k \cite{vzak2016existence} addressed the existence of equilibrium states for interacting particle systems linked to a specific class of degenerate elliptic operators using a probabilistic technique. However, it should be noted that the adopted noise structure is based on infinitely many noise sources and still falls short of achieving high degeneracy. %and incorporates nonlinearities with linear growth characteristic.

    In this work, we focus on the second approach to investigate the long-time behaviour of a specific spin system.  Its generator acts on any twice Fr\'{e}chet differentiable function $h$ on $\mathbb{Z}^d$ belonging to a properly chosen Hilbert space (to be rigorously defined in Chapter 3) as follows
	$$
	\mathscr{L}h(x):=(a_{1,2} x_2+f(x_1))\partial_{1}h(x)+\sum_{i=2}^{\infty}(a_{i,i-1}x_{i-1}+a_{i,i+1} x_{i+1}+f(x_i))\partial_{i}h(x)+ \frac{1}{2} {\partial_{1}}\partial_{1}h(x),
	$$
    where $\partial_i$ denotes the G\^{a}teaux derivative along a family of orthogonal basis directions.
    This generates an infinite-dimensional SDE 
    \begin{equation}
    \label{infinit sde1}
			dX(t)=\left(AX(t)+F(X(t))\right)dt+BdW(t) ,  \ t\geq 0,
	\end{equation}
    on the chosen space, where 
    \begin{equation*}
	     \begin{cases}
		A(x)=\left(a_{\gamma,\gamma-1}x_{\gamma-1}+a_{\gamma,\gamma+1}x_{\gamma+1}\right)_\gamma,x=\left(x_\gamma\right)_\gamma,\\
		F(x)=\left(f_\gamma(x_\gamma)\right)_\gamma,x=\left(x_\gamma\right)_\gamma,\\
		B=(b_{\gamma,j})_{\gamma,j}, \quad \text{with} \quad b_{1,1}=1,b_{\gamma,j}=0, \gamma,j\in \left\{\mathbb{Z}^++1\right\}.
	\end{cases}
	\end{equation*}
    The noise term in this equation is highly degenerate, acting only on the first particle, with the stochasticity propagated to other particles through interactions with nearest neighbours. Since our conditions hold uniformly for all $f_i$ and each local interaction term $f_i$ plays a symmetric role, we denote them collectively as $f$ for notational simplicity, so that $F(x)=\left(f(x_\gamma)\right)_\gamma$. 
    %We provide a physical interpretation of the above degenerate noise-driven SDE, where the noise acts exclusively on a single particle and propagates throughout the system via local interparticle interactions. 
    %Assume the local interaction term is of weak dissipation. 
    For simplicity of presentation, we restrict our discussion to the results concerning $\mathbb{Z}^+$, as the general case for arbitrary $\mathbb{Z}^d$ would require more elaborate formulations.
    %but there is no essential difference.  
    %We assume that $B$ is degenerate, that $A$ has some form of non-degeneracy, and that $F$ is weakly dissipative. 
    
    The first objective of this paper is to establish the existence and uniqueness of the solution to (\ref{infinit sde1}), and further to demonstrate the existence of invariant measures. We prove that the one-dimensional noise can propagate to any spin particle in the system in the sense that there exists a unique invariant/periodic measure and geometric ergodicity holds for the Markovian system when restricted to any finite volume. The solution is given in the sense of a weak martingale solution or the solution of a martingale problem. The weak solution framework makes it convenient to consider the approximation of the infinite lattice by a sequence of finite lattices. We first consider the spin system restricted to finite lattices and then study the limit when these bounded lattices increase to the infinite lattice. We prove that when on a finite lattice, when the volume is getting sufficiently large, the spin on any fixed atom becomes asymptotically the same. This result leads to the convergence in law and their limit is the unique solution to the martingale problem and is Markovian. Then the tightness of the invariant measures on the finite-volume lattices gives the result of a weak limit which is the invariant measure of the spin system on the infinite lattice. It is important to note that this paper does not address the uniqueness of the invariant measure.

   The second aim of this paper is to study infinite-dimensional time-inhomogeneous stochastic differential equations as follows:
    \begin{equation}
      \label{introdunctionperiodicinfinite}
      dX(t)=\left(A(t)X(t)+F(t,X(t))\right)dt+BdW(t) ,  \ t\geq s,
    \end{equation}
    where both $A(t)x=\left(a_{\gamma,\gamma-1}(t)x_{\gamma-1}+a_{\gamma,\gamma+1}(t)x_{\gamma+1}\right)_\gamma$ and $F(t,x)=\left(f(t,x_\gamma)\right)_\gamma$ coefficients are $T$-periodic in the time variable.
    % The analysis we developed for (\ref{infinit sde1}) can be extended to the non-equilibrium particle systems, where interparticle interactions vary over time. 
    Consequently, the interaction term in the corresponding generator becomes time-dependent. 
        Although the long-time behaviour of time-homogeneous systems is well-characterized by invariant measures, the corresponding theory for time-inhomogeneous systems like our periodically forced SDE (\ref{introdunctionperiodicinfinite}) requires a more general framework of periodic measures and their ergodic theory that was first rigorously developed in Feng and Zhao \cite{fengzhao2020randomperiodic}. Recent theoretical advances have significantly extended the understanding of such periodic measures. Notable progress includes breakthroughs in the existence and uniqueness of periodic measures of Markov processes (Feng et al. \cite{fengzhaozhong2023existence}), studies on quasi-periodic systems (Feng et al. \cite{fengquzhao2021randomquasiperiodic}), and extensions to infinite-dimensional settings, as exemplified by results on the 2D stochastic Navier–Stokes equations (Liu and Lu \cite{liu2025exponential}).

    For our problem, we establish the existence of periodic measures for equation (\ref{introdunctionperiodicinfinite}). Our approach proceeds by first considering finite-dimensional approximations and their associated periodic measures $\mu_{n,s}$. This is followed by proving the tightness of the lifted measures $\bar{\tilde{\mu}}_{n}:=\frac{1}{T}\int_0^T \tilde{\mu}_{n,s} ds$ on an extended state space, from which we extract a subsequence $\mu_{n,s_n}$ that converges to a measure $\nu_{s_*}$. The latter is invariant under the one-step transition probability kernel $\mathcal{P}(s_*,s_*+T)$ of (\ref{introdunctionperiodicinfinite}), which ultimately yields the desired periodic measure. This construction not only provides a new result of the existence, but also clarifies the fundamental pattern of stochastic periodic systems on infinite lattice with the limiting distribution which is invariant with respect to the discrete time Markovorvian semigroup $\mathcal{P}(s_*,s_*+nT), n\in {\mathbb N}$, and varies periodically in time. The periodic measure serves as the periodic "equilibria" in the nonequilibrium spin system.
    
    Although the finite-volume approximation is a classical idea used also in the study of the time-homogeneous infinite-lattice spin system, the tightness of the average lifted invariant measures of the finite-volume lattice system and its disintegration procedure introduced in this paper for periodic measures are highly nontrivial and novel. In the absence of direct access to tightness results for finite-dimensional periodic measures, we proceed by first lifting the finite-dimensional truncation of the time-inhomogeneous equation to an time-homogeneous system. For the lifted finite-dimensional time-homogeneous system, we establish the existence of invariant measures and consequently derive their tightness. Using this tightness property, we then construct invariant measures for the one-step transition probabilities of the infinite-dimensional time-inhomogeneous equation. This approach allows us to establish the existence of periodic measures using the framework developed by Feng et al.  \cite{fengzhaozhong2023existence}.  

    The organization of this paper is as follows. In the second section, we investigate the time-homogeneous finite-dimensional systems and verify, under certain conditions, that the transition semi-group corresponding to the finite-dimensional SDEs satisfies the criteria of the Harris-Meyn-Tweedie theorem, thus establishing the existence, uniqueness, and geometric ergodicity of invariant measures. In the third section, we approximate the solutions of time-homogeneous infinite-dimensional SDEs using the tightness of the solution's distributions from time-homogeneous finite-dimensional equations, proving the existence, uniqueness, and Markov property of these solutions, which leads to the establishment of invariant measures. The fourth section is devoted to studying periodic measures for finite-dimensional equations with periodic coefficients and demonstrates the existence of periodic measures.  In the last section, we give a proof of the main result in Section 2.

	%%%%%%%%%%%%%%%%%%%%%%%%%%%%%%%%%%%%%%%%%%%%%%%%%%%%%%%%%%%%%%%%%%%%%%%%%%%%%%%%%%%%%%%%%%%%%%%%%%%%%%%%%%
	\section{Invariant measures of the time-homogeneous finite-dimensional systems}
	In this section, we investigate invariant measures for finite-dimensional SDEs driven by degenerate noise. We will use the Meyn-Tweedie theory to study its long-time behaviour. In contrast to the dissipativity method developed by Da Prato and Zabczyk \cite{da1996ergodicity}, the Markov chain approach only requires the drift term to meet a weaker condition, specifically the weak dissipative condition.
	
	The corresponding finite system of equations can be regarded as the following $\mathbb{R}^n$-valued stochastic evolution equations,
	\begin{equation}
		\label{FSDE}
		\begin{cases}
			dX^{(n)}(t)=\left(A_nX^{(n)}(t)+F_n(X^{(n)}(t))\right)dt+B_ndW^{(n)}(t) ,  \quad t\geq 0,\\
			X^{(n)}(0)=x \in \mathbb{R}^n,
		\end{cases}
	\end{equation}
	where $W^{(n)}(t)$ is an $n$-dimensional Wiener process on a probability space $(\Omega,\mathcal{F},\mathbb{P})$ with identity covariance. The operator $A_n$, $F_n$ and $B_n$ is given by the formula
	\begin{equation}
		\label{AnFnBn}
			A_n=\begin{pmatrix}
				0 & a_{1,2} & & & & \\
				a_{2,1} & 0 & a_{2,3}& & & \\
				& \ddots &\ddots& \ddots& \\
				& & \ddots& \ddots& a_{n-1,n} \\
				& & & a_{n,n-1}&0&\\
			\end{pmatrix}, 
			F_n(x)=\begin{pmatrix}
				f(x_1)\\
				f(x_2)\\
				\vdots \\
				f(x_n)\\
			\end{pmatrix}, 
			B_n=\begin{pmatrix}
				1 & 0  &\cdots  & 0& \\
				0 & 0 &  &0 & \\
				\vdots&  &\ddots& \vdots& \\
				0&0 & \cdots &0&\\
			\end{pmatrix},
	\end{equation}
	where $f$ is a function that maps $\mathbb{R}$ to $\mathbb{R}$.
	\begin{remark}
		To simplify the calculation, we take $A_n$ as shown in (\ref{AnFnBn}). In the general setting $A_n=(b_{i,j})_{1\leq i,j \leq n}$, more complicated conditions need to be imposed on $A_n$ and $f$.
	\end{remark}
	We now impose additional conditions on the coefficients $A_n$ and $F_n$ in the equation, requiring $f$ to satisfy the weak dissipation condition given in Feng, Zhao and Zhong \cite{fengzhaozhong2023existence}. The precise assumptions are stated as follows:
	\begin{condition}
		\label{A1}
		\begin{description}
			\item [(i)] All the elements $\left\{a_{i,i-1}\in \mathbb{R},2\leq i\leq n\right\} $ of the matrix $A_n$ are nonzero, and all the elements of $A_n$ are bounded by a constant $M>0$, i.e.
			$$
			\max \left\{\vert a_{i-1,i}\vert ,\vert a_{i,i-1} \vert; 2\leq i\leq n \right\}\leq M. 
			$$
			\item [(ii)] The function $f\in C^\infty (\mathbb{R} , \mathbb{R})$ is weakly dissipative, i.e. there exist constants $\eta >0$ and $\lambda>\frac{1}{2}$ such that
            \begin{equation}\label{weaklydissipative}
                 f(z)z \leq \eta -\lambda \vert z \vert^2 ,z\in \mathbb{R}.
            \end{equation}
            \item[(iii)] The constants $M$ and $\lambda$ satisfy $2M< \lambda-\frac{1}{2}$.
		\end{description}
	\end{condition} 
	Under Assumption \ref{A1}, it is possible to prove the global in time existence and uniqueness of solutions to (\ref{FSDE}), see Khasminskii \cite[Theorem 3.5]{khasminskii2011stochastic}. Consequently, the solution $X^{(n),x}(t)$ of Equation (\ref{FSDE}) constitutes a Markov process, which enables us to define the corresponding transition semi-group $\{\mathcal{P}^{(n)}_t;t\geq 0\}$ on $B_b(\mathbb{R}^n)$, the space of bounded measurable functions on $(\mathbb{R}^n,\mathscr{B}(\mathbb{R}^n))$, and the transition probability kernel $\mathcal{P}^{(n)}_t(x,\Gamma)$, $t\geq 0$, $x\in\mathbb{R}^n$, $\Gamma\in \mathscr{B}(\mathbb{R}^n)$ for Equation (\ref{FSDE}) as follows: 
    \begin{equation*}
        \mathcal{P}^{(n)}_t\phi(x)=\mathbb{E}[\phi(X^{(n),x}(t))], \quad t\geq 0,x\in\mathbb{R}^n,\phi\in B_b(\mathbb{R}^n),
    \end{equation*}
    and 
    \begin{equation*}
        \mathcal{P}_t^{(n)}(x,\Gamma)=\mathcal{P}_t\mathbf{1}_\Gamma(x)=Law(X^{(n),x}(t))(\Gamma),\quad t\geq 0,x\in\mathbb{R}^n,\Gamma\in\mathscr{B}(\mathbb{R}^n),
    \end{equation*}
    where $\mathbf{1}_{\Gamma}$ represents the indicator function. The following theorem establishes that the transition semi-group $\{\mathcal{P}^{(n)}_t\}$ admits a unique invariant measure, which is geometrically ergodic. In order not to interrupt the main focus for the tightness argument on the convergence to infinite dimensional setting, we postpone the proof to Section 6, where the theorem is proved by adapting the methodology developed by Meyn and Tweedie.
	\begin{theorem}\label{finitdimensionalergodicresult}
		(finite-dimensional result) Assume  SDE (\ref{FSDE}) coefficients satisfy  Assumption \ref{A1}. Then there exists a unique invariant measure $\mu$ for $X^{(n),x}(t)$. Specifically, there exist constant $K,\alpha>0$ such that for the function $V_n(x)=1+x_1^2+\cdots +x_n^2$,
		\begin{equation}
			\label{EConverge}
			\sup\limits_{\{g:\vert g(x) \vert \leq V_n(x)\}}\vert \mathbb{E}g(X^{(n),a}(t))-\mu (g) \vert \leq KV_n(a)e^{-\alpha t},
		\end{equation}
		  for any $a\in \mathbb{R}^n$. Notice that (\ref{EConverge}) implies convergence of the transition probability kernel $\mathcal{P}^{(n)}_t(x,\Gamma)$, $ t\geq 0$, $\Gamma \in \mathscr{B}(\mathbb{R}^n)$ under the total variation norm, i.e.
		$$
		\Vert \mathcal{P}^{(n)}_t(x,\cdot)-\mu \Vert_{TV}\leq KV_n(x)e^{-\alpha t}.
		$$
	\end{theorem}

	%%%%%%%%%%%%%%%%%%%%%%%%%%%%%%%%%%%%%%%%%%%%%%%%%%%%%%%%%%%%%%%%%%%%%%%%%%%%%%%%%%%%%%%%%%%%%%%%    
	%%%%%%%%%%%%%%%%%%%%%%%%%%%%%%%%%%%%%%%%%%%%%%%%%%%%%%%%%%%%%%%%%%%%%%%%%%%%%%%%%%%%%%%%%%%%%%%%%%%%    
	%%%%%%%%%%%%%%%%%%%%%%%%%%%%%%%%%%%%%%%%%%%%%%%%%%%%%%%%%%%%%%%%%%%%%%%%%%%%%%%%%%%%%%%%%%%%%%%%%    

	\section{Time-homogeneous infinite-dimensional systems}
    This section is devoted to the study of invariant measures for an infinite-dimensional  SDEs driven by highly degenerate noise, which corresponds to a spin system on lattices $\mathbb{Z}^+$. Our analysis proceeds in two stages: first, we establish the existence and uniqueness of solutions to the SDEs along with their Markov property; second, we study the existence of invariant measures for the system. The SDE under consideration is formulated as follows:
	\begin{equation}
		\label{INFSDE1}
		\begin{cases}
			dX_1(t)=\left(a_{1,2}X_{2}(t)+f(X_1(t))\right)dt +dW_1(t)\\
			dX_i(t)=\left(a_{i,i-1}X_{i-1}(t)+a_{i,i+1}X_{i+1}(t)+f(X_i(t))\right)dt , \ i \in \left\{\mathbb{Z}^++1\right\},t\geq 0,\\
			X_\gamma(0)=x_\gamma \in \mathbb{R},\ \gamma \in \mathbb{Z}^+,
		\end{cases}
	\end{equation}
	where $W_1(t)$ is a one-dimensional Wiener process on a probability space $(\Omega,\mathcal{F},\mathbb{P})$ with identity covariance. We rewrite equation (\ref{INFSDE1}) in a concise form as
	\begin{equation}
		\label{INFSDE2}
		\begin{cases}
			dX(t)=\left(AX(t)+F(X(t))\right)dt+BdW(t) ,  \ t\geq 0,\\
			X(0)=x \in K,
		\end{cases}
	\end{equation}
	where $K=\ell_\rho^{2\theta}(\mathbb{Z}^+)\coloneqq \left\{x=\left(x_\gamma\right)_\gamma;\Vert x \Vert_{2\theta,\rho}^{2\theta}:=\sum_{\gamma\in \mathbb{Z}^+}\vert x_\gamma \vert^{2\theta} \rho(\gamma)< \infty \right\}$ is a weighted  space of sequences $\left(x_\gamma\right)_\gamma$ with a positive summable weight $\rho:\mathbb{Z}^+\rightarrow\mathbb{R}^+$, the constant $\theta$ will be taken to be the same as in Assumption \ref{A2}, $W(t)$ is a Wiener process on a probability space $(\Omega,\mathcal{F},\mathbb{P})$ taking values in $U=\ell^2 (\mathbb{Z}^+)\coloneqq \left\{x=\left(x_\gamma\right)_\gamma;\Vert x \Vert^2=\sum_{\gamma\in \mathbb{Z}^+}\vert x_\gamma \vert^2 < \infty \right\}$ and with identity covariance, and the operators $A$, $F$ and $B$ are given by the following formula, 
	\begin{equation}\label{AFBdef}
	     \begin{cases}
		A(x)=\left(a_{\gamma,\gamma-1}x_{\gamma-1}+a_{\gamma,\gamma+1}x_{\gamma+1}\right)_\gamma,x=\left(x_\gamma\right)_\gamma,\\
		F(x)=\left(f(x_\gamma)\right)_\gamma,x=\left(x_\gamma\right)_\gamma,\\
		B=(b_{\gamma,j})_{\gamma,j},
	\end{cases}
	\end{equation}
   where $ b_{1,1}=1,b_{\gamma,j}=0, \gamma,j\in \left\{\mathbb{Z}^++1\right\}$. In what follows, the symbol $x=\left(x_\gamma\right)_\gamma$ will represent a function defined on the positive integers $\mathbb{Z}^+$, while $x^{(n)}$ will denote an 
   $n$-dimensional vector. Analogously to Assumption \ref{A1}, we give the counterpart assumptions for the infinite-dimensional case.	
	\begin{condition}
		\label{A2}
		\begin{description}
			\item [(i)] All the elements $\left\{a_{i,i-1}\in \mathbb{R}, i\geq 2\right\} $ of the operator $A$ are nonzero, and all the elements of $A$ are bounded by a constant $M>0$, i.e.
			$$
			\max \left\{\vert a_{i-1,i}\vert ,\vert a_{i,i-1} \vert\right\}\leq M,\   i\geq 2.
			$$
			\item[(ii)] There exist constants $R,N>0$ such that
			$$
			\left\vert \frac{\rho(\gamma)}{\rho(j)} \right\vert\leq N \ \text{if} \ \vert\gamma-j\vert \leq R,
			$$
			$$
			\sum\limits_{i=1}^\infty \rho(i)< +\infty,\rho(i)>0. 
			$$
			\item [(iii)] The function $f\in C^\infty (\mathbb{R} , \mathbb{R})$ is weakly dissipative, i.e. there exists constants $\eta >0$ and $\lambda>\frac{1}{2}$ such that
			$$
			  f(z)z \leq \eta -\lambda \vert z \vert^2 ,z\in \mathbb{R}.
			$$
			Moreover there exists some constants $\theta\geq 1$, and $\eta_0\ge 0$ such that
			\begin{equation}
				\label{fdissipa}
				\left|f(z)\right|\leq \eta_0\left(1+\left\vert z\right\vert^\theta \right), z\in \mathbb{R}.
			\end{equation}
            \item[(iv)] The constants $M$ and $\lambda$ satisfy $2M< \lambda-\frac{1}{2}$.
		\end{description}
	\end{condition} 
	\begin{remark}\label{remarkdefl2}
        For notational simplicity, we set $a_{1,0}=0$ in the remainder of this paper. From Proposition 12.2.1 of Da Prato and Zabczyk \cite{da1996ergodicity}, we conclude that the operator $A$ is linear and bounded on $\ell_\rho^{p}(\mathbb{Z}^+)\coloneqq \left\{x=\left(x_\gamma\right)_\gamma;\right.$ $\left.\Vert x \Vert_{p,\rho}^{p}:=\sum_{\gamma\in \mathbb{Z}^+}\vert x_\gamma \vert^{p} \rho(\gamma)< \infty \right\}$ under Assumptions \ref{A2} (i) and (ii), for $p\in [1,\infty)$.
	\end{remark}
    
    \subsection{The concepts of solutions: weak martingale solutions and solutions to a martingale problem}
    The objective of this section is to introduce two distinct definitions of solutions corresponding to equation (\ref{INFSDE2}): weak martingale solutions and solutions to a martingale problem. Under our setting, $f$ is locally Lipschitz, but this condition does not guarantee that $F$ is locally Lipschitz, nor does the weak dissipativity condition (\ref{weaklydissipative}) on $f$ yield pathwise uniqueness of solutions. Consequently, the study of strong solutions faces obstacles, and we therefore turn to investigating the existence and uniqueness of weak solutions. 
    
    The definition of the martingale problem depends on the choice of test functions. Following the approach of Kallianpur and Xiong \cite[pp. 166-167]{kallianpur1995stochastic}, we adopt the cylinder functions specified in Definition \ref{clydinderfunctiononell2} as our test functions. This selection serves two key purposes: (i). It ensures that the application of the Kolmogorov operator $\mathscr{L}$ yields a finite summation.
    (ii). The resulting class of test functions is sufficiently restrictive to facilitate the proof of uniqueness in Section 3.4. Since the mathematical framework of our problem differs fundamentally from that in  \cite[Theorem 5.2.7]{kallianpur1995stochastic} - with neither set of conditions implying the other - the proof of equivalence between weak solutions and solutions to the martingale problem requires substantially different technical approaches.   Accordingly, this section provides a rigorous proof of this equivalence by adapting classical techniques from finite-dimensional SDE theory (see Karatzas and Shreve \cite[Proposition 4.6]{karatzas1998brownian}) to our framework. 
    
    We consider $\Omega_{Spin}=C([0,\infty),\ell_\rho^{2}(\mathbb{Z}^+))$ and denote by $(X(t))_{t\geq0}$ the canonical process on $\Omega_{Spin}$, that is, $X(t,\omega):=\omega(t)$, for $\omega\in\Omega_{Spin}$, by $\mathscr{B}^{Spin}$ the Borel $\sigma$-algebra in $\Omega_{Spin}$ and by $\mathscr{B}_t^{Spin}=\sigma(X|_{[0,t]};X\in \Omega_{Spin})$. Before presenting the solution concept, we verify the Borel measurability of several sets that will be used in our study. Since the proof technique parallels that of Flandoli and Romito \cite[Lemma 2.1]{FlandoliRomito2008Markovselections}, we shall not reproduce the arguments here.
    \begin{lemma}\label{borelmeasurablesetinomega}
        The set $L^{\infty}_{loc}([0,\infty);\ell^{2\theta}_{\rho}(\mathbb{Z}^{+})])\bigcap\Omega_{Spin}$ is a Borel set in $\Omega_{Spin}$. Moreover, 
        \begin{equation*}
            L^{\infty}_{loc}([0,\infty);\ell^{2\theta}_{\rho}(\mathbb{Z}^+))\bigcap \Omega_{Spin}=C([0,\infty);\ell^{2\theta}_{\rho,\sigma}(\mathbb{Z}^+))\bigcap \Omega_{Spin},
        \end{equation*}
        where $\ell^{2\theta}_{\rho,\sigma}(\mathbb{Z}^+)$ denotes the space $\ell^{2\theta}_{\rho}(\mathbb{Z}^+)$ endowed with the weak topology.
    \end{lemma}
    
    \begin{definition}\label{weakmartingalesolutiondef}(weak martingale solutions). Given a probability measure $\mu$ on $\ell_\rho^{2\theta}(\mathbb{Z}^+)$, a weak martingale solution of equation (\ref{INFSDE2}) with initial condition $\mu$ is a triple $(\tilde{X},\tilde{W})$, $(\tilde{\Omega},\tilde{\mathscr{F}},\tilde{\mathbb{P}})$, $\{\tilde{\mathscr{F}}_t\}$, where
    \begin{itemize}
        \item[(i)] $(\tilde{\Omega},\tilde{\mathscr{F}},\tilde{\mathbb{P}})$ is a probability space, and $\{\tilde{\mathscr{F}}_t\}$ is a filtration of sub-$\sigma$-fields of $\tilde{\mathscr{F}}$ satisfying the usual conditions,
        \item[(ii)] $\tilde{W}=\{\tilde{W}(t),\tilde{\mathscr{F}_t};0\leq t\leq \infty\}$ is cylindrical Wiener process on $\ell^{2}(\mathbb{Z}^+)$, $\tilde{X}=\{\tilde{X}(t),\tilde{\mathscr{F}_t};0\leq t\leq \infty\}$ is an adapted process in $\ell_\rho^{2}(\mathbb{Z}^+)$ and 
        \begin{equation*}
             \tilde{X}(\cdot,\tilde{\omega})\in L^{\infty}(0,T;\ell_\rho^{2\theta}(\mathbb{Z}^+))\bigcap C([0,T];\ell_\rho^{2}(\mathbb{Z}^+)) \quad \tilde{\mathbb{P}}-a.s.
        \end{equation*}
        for every $T>0$,
        \item[(iii)] the integral version of (\ref{INFSDE2})
        \begin{equation*}
            \tilde{X}(t)=\tilde{X}(0)+\int_0^tA\tilde{X}(s)+F(\tilde{X}(s))ds+\int_0^tBd\tilde{W}(s), t\in [0,\infty),
        \end{equation*}
        holds $\tilde{\mathbb{P}}$ almost surely in $\ell_\rho^{2}(\mathbb{Z}^+)$,
        \item[(iv)] $\tilde{X}(0)$ has law $\mu$.
    \end{itemize}
    \end{definition}
    Building upon the framework of {\v{Z}}{\'a}k \cite{vzak2016existence}, we systematically define three fundamental function classes that will facilitate both the precise formulation of the solutios of martingale problems and the subsequent notational efficiency.
	\begin{definition}
    \label{clydinderfunctiononell2}
		A function $h:\ell_\rho^{2}(\mathbb{Z}^+) \rightarrow\mathbb{R}$ is called a cylinder function if there exists a finite subset $\Lambda\subset\mathbb{Z}^+$ and $g:\mathbb{R}^{\Lambda}\rightarrow\mathbb{R}$ such that $h(x)=g(\Pi_{\Lambda}x)$, for $x=(x_\gamma)_{\gamma\in\mathbb{\mathbb{Z}^+}}\in\ell_\rho^{2}(\mathbb{Z}^+)$, where $\Pi_{\Lambda}:\ell_\rho^{2}(\mathbb{Z}^+)\rightarrow\mathbb{R}^{\Lambda}$ is the restriction mapping, i.e. $\Pi_{\Lambda}x=(x_\gamma)_{\gamma\in\Lambda}\in\mathbb{R}^{\Lambda}$. Moreover,
		\begin{itemize}
			%\item [(1)]If $g$ is Lipschitz  and compactly supported i.e. there are constants $N\in\mathbb{N}$ and $L>0$ such that $supp(g)\subset[-N,N]^{\Lambda}$ and 
			%$$
			%|g(z) - g(\tilde{z})| \leq L\|z - \tilde{z}\|_{\mathbb{R}^{\Lambda}} \ \forall %z, \tilde{z} \in \mathbb{R}^{\Lambda},
			%$$
			%we say that $h\in C_{c,Lip}^{Cyl}(\ell_\rho^{2}(\mathbb{Z}^+) )$.
			\item[(1)]If $g\in C^2(\mathbb{R}^{\Lambda},\mathbb{R})$, we say that $h\in C^{2,Cyl}(\ell_\rho^{2}(\mathbb{Z}^+))$;
			\item[(2)] If $g\in C^2_c(\mathbb{R}^{\Lambda},\mathbb{R})$, we say that $h\in C_{c}^{2,Cyl}(\ell_\rho^{2}(\mathbb{Z}^+) )$.
		\end{itemize}
	\end{definition}
	The Kolmogorov operator $\mathscr{L}$ corresponding to (\ref{INFSDE2}) is defined as follows: for any twice Fr\'{e}chet differentiable function $h$ on $\ell^2_{\rho}(\mathbb{Z}^+)$,
	$$
	\mathscr{L}h(x):=(a_{1,2} x_2+f(x_1))\partial_{1}h(x)+\sum_{i=2}^{\infty}(a_{i,i-1}x_{i-1}+a_{i,i+1} x_{i+1}+f(x_i))\partial_{i}h(x)+ \frac{1}{2} {\partial_{1}}\partial_{1}h(x),
	$$
	where $\partial_i$ are given by that for Fr\'{e}chet differentiable functions $h$ with Fr\'{e}chet derivative $h^{\prime}$, $\partial_ih:=<h^{\prime},e^i>_{2,\rho}$ where $e^i=(e^i_j)_{j\in\mathbb{Z}^+}$ with $e^i_i=1$ and $e^i_j=0$ for $j\neq i$.
	
    \begin{definition}\label{solutiontothemartingaleproblemdef}(solution to the martingale problem). Given a probability measure $\mu$ on $\ell_\rho^{2\theta}(\mathbb{Z}^+)$, a probability measure $\mathbf{P}$ on $(\Omega_{Spin},\mathscr{B}^{Spin})$ is called a solution to the martingale problem associated to (\ref{INFSDE2}) with initial law $\mu$ if 
    \begin{itemize}
        \item[(i)] for every $T>0$
        \begin{equation*}
            \mathbf{P}\left(\sup\limits_{t\in[0,T]}\Vert X(t)\Vert_{2\theta,\rho}<\infty\right)=1;
        \end{equation*}
        \item[(ii)] for every $h\in C_c^{2,Cly}(\ell_\rho^{2}(\mathbb{Z}^+))$(or $h\in C^{2,Cyl}(\ell_\rho^{2}(\mathbb{Z}^+) )$), the process $M^h(t)$ defined $\mathbf{P}-a.s.$ on $(\Omega_{Spin},\mathscr{B}^{Spin})$ as 
        \begin{equation*}
            M^h(t):=h(X(t))-h(X(0))-\int_0^t\mathscr{L}h(X(s))ds,
        \end{equation*}
        is a continuous (local) martingale under $(\Omega_{Spin},\mathscr{B}^{Spin}, \mathbf{P},\{\mathscr{B}^{Spin}_t\})$;
        \item[(iii)] $\mu=\mathbf{P}\circ X(0)^{-1}$.
    \end{itemize}
	\end{definition}
        The boundedness of the operator $B$ implies the equivalence between the two alternative conditions imposed on $h$ in Definition \ref{solutiontothemartingaleproblemdef} (ii). We now present the theorem connecting the notions of weak martingale solutions and solutions to the martingale problem. 
	\begin{theorem}
    \label{eqalweakandmartingale}
	    The existence of a solution $\mathbf{P}$ to the martingale problem is equivalent to the existence of a weak martingale solution $(\tilde{X},\tilde{W})$, $(\tilde{\Omega},\tilde{\mathscr{F}},\tilde{\mathbb{P}})$, $\{\tilde{\mathscr{F}}_t\}$. The two solutions are related by $\mathbf{P}=\tilde{\mathbb{P}}\circ \tilde{X}^{-1}$.
	\end{theorem}
    \begin{proof}
        Step 1. Let $\mu$ be a probability measure on $\ell_\rho^{2\theta}(\mathbb{Z}^+)$ and $(\tilde{X},\tilde{W})$, $(\tilde{\Omega},\tilde{\mathscr{F}},\tilde{\mathbb{P}})$, $\{\tilde{\mathscr{F}}_t\}$ be a weak martingale solution of equation (\ref{INFSDE2}) with initial condition $\mu$. Let $\mathbf{P}:=\tilde{\mathbb{P}}\circ \tilde{X}^{-1}$ be the law of the canonical  process $X$ on $\Omega_{Spin}$. Our objective is to prove that $\mathbf{P}$ constitutes a solution to the martingale problem associated to (\ref{INFSDE2}) with initial law $\mu$. From conditions (ii) and (iii) of Definition \ref{weakmartingalesolutiondef}, we deduce
        \begin{equation*}
            \mathbf{P}\left(\sup\limits_{t\in[0,T]}\Vert X(t)\Vert_{2\theta,\rho}<\infty\right)=\tilde{\mathbb{P}}\left(\sup\limits_{t\in[0,T]}\Vert \tilde{X}(t)\Vert_{2\theta,\rho}<\infty\right)=1
        \end{equation*}
        and 
        \begin{equation*}
            \mathbf{P}\circ(X(0))^{-1}=\tilde{\mathbb{P}}\circ(\tilde{X}(0))^{-1}=\mu.
        \end{equation*}
        Moreover, for a function $h\in C^{2,Cyl}(\ell_\rho^{2}(\mathbb{Z}^+))$, there exist a finite subset $\Lambda\in\mathbb{Z}^+$ and a function $g\in C^2(\mathbb{R}^{\Lambda},\mathbb{R})$ such that $h(x)=g(\Pi_{\Lambda}x)$. By It$\hat{\text{o}}$'s formula, the process $\tilde{M}^h=\{\tilde{M}_t^h,\tilde{\mathscr{F}}_t;0\leq t<\infty\}$ given by
        \begin{equation*}
            \tilde{M}_t^h:=h(\tilde{X}(t))-h(\tilde{X}(0))-\int_0^t\mathscr{L}h(\tilde{X}(s))ds
        \end{equation*}
        can be expressed as 
        \begin{equation*}
            \tilde{M}_t^h=\int_0^t<h^{\prime}(\tilde{X}(s)),BdW(s)>_{2,\rho}=\int_0^t\partial_1h(\tilde{X}(s))dW_1(s),
        \end{equation*}
        where $<\cdot,\cdot>_{2,\rho}$ represents the inner product in $\ell_\rho^{2}(\mathbb{Z}^+)$.
        Introducing the stopping time 
        \begin{equation*}
            S_m:=\inf\{t\geq 0;\Vert \tilde{X}(t)\Vert_{\mathbb{R}^{\Lambda}}\geq m \},
        \end{equation*}
        condition (ii) of Definition \ref{weakmartingalesolutiondef} yields $\lim\limits_{m\rightarrow\infty}S_m=\infty$ $\tilde{\mathbb{P}}-a.s.$ Therefore, the processes
        \begin{equation*}
            \tilde{M}_t^h(m):=\tilde{M}_{t\wedge S_m}^h=\int_0^{t\wedge S_m}\partial _1h(\tilde{X}(s))dW_1(s)=\int_0^{t\wedge S_m}\frac{\partial g}{\partial x_1}(\Pi_{\Lambda}\tilde{X}(s))dW_1(s)
        \end{equation*}
        are continuous martingales, and so $\tilde{M}^h$ is continuous local martingale. Based on Lemma 4.1 of Flandoli \cite{flandolianintroduction2008}, we conclude that
        \begin{equation*}
            M_t^h:=h(X(t))-h(X(0))-\int_0^t\mathscr{L}h(X(s))ds
        \end{equation*}
        is a local martingale under $(\Omega_{Spin},\mathscr{B}^{Spin},\mathbf{P},\{\mathscr{B}^{Spin}_t\})$. For $h\in C_c^{2,Cly}(\ell_\rho^{2}(\mathbb{Z}^+))$, analogous arguments apply, and thus we omit the details. \\
        Step 2. Let now $\mathbf{P}$ be a solution to the martingale problem associated to (\ref{INFSDE2}) with initial law $\mu$.   The particular selection $h(x) = x_i$ combined with the assumption yields that
        \begin{equation*}
            M_t^{(i)}:=X_i(t)-X_i(0)-\int_0^t\left(a_{i,i-1}X_{i-1}(s)+a_{i,i+1}X_{i+1}(s)+f(X_i(s))\right)ds
        \end{equation*}
        are continuous local martingales under $\mathbf{P}$. With $h(x)=x_ix_j$, we see that 
        \begin{equation*}
        \begin{split}
           M_t^{(i,j)}:&=X_i(t)X_j(t)-X_i(0)X_j(0)-\int_0^t\left(a_{i,i-1}X_{i-1}(s)+a_{i,i+1}X_{i+1}(s)+f(X_i(s))\right)X_j(s)ds \\
           &\quad -\int_0^t\left(a_{j,j-1}X_{j-1}(s)+a_{j,j+1}X_{j+1}(s)+f(X_j(s))\right)X_i(s)ds, \quad \text{for}\quad (i,j)\neq (1,1),
        \end{split}
        \end{equation*}
        and 
        \begin{equation*}
        \begin{split}
            M_t^{(i,j)}:=X_1(t)X_1(t)-X_1(0)X_1(0)-2\int_0^t(a_{1,2}X_{2}(s)+f&(X_1(s)))X_1(s)ds-t, \\
            &\quad \text{for} \quad (i,j)=(1,1),
        \end{split}
        \end{equation*}
        are also continuous local martingales. For $i\neq j$, we have 
        \begin{equation*}
            M_t^{(i)}M_t^{(j)}=M_t^{(i,j)}-X_i(0)M_t^{(j)}-X_j(0)M_t^{(i)}+U_t=:V_t^{(i,j)}+U_t^{(i,j)},
        \end{equation*}
        where 
        \begin{equation*}
        \begin{split}
           U_t^{(i,j)}&=\int_0^t(X_i(s)-X_i(t))\left(a_{j,j-1}X_{j-1}(s)+a_{j,j+1}X_{j+1}(s)+f(X_j(s))\right)ds\\
           &\quad +\int_0^t(X_j(s)-X_j(t))\left(a_{i,i-1}X_{i-1}(s)+a_{i,i+1}X_{i+1}(s)+f(X_i(s))\right)ds\\
           &\quad +\int_0^t\left(a_{j,j-1}X_{j-1}(s)+a_{j,j+1}X_{j+1}(s)+f(X_j(s))\right)ds\int_0^t(a_{i,i-1}X_{i-1}(s)\\
           &\quad +a_{i,i+1}X_{i+1}(s)+f(X_i(s))ds\\
           &=\int_0^t(M_s^{(i)}-M_t^{(i)})\left(a_{j,j-1}X_{j-1}(s)+a_{j,j+1}X_{j+1}(s)+f(X_j(s))\right)ds\\
           &\quad +\int_0^t(M_s^{(j)}-M_t^{(j)})\left(a_{i,i-1}X_{i-1}(s)+a_{i,i+1}X_{i+1}(s)+f(X_i(s))\right)ds\\
           &=-\int_0^t\int_0^s\left(a_{j,j-1}X_{j-1}(u)+a_{j,j+1}X_{j+1}(u)+f(X_j(u))\right)dudM_s^{(i)}\\
           &\quad -\int_0^t\int_0^s\left(a_{i,i-1}X_{i-1}(u)+a_{i,i+1}X_{i+1}(u)+f(X_i(u))\right)dudM_s^{(j)}.
        \end{split}
        \end{equation*}
        The continuity and local martingale properties of both $V_t^{(i,j)}$ and $U_t^{(i,j)}$ with respect to measure $\mathbf{P}$ immediately imply that $M_t^{(i)}M_t^{(j)}$ shares these same properties under $\mathbf{P}$. An analogous argument applied to the case $i = j = 1$ shows that $M_t^{(1)}M_t^{(1)}-t$ is likewise a $\mathbf{P}$-continuous local martingale. These results immediately imply the following quadratic variation structure:
        \begin{equation*}
        \begin{split}
            &[M^{(i)},M^{(j)}]_t=0,\quad \text{for} \quad (i,j)\neq (1,1), \\
            &[M^{(1)},M^{(1)}]_t=t,\quad \text{for} \quad (i,j)=(1,1).
        \end{split}
        \end{equation*}
         By L\'{e}vy martingale characterization of Brownian motion it follows that $M_t^{(1)}$ is a 1-dimensional Brownian motion on $(\Omega_{Spin},\mathscr{B}^{Spin},\mathbf{P},\{\mathscr{B}^{Spin}_t\})$. Thus, there is a cylindrical Wiener process $\tilde{W}=\{\tilde{W}(t),\tilde{\mathscr{F}_t};0\leq t\leq \infty\}$ defined on an extension $(\tilde{\Omega},\tilde{\mathscr{F}},\tilde{\mathbb{P}})$ of $(\Omega_{Spin},\mathscr{B}^{Spin},\mathbf{P})$ with $\tilde{W}_1(t)=M_t^{(1)}$ such that $(\tilde{X}(t)=X(t),\tilde{W}_t)$, $(\tilde{\Omega},\tilde{\mathscr{F}},\tilde{\mathbb{P}})$, $\{\tilde{\mathscr{F}}_t\}$ is a weak martingale solution of equation (\ref{INFSDE2}) with initial condition $\mu$.
    \end{proof}
   \begin{corollary}
   \label{equvalentuniquepropertyweakmartingale}
       Given a probability measure $\mu$ on $\ell_\rho^{2\theta}(\mathbb{Z}^+)$, the following statements are equivalent:
       \begin{itemize}
           \item[(i)] Uniqueness holds for solutions of the martingale problem associated to (\ref{INFSDE2}) with initial law $\mu$;
           \item[(ii)] Equation (\ref{INFSDE2}) with initial condition $\mu$ admits a weak martingale solution which is unique in the sense of probability law.
       \end{itemize}
   \end{corollary}

    \subsection{Preliminary results}
	This subsection is devoted to analyzing the finite-dimensional truncations of equation (\ref{INFSDE2}). By establishing the tightness of the corresponding solution distributions, we lay the groundwork for proving the existence of solutions to the full system.
	Let $\{\Lambda_n:=\{1,\cdots,n\};n\in \mathbb{N}\}$ be a sequence of finite subsets of $\mathbb{Z}^+$, which obviously satisfies $\Lambda_n\uparrow\mathbb{Z}^+$, as $n\rightarrow+\infty$. For any fixed initial data $x\in \ell_\rho^{2\theta}(\mathbb{Z}^+)$ and any $\Lambda_n$ let us consider the following system of finite-dimensional SDEs, derived from (\ref{INFSDE2}) by a cutoff procedure,
	\begin{equation}
		\label{CUTSDE}
		\begin{cases}
			dX_1^{(n)}(t)=\left(a_{1,2}X_{2}^{(n)}(t)+f(X_1^{(n)}(t))\right)dt +dW_1(t),\\
			dX_i^{(n)}(t)=\left(a_{i,i-1}X_{i-1}^{(n)}(t)+a_{i,i+1}X_{i+1}^{(n)}(t)+f(X_i^{(n)}(t))\right)dt  , \ 2\leq i \leq n-1,\\
			dX_n^{(n)}(t)=\left(a_{n,n-1}X_{n}^{(n)}(t)+f(X_n^{(n)}(t))\right)dt, \\
			X_\gamma^{(n)}(0)=x_\gamma \in \mathbb{R},\ 1\leq\gamma\leq n, t\geq 0.   
		\end{cases}
	\end{equation}
	Building upon the analysis in Section 2, we observe that under Assumption \ref{A2}, the finite-dimensional system (\ref{CUTSDE}) automatically verifies all conditions of Assumption \ref{A1}. Hence, it admits a unique solution denoted by $X^{(n)}(t)\in \mathbb{R}^n$, as well as a unique invariant measure $\mu_n$.  Set $\tilde{X}^{(n)}\coloneqq\left(X^{(n)},0_{i\in\left\{\mathbb{Z}^+\backslash \Lambda_n\right\}}\right)$, then it is obvious that each $\tilde{X}^{(n)}(t)$ takes values in $\ell_\rho^{2}(\mathbb{Z}^+)$ and therefore $\tilde{X}^{(n)}$ lives in $\Omega_{Spin}=C([0,\infty),\ell_\rho^{2}(\mathbb{Z}^+))$. In order to obtain the tightness of the above approximations, we introduce the following version of the Arzel$\grave{a}$-Ascoli theorem (see Munkres \cite{munkres2019topology}). 
	\begin{theorem}
		\label{AAthm}
		(Arzel$\grave{a}$-Ascoli). Let $Y$ be a complete metric space and $f_n\in C([0,\infty),Y)$ be a sequence of equicontinuous functions, where $C([0,\infty),Y)$ is endowed with the topology of uniform convergence on compacts. If $\left\{f_n(t)\right\}$ is precompact in $Y$ on a dense set of $t\in [0,\infty)$, then $\left\{f_n(t)\right\}$ is precompact in $C([0,\infty),Y)$.
	\end{theorem}
	By virtue of Theorem \ref{AAthm}, our objective now turns to studying the equicontinuity of $\tilde{X}^{(n)}$ and the precompactness of $\tilde{X}^{(n)}(t)$ in $\ell_\rho^{2}(\mathbb{Z}^+)$, for $t\in[0,\infty)\bigcap\mathbb{Q}$. In order to demonstrate equicontinuity, we introduce a variant of the Kolmogorov continuity theorem in the following form (see Bass \cite{bass2011stochastic}).
	
	\begin{theorem}
		\label{Kolmogorovthm}
		Let $\{X^{(n)}\}_n$ be a set of continuous processes taking values in some metric space $(S,d)$. Suppose that, for any $T>0$, there exist constants $C(T),\epsilon>0$ and $p>0$ such that 
		$$
		\sup\limits_n \mathbb{E}\left[d (X_s^{(n)},X_t^{(n)})^p\right]\leq C(T)\vert t-s\vert^{1+\epsilon},0\leq s\leq t \leq T.	
		$$
		Then $\left\{X^{(n)}\right\}$ is an equicontinuous family of processes with probability 1.
	\end{theorem}
	Moreover, to better characterize precompact sets in $\ell_\rho^{p}(\mathbb{Z}^+)$ with $p\geq 1$, we provide a necessary and sufficient condition for a set to be precompact in $\ell_\rho^{p}(\mathbb{Z}^+)$.
	\begin{lemma}
		\label{Precompactcondition}
		Let $M$ be a subset of the metric space $ \ell_\rho^{p}(\mathbb{Z}^+)$ for $p\geq 1$. Then $M$ is precompact in $\ell_\rho^{p}(\mathbb{Z}^+)$ if and only if $M$ is bounded and the following condition holds: 
		$$
		\text{for any }  \epsilon> 0, \ \text{there exists } n_0\in \mathbb{N} \  \ \text{such that for any }   x\in M:\sum\limits_{i=n_0}^{\infty}\vert x_i\vert^{p}\rho(i)  <\epsilon.
		$$
	\end{lemma}
	\begin{proof}
		The sufficiency of the result has been established by {\v{Z}}{\'a}k  \cite{vzak2016existence}, we only need to demonstrate the necessity. If $M$ is a precompact set, then $M$ is totally bounded. The fact that $M$ is bounded is a consequence of the fact that $M$ is totally bounded. For every $\epsilon>0$, there is a finite cover of $M$ by $\sqrt[p]{\frac{\epsilon}{2}}$-balls $\{B(x^j,\sqrt[p]{\frac{\epsilon}{2}})\}_{j=1,\cdots,m}$, where $x^j\in M,j=1,\cdots,m$. Moreover, for the given $\epsilon$, there is a constant $n_0$ such that
        \begin{equation*}
            \sup\limits_{j=1,\cdots,m}\sum\limits_{i=n_0}^{\infty}\vert x^j_i\vert^{p}\rho(i) <\frac{\epsilon}{2}.
        \end{equation*}
		Therefore, for every $x\in M$, there is a $j\in\{1,\cdots,m\}$ such that $x\in B(x^j,\sqrt[p]{\frac{\epsilon}{2}})$. Hence, we conclude that 
        \begin{equation*}
            \sum\limits_{i=n_0}^{\infty}\vert x_i\vert^{p}\rho(i)\leq \|x-x^j\|^{p}_{p,\rho}+	\sum\limits_{i=n_0}^{\infty}\vert x^j_i\vert^{p}\rho(i)<\frac{\epsilon}{2}+\frac{\epsilon}{2}=\epsilon.
        \end{equation*}
	\end{proof}
    We now proceed to verify that $\tilde{X}^{(n)}$ satisfies all the conditions required by Theorem \ref{Kolmogorovthm} and Lemma \ref{Precompactcondition}, which enables us to establish the tightness of the distributions of $\tilde{X}^{(n)}$ through an application of Theorem \ref{AAthm}. First, we establish the moment estimates of the approximations $\{\tilde{X}^{(n)}\}_n$ based on Assumption \ref{A2}.
	\begin{lemma}
		\label{ineqXn}
		Let $x\in\ell_\rho^{2\theta}(\mathbb{Z}^+)$, and for $n\in \mathbb{N}$, $X^{(n)}$ the solution of (\ref{CUTSDE}) with initial condition $X^{(n)}(0)=\Pi_{\Lambda_n}(x)$. Let Assumption \ref{A2} hold and $T>0,p\geq1$ be given, then there exists constant $C(T)>0$ and $C(t,p)>0$ such that
		\begin{align}
            &\label{ineq3}
            \forall n\in\mathbb{N}:\mathbb{E}\sup\limits_{0\leq s\leq T}\Vert \tilde{X}^{(n)}(s)\Vert_{2\theta,\rho}^p\leq C(T,p)(1+\Vert x \Vert_{2\theta,\rho}^p),\\
            &\label{ineq4}
            \forall n\in\mathbb{N}:\mathbb{E}\sup\limits_{0\leq s\leq T}\Vert \tilde{X}^{(n)}(s)\Vert_{2,\rho}^p\leq C(T,p)(1+\Vert x \Vert_{2,\rho}^p),\\
			&\label{ineq1}
			\forall 0\leq s \leq t\leq T:\sup\limits_n \mathbb{E}\Vert \tilde{X}^{(n)}(t)-\tilde{X}^{(n)}(s) \Vert_{2,\rho}^4 \leq C(T)(1+ \Vert x \Vert_{2\theta,\rho}^{4\theta})\vert t-s \vert^2,\\
			&\label{ineq2}
			\forall \delta > 0 \ \forall t\geq 0 \ \exists N_0(t,\delta):\sup\limits_n \mathbb{E}\sum\limits_{i=N_0+1}^\infty\vert \tilde{X}^{(n)}_{i}(t) \vert^2\rho(i)< \delta.
		\end{align}
	\end{lemma}
    
	\begin{proof}
     We proceed in three steps:
    
    Step 1: Proof of (\ref{ineq3}). We begin by defining a twice continuously differentiable function on the space $\ell_\rho^{2\theta}(\mathbb{Z}^+)$,
    \begin{equation*}
    V(x):=\sum\limits_{i=1}^{\infty}\vert x_i\vert^{2\theta}\rho(i)+1, x=(x_i)_{i\in\mathbb{Z}^+}\in\ell_\rho^{2\theta}(\mathbb{Z}^+).
    \end{equation*}
    By the hypotheses of Assumption \ref{A2} and Young's inequality, we establish that for any fixed $n\in \mathbb{N}$ and $p\geq 1$
    \begin{align}
        \mathscr{L}_nV^p(x)&=pV^{p-1}(x)2\theta(a_{1,2}x_2+f(x_1))x_1^{2\theta-1}\rho(1)\nonumber\\
        &\quad +pV^{p-1}(x)2\theta(a_{2,1}x_1+a_{2,3}x_3+f(x_2))x_2^{2\theta-1}\rho(2)+\cdots \nonumber\\
        & \quad +pV^{p-1}(x)2\theta(a_{n,n-1}x_{n-1}+f(x_n))x_n^{2\theta-1}\rho(n)+\frac{1}{2}pV^{p-1}(x)2\theta(2\theta-1)x_1^{2\theta-2}\rho(1)\nonumber\\
        &\quad +\frac{1}{2}p(p-1)V^{p-2}(x)\left(2\theta x_1^{2\theta-1}\rho(1)\right)^2\nonumber\\
        & \leq pV^{p-1}(x)2\theta \left\{\left[M(\vert x_2\vert^{2\theta}+\vert x_1\vert^{2\theta})+(\eta-\lambda\vert x_1\vert^{2})x_1^{2\theta-2}\right]\rho(1) \right. \nonumber\\
        &\quad + \left[M(\vert x_1\vert^{2\theta}+\vert x_2\vert^{2\theta}+\vert x_3\vert^{2\theta}+\vert x_2\vert^{2\theta})+(\eta-\lambda\vert x_2\vert^{2})x_2^{2\theta-2}\right]\rho(2)+\cdots \nonumber\\
        &\quad +\left[M(\vert x_{n-1}\vert^{2\theta}+\vert x_n\vert^{2\theta})+(\eta-\lambda\vert x_n\vert^{2})x_n^{2\theta-2}\right]\rho(n)\nonumber\\
        &\quad +\frac{1}{2}pV^{p-1}(x)2\theta(2\theta-1)(x_1^{2\theta}+1)\rho(1)+\frac{1}{2}p(p-1)V^{p-2}(x)4\theta^2\left( x_1^{4\theta}+1\right)\rho^2(1)\nonumber\\
        &\leq CV^p(x),\label{ineqCVx}
    \end{align}
    where $C=8\theta MNp+2\theta p \eta  (1+\sum_{i=1}^{\infty}\rho(i))+\theta(2\theta-1)p(\rho(1)+1)+2\theta^2p(p-1)(\rho^2(1)+1)$ and $N$ is the constant in Assumption \ref{A2} (ii). From (\ref{ineqCVx}) and $V(x)>0$ it follows that the function
    \begin{equation*}
    W(t,x):=V^p(x)\exp \{-Ct\}
    \end{equation*}
    satisfies $\mathscr{L}_nW(t,x)+\frac{\partial W}{\partial t}(t,x)\leq 0$.
    We define the first exit time of $\tilde{X}^{(n)}$
    \begin{equation*}
    \tau^n_R:=\inf \left\{t\geq 0: \Vert \tilde{X}^{(n)}(t) \Vert^2_{2\theta,\rho}>R \right\}.
    \end{equation*}
    By Lemma 3.2 of Khasminskii \cite{khasminskii2011stochastic}, we can derive 
    \begin{equation*}
    \begin{split}
    \mathbb{E}W(s\wedge\tau^n_R,\tilde{X}^{(n)}(s\wedge\tau^n_R))-W(0,\Pi_n x)&=\mathbb{E}\int_0^{s\wedge\tau^n_R}\mathscr{L}W(r,\tilde{X}^{(n)}(r))\\
    &\quad +\frac{\partial W}{\partial r}(r,\tilde{X}^{(n)}(r))dr \leq 0,
    \end{split}
    \end{equation*}
    Therefore,
    \begin{equation*}
    \mathbb{E}V^p(\tilde{X}^{(n)}(s\wedge\tau^n_R))\exp \{-Cs\}\leq V^p(\Pi_nx)\leq V^p(x),
    \end{equation*}
    where $\Pi_n:x=(x_i)_{i\in\mathbb{Z}^+}\mapsto (x_i)_{i=1}^n$ is a restriction mapping from $\ell_\rho^{p}(\mathbb{Z}^+)$ to $\mathbb{R}^n$. Since $s\wedge\tau^n_R\leq T$ and $V\geq 0$, we have 
    \begin{equation*}
    \mathbb{E}V^p(\tilde{X}^{(n)}(s\wedge\tau^n_R))\leq V^p(x)e^{CT}.
    \end{equation*}
    Moreover, since $\mathbb{P}(\lim\limits_{R\rightarrow+\infty}\tau^n_R=+\infty)=1$ holds, Fatou’s lemma yields
    \begin{equation}
    \label{iVpconclusion}
    \mathbb{E}V^p(\tilde{X}^n(s))\leq V^p(x)e^{CT}.
    \end{equation}
    By applying It\^{o}'s formula together with the BDG inequality and (\ref{iVpconclusion}), we obtain the existence of a constant $C_1>0$ such that
    \begin{align*}
        &\mathbb{E}\sup\limits_{0\leq s\leq T}W(s\wedge\tau^n_R,\tilde{X}^{(n)}(s\wedge\tau^n_R))-W(0,\Pi_n x)\\
        &=\mathbb{E}\sup\limits_{0\leq s\leq T}\int_0^{s\wedge\tau^n_R}\mathscr{L}W(r,\tilde{X}^{(n)}(r))+\frac{\partial W}{\partial r}(r,\tilde{X}^{(n)}(r))dr\\
        &\quad +\mathbb{E}\sup\limits_{0\leq s\leq T}\int_0^{s\wedge\tau^n_R} 2\theta pV^{p-1}(\tilde{X}^{(n)}(r))(\tilde{X}^{(n)}_1(r))^{2\theta-1}\rho(1)dW_1(r)\\
        &\leq C_1\mathbb{E}\left[\int_0^T \left(2\theta pV^{p-1}(\tilde{X}^{(n)}(r))(\tilde{X}^{(n)}_1(r))^{2\theta-1}\rho(1)\right)^2dr\right]^{1/2}\\
        &\leq C_1[\int_0^T \mathbb{E}V^{2p}(\tilde{X}^{(n)}(r))dr]^{1/2}\\
        &\leq C_1V^{p}(x).
    \end{align*}
    Similarly, applying Fatou's lemma yields
    \begin{equation*}
        \mathbb{E}\sup\limits_{0\leq s\leq T}V^p(\tilde{X}^n(s))\leq (1+C_1)V^p(x)e^{CT}.
    \end{equation*}
    This immediately gives,
    \begin{equation*}
    \begin{split}
    \mathbb{E}\sup\limits_{0\leq s\leq T}&\left[\sum\limits_{i=1}^{\infty}\vert \tilde{X}^{(n)}(s)\vert^{2\theta}\rho(i)\right]^p\leq \mathbb{E}\sup\limits_{0\leq s\leq T}\left[\sum\limits_{i=1}^{\infty}\vert \tilde{X}^{(n)}(s)\vert^{2\theta}\rho(i)+1\right]^p=\mathbb{E}\sup\limits_{0\leq s\leq T}V^p(\tilde{X}^{(n)}(s))\\
    &\leq (1+C_1)\left[\sum\limits_{i=1}^{\infty}\vert x_i\vert^{2\theta}\rho(i)+1\right]^pe^{CT}\leq 2^{p}(1+C_1)e^{CT}(1+\Vert x\Vert_{2\theta,\rho}^p)
    \end{split}
    \end{equation*}
    which is precisely the statement of (\ref{ineq3}). The proof of (\ref{ineq4}) can be obtained by the same method as in Step 1, therefore, we omit the details here.

		Step 2: Proof of (\ref{ineq1}). Assuming that $0\leq s\leq t\leq T$, we have
		\begin{equation}
			\label{CALEXPECT1} 
            \begin{split}
                \mathbb{E}\Vert \tilde{X}^{(n)}(t)-\tilde{X}^{(n)}(s) \Vert_{2,\rho}^{4}  	&=\mathbb{E}\left[\left(\sum\limits_{i=1}^n\vert {X}_i^{(n)}(t)-{X}_i^{(n)}(s)\vert^{2}\rho(i)\right)^2\right]\\
                &=\mathbb{E}\left[ \sum\limits_{i,j=1}^n\vert {X}_i^{(n)}(t)-{X}_i^{(n)}(s)\vert^{2}\vert {X}_j^{(n)}(t)-{X}_j^{(n)}(s)\vert^{2}\rho(i)\rho(j) \right].	
            \end{split}	
		\end{equation} 
        The estimate of each term in the summation of (\ref{CALEXPECT1}) is obtained by employing (\ref{fdissipa}), H$\ddot{\text{o}}$lder and Burkholder-Davis-Gundy inequalities,
		\begin{align}
				&\mathbb{E}\vert {X}_i^{(n)}(t)-{X}_i^{(n)}(s)\vert^2\vert {X}_j^{(n)}(t)-{X}_j^{(n)}(s)\vert^{2}\rho(i)\rho(j)\nonumber\\
                &\leq \frac{1}{2}\mathbb{E}\vert {X}_i^{(n)}(t)-{X}_i^{(n)}(s)\vert^4\rho^2(i)+\frac{1}{2}\mathbb{E}\vert {X}_j^{(n)}(t)-{X}_j^{(n)}(s)\vert^{4}\rho^2(j)\nonumber\\
                &=\frac{1}{2}\mathbb{E}\left(\int_{s}^{t}a_{i,i-1}{X}_{i-1}^{(n)}(u)+a_{i,i+1}{X}_{i+1}^{(n)}(u)+f({X}_{i}^{(n)}(u))du+\mathbf{1}_{\{i=1\}}\int_{s}^{t}dW_1(u)\right)^4\rho^2(i)\nonumber\\
                & \quad +\frac{1}{2}\mathbb{E}\left(\int_{s}^{t}a_{j,j-1}{X}_{j-1}^{(n)}(u)+a_{j,j+1}{X}_{j+1}^{(n)}(u)+f({X}_{j}^{(n)}(u))du+\mathbf{1}_{\{j=1\}}\int_{s}^{t}dW_1(u)\right)^4\rho^2(j)\nonumber\\
				&\leq \tilde{C} \left[\mathbb{E}\left(\int_{s}^{t}a_{i,i-1}{X}_{i-1}^{(n)}(u)+a_{i,i+1}{X}_{i+1}^{(n)}(u)+f({X}_{i}^{(n)}(u))du\right)^4\rho^2(i) \nonumber\right.\\
                &\quad  +\mathbb{E}\left(\int_{s}^{t}a_{j,j-1}{X}_{j-1}^{(n)}(u)+a_{j,j+1}{X}_{j+1}^{(n)}(u)+f({X}_{j}^{(n)}(u))du\right)^4\rho^2(j) \nonumber\\
                &\quad \left.+\mathbb{E}\left(\mathbf{1}_{\{i=1\}}\int_{s}^{t}dW_1(u)\right)^4\rho^2(i) +\tilde{C}\mathbb{E}\left(\mathbf{1}_{\{j=1\}}\int_{s}^{t}dW_1(u)\right)^4\rho^2(j)\right]\nonumber\\
				&\leq \tilde{C}\left[ \vert t-s \vert^3 \mathbb{E}\int_{s}^{t}\left(a_{i,i-1}{X}_{i-1}^{(n)}(u)+a_{i,i+1}{X}_{i+1}^{(n)}(u)+f({X}_{i}^{(n)}(u))\right)^4du\rho^2(i)\right.\nonumber\\
                &\quad + \vert t-s \vert^3 \mathbb{E}\int_{s}^{t}\left(a_{j,j-1}{X}_{j-1}^{(n)}(u)+a_{j,j+1}{X}_{j+1}^{(n)}(u)+f({X}_{j}^{(n)}(u))\right)^4du\rho^2(j)\nonumber\\
                &\quad \left.+\mathbf{1}_{\{i=1\}}\left(\int_{s}^{t}du\right)^2\rho^2(i)+\mathbf{1}_{\{j=1\}}\left(\int_{s}^{t}du\right)^2\rho^2(j)\right]\nonumber\\
				&\leq \tilde{C}\left[ \vert t-s \vert^3 \int_{s}^{t}\mathbb{E} \left(\vert{X}_{i-1}^{(n)}(u)\vert^4+\vert{X}_{i+1}^{(n)}(u)\vert^4+1+\vert {X}_{i}^{(n)}(u)\vert^{4\theta} \right)du\rho^2(i) +\mathbf{1}_{\{i=1\}}\vert t-s \vert^2\rho^2(i)\right.\nonumber\\
                &\quad \left.+\vert t-s \vert^3 \int_{s}^{t}\mathbb{E} \left(\vert{X}_{j-1}^{(n)}(u)\vert^4+\vert{X}_{j+1}^{(n)}(u)\vert^4+1+\vert {X}_{j}^{(n)}(u)\vert^{4\theta} \right)du\rho^2(j)+\mathbf{1}_{\{j=1\}}\vert t-s \vert^2\rho^2(j)\right],\label{CALEXPECT2}
		\end{align}
		where $\mathbf{1}$ denotes the indicator function and $\tilde{C}$ represents a positive constant that can take different values at different occurrences. An application of (\ref{ineq3}) and (\ref{ineq4}), together with the substitution of (\ref{CALEXPECT2}) into (\ref{CALEXPECT1}), yields positive constants $C_1(T),C_2(T)$ and $C(T)>0$ satisfying
		\begin{equation*}
			\begin{split}
				&\mathbb{E}\Vert \tilde{X}^{(n)}(t)-\tilde{X}^{(n)}(s) \Vert_{2,\rho}^{4} \\
                &\leq C_1(T) \vert t-s\vert^3\int_{s}^{t}\mathbb{E} \left(\Vert{X}^{(n)}(u)\Vert_{2,\rho}^4+\Vert{X}^{(n)}(u)\Vert_{2,\rho}^4+1+\Vert {X}^{(n)}(u)\Vert_{2\theta,\rho}^{4\theta} \right)du\\
                &\quad +C_1(T)\vert t-s \vert^2\rho^2(1) \leq C_2(T)\vert t-s \vert^4 \left(1+\Vert x \Vert_{2,\rho}^4+ \Vert x \Vert_{2\theta,\rho}^{4\theta} \right)+C_2(T)\vert t-s \vert^2\rho^2(1)\\
				&\leq C(T)(1+ \Vert x \Vert_{2\theta,\rho}^{4\theta})\vert t-s \vert^2.
			\end{split}
		\end{equation*}
        Thus, (\ref{ineq1}) is proved.

        Step 3: Proof of (\ref{ineq2}). In order to establish the validity of (\ref{ineq2}), we claim that 
        \begin{equation}
            \label{expectationboundedbyinitialcondition}
            \begin{split}
                &\mathbb{E}\sum\limits_{i=1}^{N_0(t,\delta)}\frac{1}{N_1(t,\delta)}\left(\vert \tilde{X}_i^{(n)}(t) \vert^2+1\right)\rho(i)+\mathbb{E}\sum\limits_{i=N_0(t,\delta)+1}^{\infty}\left(\vert \tilde{X}_i^{(n)}(t) \vert^2+1\right)\rho(i)\\
                &\leq \frac{1}{N_1(t,\delta)}\sum\limits_{i=1}^{N_0(t,\delta)}e^{ct}\left(\vert x_i \vert^2+1\right)\rho(i)+\sum\limits_{i=N_0(t,\delta)+1}^{\infty}e^{ct}\left(\vert x_i \vert^2+1\right)\rho(i),
            \end{split}
        \end{equation}
        where $c=4M(2N+1)+\eta\sum\limits_{i=1}^{\infty}\rho(i)+1$, $M$ and $N$ are given in Assumption \ref{A2}, $N_0(t,\delta)$ and $N_1(t,\delta)$ will be determined in the following way.
		  For arbitrary $x\in \ell_\rho^{2\theta}(\mathbb{Z}^+)\subset \ell_\rho^{2}(\mathbb{Z}^+)$,  $\delta>0$ and $t\geq 0$, there exist constants $N_0(t,\delta),N_1(t,\delta)\in \mathbb{N}$ such that
		\begin{align}
			\label{deltainequality}
			  \sum\limits_{i=N_0(t,\delta)+1}^\infty e^{ct}\left(\vert x_i \vert^{2}+1\right)\rho(i)<\frac{\delta}{2}, \ \frac{1}{N_1(t,\delta)}\sum\limits_{i=1}^{N_0(t,\delta)} e^{ct} \left(\vert x_i \vert^{2}+1\right)\rho(i)<\frac{\delta}{2}.
		\end{align}
        Then, we get the desired result by combining (\ref{expectationboundedbyinitialcondition}) and (\ref{deltainequality}), leading to the following estimate,
        \begin{align*}
                \mathbb{E}\sum\limits_{i=N_0(t,\delta)+1}^\infty\vert \tilde{X}_{i}^{(n)}(t) \vert^2\rho(i)&\leq \frac{1}{N_1(t,\delta)}\sum\limits_{i=1}^{N_0(t,\delta)}e^{ct}\left(\vert x_i \vert^2+1\right)\rho(i)+\sum\limits_{i=N_0(t,\delta)+1}^{\infty}e^{ct}\left(\vert x_i \vert^2+1\right)\rho(i)\\
				&\leq \delta.
        \end{align*}
	      To complete the proof of this lemma, we finally demonstrate the claim (\ref{expectationboundedbyinitialcondition}). For $y\in \ell^2_{\rho}(\mathbb{Z^+})$, let 
          \begin{equation*}
             W(y):=\frac{1}{N_1(t,\delta)}\sum\limits_{i=1}^{N_0(t,\delta)}\left(\vert y_i \vert^2+1\right)\rho(i)+\sum\limits_{i=N_0(t,\delta)+1}^{\infty}\left(\vert y_i \vert^2+1\right)\rho(i). 
          \end{equation*}
          We obtain the following estimate,
              \begin{align}
                  &\mathscr{L}_nW(y)\\
                  &=\frac{2}{N_1(t,\delta)}\left[(a_{1,2}y_2+f(y_1))y_1\rho(1)+\sum\limits_{i=2}^{N_0(t,\delta)}(a_{i,i-1}y_{i-1}+a_{i,i+1}y_{i+1}+f(y_i))y_i\rho(i)\right] \nonumber\\
                  &\quad +2\left[\sum\limits_{i=N_0(t,\delta)+1}^{n-1}(a_{i,i-1}y_{i-1}+a_{i,i+1}y_{i+1}+f(y_i))y_i\rho(i)+(a_{n,n-1}y_{n-1}+f(y_n))y_n\rho(n)\right] \nonumber\\
                  &\quad +\frac{2}{N_1(t,\delta)}\rho(1)\nonumber\\
                  &\leq \frac{4M}{N_1(t,\delta)}\sum\limits_{i=1}^{N_0(t,\delta)}\rho(i)(\vert y_{i-1}\vert^2+\vert y_{i}\vert^2+\vert y_{i+1}\vert^2)+\frac{2}{N_1(t,\delta)}\sum\limits_{i=1}^{N_0(t,\delta)}(\eta-\lambda\vert y_i\vert^2)\rho(i) \nonumber\\
                  &\quad +4M\sum\limits_{i=N_0(t,\delta)+1}^{n-1}\rho(i)(\vert y_{i-1}\vert^2+\vert y_{i}\vert^2+\vert y_{i+1}\vert^2)+2\sum\limits_{i=N_0(t,\delta)+1}^{n}(\eta-\lambda\vert y_i\vert^2)\rho(i) +\frac{2}{N_1(t,\delta)}\rho(1)\nonumber\\ 
                  &\leq cW(y), \label{inequalityLnWn}
              \end{align}
          where $c=4M(2N+1)+\eta\sum\limits_{i=1}^{\infty}\rho(i)+1$, as it is in (\ref{expectationboundedbyinitialcondition}). From (\ref{inequalityLnWn}) it follows that the function 
          \begin{equation*}
              G(t,y):=W(y)e^{-ct}
          \end{equation*}
          satisfies $\mathscr{L}_nG+\frac{\partial G}{\partial t}\leq 0$. Let the first exit times of $X^{(n)}$ from the set $\{y\in \ell^{2}_{\rho}(\mathbb{Z}^+);\Vert y \Vert_{2,\rho}^2<m\}$ be defined as $\iota^n_m:=\inf\limits_{t}\{t\geq0;\Vert \tilde{X}^{(n)}(t)\Vert_{2,\rho}^2\geq m\}$. Hence, for $\iota^n_m(t):=min(\iota^n_m,t)$, we have 
          \begin{equation*}
              \mathbb{E}\{W(\tilde{X}^{(n)}(\iota^n_m(t)))e^{c\iota^n_m(t)}\}-\mathbb{E}W(\tilde{X}^{(n)}(0))=\mathbb{E}\int_0^{\iota^n_m(t)}\mathscr{L}_nG(u,\tilde{X}^{(n)}(u))+\frac{\partial G}{\partial t}(u,\tilde{X}^{(n)}(u))du\leq 0.
          \end{equation*}
          This, together with the inequalities $\iota^n_m(t)\leq t,W\geq 0,$ implies
          \begin{equation*}
              \mathbb{E}W(\tilde{X}^{(n)})(\iota^n_m(t)))\leq e^{ct}\mathbb{E}W(\tilde{X}^{(n)}(0)).
          \end{equation*}
          Since we know that $\iota^n_m(t)\rightarrow t,a.s.$ see Khasminskii \cite{khasminskii2011stochastic}, letting $m\rightarrow\infty$, by Fatou's lemma, we now get
          \begin{equation*}
              \mathbb{E}W(\tilde{X}^{(n)}(t))\leq e^{ct}\mathbb{E}W(\tilde{X}^{(n)}(0)).
          \end{equation*}
          Then the desired result follows.
        \end{proof}
        \begin{remark}
            The constants $C(T)$, $C(T,p)$, and $N_0(t,\delta)$ obtained in Lemma \ref{ineqXn} are independent of the dimension of the finite dimensional approximation scheme.
        \end{remark}
        For convenience in the narration, we present the following lemma.
    \begin{lemma}
    \label{inequalityofP}
        For arbitrary $\epsilon>0$ and $t,\delta\in (0,\infty)\bigcap\mathbb{Q}$, there exists a constant $N(t,\delta,\epsilon)\in\mathbb{N}$ such that 
        \begin{equation*}
                \sup\limits_{n}\mathbb{P}\left( \sum\limits_{i=N(t,\delta,\epsilon)+1}^\infty \vert\tilde{X}_i^{(n)}(t) \vert^2 \rho(i)\geq\delta \right)\leq \frac{\epsilon}{2^{m(t)+m(\delta)}},
        \end{equation*}
        where $m$ is a bijection between the rational numbers in $(0,+\infty)$ and the natural numbers $\mathbb{N}$.
    \end{lemma}
    \begin{proof}
    By (\ref{ineq2}), we can deduce that, for arbitrary $\epsilon,t,\delta$, there exists a constant $N(t,\delta,\epsilon)$ such that
    \begin{equation}
    \label{estimateofENinfity}
        \sup\limits_{n}\mathbb{E} \left[\sum\limits_{i=N(t,\delta,\epsilon)+1}^\infty (\vert\tilde{X}_i^{(n)}(t) \vert^2 +1)\rho(i)\right]\leq \frac{\delta\epsilon}{2^{m(t)+m(\delta)}}.
    \end{equation}
    Then, by Chebyshev's inequality and (\ref{estimateofENinfity}), we can obtain 
    \begin{equation*}
        \sup\limits_{n}\mathbb{P}\left( \sum\limits_{i=N(t,\delta,\epsilon)+1}^\infty \vert\tilde{X}_i^{(n)}(t) \vert^2 \rho(i)\geq \delta \right)\leq\sup\limits_{n} \frac{1}{\delta}\mathbb{E}\sum\limits_{i=N(t,\delta,\epsilon)+1}^\infty (\vert\tilde{X}_i^{(n)}(t) \vert^2 +1)\rho(i)\leq \frac{\epsilon}{2^{m(t)+m(\delta)}}.
    \end{equation*}
    \end{proof}
    
	\begin{corollary}
		\label{tightness}
		Let $x\in\ell_\rho^{2\theta}(\mathbb{Z}^+)$ and $\tilde{X}^{(n)}=\left(X^{(n)},0_{i\in\left\{\mathbb{Z}^+\backslash \Lambda_n\right\}}\right)$ where $X^{(n)}$ is the solution of SDE (\ref{CUTSDE}). Then $\tilde{P}_n\coloneqq \mathbb{P}\circ (\tilde{X}^{(n)})^{-1},n\geq 1$ is a tight sequence of measures in $\Omega_{Spin}$.
	\end{corollary}
	
	\begin{proof}
		The estimate (\ref{ineq1}) implies, according to Theorem \ref{Kolmogorovthm}, that the equicontinuity condition is satisfied, that is to say, for arbitrary $\beta>0$ and $T>0$
		\begin{equation}
			\label{T1}
			\lim\limits_{h\downarrow0}\sup\limits_{n}\mathbb{P}\left(\sup\limits_{s,t\in[0,T],\vert t-s \vert \leq h}\Vert \tilde{X}^{(n)}(t)-\tilde{X}^{(n)}(s)\Vert_{2,\rho}^2 > \beta \right)=0.
		\end{equation}
        According to (\ref{T1}), we can deduce that for arbitrary $\epsilon>0$ and $k\in\mathbb{N}$, there exists $h_k\downarrow0$, for any $n\geq1$
        \begin{equation*}
            \mathbb{P}\left(\sup\limits_{s,t\in[0,k],\vert t-s \vert \leq h_k}\Vert \tilde{X}^{(n)}(t)-\tilde{X}^{(n)}(s)\Vert_{2,\rho}^2 > \frac{1}{k}\right)\leq \frac{\epsilon}{2^{k+1}}.
        \end{equation*}
        Hence, we obtain 
        \begin{equation*}
            \mathbb{P}\left(\bigcap\limits_{k=1}^{\infty}\left\{\sup\limits_{s,t\in[0,k],\vert t-s \vert \leq h_k}\Vert \tilde{X}^{(n)}(t)-\tilde{X}^{(n)}(s)\Vert_{2,\rho}^2 \leq \frac{1}{k}\right\}\right)>1-\frac{\epsilon}{2}.
        \end{equation*}
		Moreover, set
        \begin{equation*}
            A_n:=\bigcap\limits_{\delta\in(0,\infty)\bigcap\mathbb{Q}}\bigcap\limits_{t\in(0,\infty)\bigcap\mathbb{Q}}\left\{\sum\limits_{i=N(t,\delta,\epsilon)+1}^{\infty}\vert \tilde{X}^{(n)}_i(t)\vert^2\rho(i)<\delta\right\},
        \end{equation*}
        where $N(t,\delta,\epsilon)$ as in lemma \ref{inequalityofP}. Then we can deduce
        \begin{align}
            &\inf\limits_n \mathbb{P}\left(\forall \delta \in (0,+\infty)\bigcap \mathbb{Q},\forall t \in (0,+\infty)\bigcap \mathbb{Q}, \sum\limits_{i=N(t,\delta,\epsilon)+1}^{+\infty}\vert\tilde{X}_i^{(n)}(t) \vert^2 \rho(i)<\delta \right)\nonumber\\
            &=\inf\limits_n \mathbb{P}(A_n)=1-\sup\limits_n \mathbb{P}(A_n^c)\nonumber\\
            &=1-\sup\limits_n \mathbb{P}\left(\bigcup\limits_{\delta\in(0,\infty)\bigcap\mathbb{Q}}\bigcup\limits_{t\in(0,\infty)\bigcap\mathbb{Q}}\left\{\sum\limits_{i=N(t,\delta,\epsilon)+1}^{\infty}\vert \tilde{X}^{(n)}_i(t)\vert^2\rho(i)\geq\delta\right\}\right)\nonumber\\
            &\geq1-\sum\limits_{\delta\in(0,\infty)\bigcap\mathbb{Q}}\sum\limits_{t\in(0,\infty)\bigcap\mathbb{Q}}\sup\limits_n \mathbb{P}\left(\sum\limits_{i=N(t,\delta,\epsilon)+1}^{\infty}\vert\tilde{X}^{(n)}_i(t)\vert^2\rho(i)\geq\delta\right)\nonumber\\
            &\geq 1-\sum\limits_{\delta\in(0,\infty)\bigcap\mathbb{Q}}\sum\limits_{t\in(0,\infty)\bigcap\mathbb{Q}}\frac{\epsilon}{2^{m(t)+m(\delta)}}>1-\frac{\epsilon}{2}.\label{T2}
        \end{align}
		Set
		\begin{equation*}
			\begin{split}
				K_\epsilon&\coloneqq  \left\{\bigcap_{k=1}^\infty\left\{ \omega :\sup\limits_{s,t\in[0,k],\vert t-s \vert \leq h_k}\Vert \omega(t)-\omega(s)\Vert_{2,\rho}^2\leq \frac{1}{k} \right\}\right\} \\
				&\quad \bigcap \left\{\forall \delta \in (0,+\infty)\bigcap \mathbb{Q},\forall t \in (0,+\infty)\bigcap \mathbb{Q}, \sum\limits_{i=N(t,\delta,\epsilon)+1}^{+\infty}\vert\omega_i(t) \vert^2 \rho(i)<\delta  \right\}
			\end{split}	
		\end{equation*}
		By Lemma \ref{Precompactcondition} and Theorem \ref{AAthm}, we obtain that the set $K_\epsilon$ is precompact in $\Omega_{Spin}=C([0,\infty),$ $\ell_\rho^{2}(\mathbb{Z}^+))$. In addition, (\ref{T1}) and (\ref{T2}) yield
		$$
		\inf\limits_{n}\tilde{P}_n(K_\epsilon)\geq 1-\epsilon.
		$$
		By definition, $\left\{\tilde{P}_n\right\}$ is tight.
	\end{proof}
	
	%%%%%%%%%%%%%%%%%%%%%%%%%%%%%%%%%%%%%%%%%%%%%%
	%%%%%%%%%%%%%%%%%%%%%%%%%%%%%%%%%%%%%%%%%%%%%%%%%%%%%%%%%%%%%%%%
	%%%%%%%%%%%%%%%%%%%%%%%%%%%%%%%%%%%%%%%%%%%%%%%%%%%%%%%%%%%%%%%%%%%%%%%%%%%%%
    
	\subsection{Existence of solutions of the spin system in infinite lattice}
	In this subsection, we examine the existence of solutions to equation (\ref{INFSDE2}). For $x\in  \ell_\rho^{2\theta}(\mathbb{Z}^+)$, we will show that the weak limit $\mathbf{P}^x$ of the sequence $P_n=\left\{\mathbb{P} \circ\left(\tilde{X}^{(n)}\right)^{-1}\right\}$  serves as a solution to the martingale problem defined by equation (\ref{INFSDE2}) with the initial law $\delta_x$.
	
	For the sake of clarity, we establish the following symbols before presenting the main conclusions and their proofs. Recall that $X(t,\omega)=\omega(t)$ denotes the canonical process on the measurable space $(\Omega_{Spin},\mathscr{B}^{Spin})$ with filtration $\{\mathscr{B}^{Spin}_t\}$. We further introduce spaces $\Omega_n=C\left([0, \infty),\mathbb{R}^{\Lambda_n}\right)$, $X_n(t,\omega_n)=\omega_n(t)$, $\omega_n\in\Omega_n$ the canonical process on $\Omega_n$ and the embedding mappings
    \begin{align}
        & \chi_n:\mathbb{R}^{\Lambda_n} \rightarrow \mathbb{R}^{\mathbb{Z}^+}, \quad \chi_n\left(x_1, \ldots, x_n\right)=\left(x_1, \ldots, x_n, 0_{i \in \mathbb{Z}^+ \backslash \Lambda_n}\right), \nonumber\\
		& \psi_n: \Omega_n \rightarrow \Omega, \quad \psi_n(\omega_n)= \left(\omega_n(\cdot), 0_{i \in \mathbb{Z}^+ \backslash \Lambda_n}\right).\label{chidefinition}
    \end{align}
	For a given $x \in \ell_\rho^{2\theta}(\mathbb{Z}^+)$, we denote $X^{(n), x}$ and $\tilde{X}^{(n), x}$ the processes constructed in the previous section to accentuate their dependence on $x$, i.e. $X^{(n), x}$ is the solution to SDE (\ref{CUTSDE}) with $ X^{(n),x}(0)=\Pi_{\Lambda_n}(x)$ and $\tilde{X}^{(n),x}=\left(X^{(n),x}, 0_{i \in \mathbb{Z}^+ \backslash \Lambda_n}\right)$. To simplify the notation, we denote $P_n^x=\mathbb{P} \circ\left(X^{(n), x}\right)^{-1}$ on $\Omega_n$ and $\tilde{P}_n^x=\mathbb{P} \circ(\tilde{X}^{(n), x})^{-1}$ on $\Omega_{Spin}$, the corresponding expectations will then be denoted by $E_n^x$ and $\tilde{E}_n^x$, respectively. By virtue of the tightness of $\{\tilde{P}^x_n\}$ in $\Omega_{Spin}$ established in Corollary \ref{tightness}, we may extract a weakly convergent subsequence $\{\tilde{P}^x_{n_k}\}$ (by Prokhorov's theorem, see Da Prato and Zabczyk \cite[Theorem 2.3]{stochasticequationsininfinitedimensionsdaza2014}). For simplicity, we retain the original indexing and write the convergent subsequence simply as $\{\tilde{P}^x_n\}$. Denoting its weak limit by $\mathbf{P}^x$, we have: $\tilde{P}^x_{n}\stackrel{\text{w}}{\rightarrow} \mathbf{P}^x$ weakly as $k\rightarrow\infty$, where the convergence is in the topology of weak convergence of probability measures on $\Omega_{Spin}$. We denote by $\mathbf{E}^x$  the corresponding expectation with respect to $\mathbf{P}^x$.  We observe that $\tilde{P}_n^x=P_n^x \circ \psi_n^{-1}$. Indeed, for any measurable set $C \in \mathcal{F}^{Spin}$, we have
	$$
	\tilde{P}_n^x(C)=\mathbb{P}\left(\tilde{X}^{(n), x}(\cdot) \in C\right) =\mathbb{P}\left(\psi_n\left(X^{(n), x}\right) \in C\right) 
	=P_n^x \circ \psi_n^{-1}(C).
	$$
	Employing the notational framework introduced above and the characterization of the solution developed in Section 3.1, we are positioned to formulate the fundamental existence theorem.
	\begin{theorem}
		\label{existencesolution}
		Let $x\in  \ell_\rho^{2\theta}(\mathbb{Z}^+)$. Then there exists a probability measure $\mathbf{P}^x$ on $(\Omega_{Spin},\mathscr{B}^{Spin})$ which is a solution to the martingale problem associated to equation (\ref{INFSDE2}) with the initial law $\delta_x$.
		As a result, a weak martingale solution to  SDE (\ref{INFSDE2}) is available.
	\end{theorem}
	\begin{proof}
		Having defined the probability measure $\mathbf{P}^x$ in the preceding construction, we now aim to demonstrate that $\mathbf{P}^x$ constitutes a solution to the martingale problem in the sense of Definition \ref{solutiontothemartingaleproblemdef}. 
        
        Step 1. For arbitrary but fixed parameters $R,N>0$, define the functional
        \begin{equation*}
            h_{N,R}(\omega)=\left[\sup\limits_{t\in[0,T]}\sum\limits_{i=1}^{N}\vert \omega_i(t)\vert^{2\theta}\rho(i)\right]\wedge R ,\quad \omega\in\Omega_{Spin}
        \end{equation*}
        which clearly belongs to $C_b(\Omega_{Spin})$. Under the weak convergence $\tilde{P}^x_{n}\stackrel{\text{w}}{\rightarrow} \mathbf{P}^x$, it follows that
        \begin{equation*}
            \lim\limits_{n\rightarrow\infty}\mathbb{E}h_{N,R}(\tilde{X}^{(n),x})=\lim\limits_{n\rightarrow\infty}\tilde{E}_n^x h_{N,R}(X)=\mathbf{E}^xh_{N,R}(X).
        \end{equation*}
        From (\ref{ineq3}) and
        \begin{equation*}
            h_{N,R}(X)\leq \sup\limits_{t\in[0,T]}\Vert X(t) \Vert_{2\theta,\rho}^{2\theta},
        \end{equation*}
        we obtain 
        \begin{equation*}
            \mathbf{E}^xh_{N,R}(X)\leq \lim\limits_{n\rightarrow\infty}\mathbb{E}\sup\limits_{t\in[0,T]}\Vert \tilde{X}^{(n),x}(t)\Vert_{2\theta,\rho}^{2\theta}\leq  C(T,2\theta)(1+\Vert x \Vert_{2\theta,\rho}^{2\theta}).
        \end{equation*}
        By additionally employing the monotone convergence theorem, we derive
        \begin{equation*}
            \mathbf{E}^x\sup\limits_{t\in[0,T]}\Vert X(t)\Vert_{2\theta,\rho}^{2\theta}\leq C(T,2\theta)(1+\Vert x \Vert_{2\theta,\rho}^{2\theta}).
        \end{equation*}
        Moreover, an application of Chebyshev's inequality yields
        \begin{equation*}
            \mathbf{P}^x\left(\sup\limits_{t\in[0,T]}\Vert X(t)\Vert_{2\theta,\rho}<\infty\right)=1,
        \end{equation*}
        which means $\mathbf{P}^x$ satisfies condition (i) of Definition \ref{solutiontothemartingaleproblemdef}.
        
		Step 2. Let $h \in C_c^{2, C y l}(\ell_\rho^{2}(\mathbb{Z}^+))$ be given. To prove that 
        \begin{equation*}
            M^h(t)=h(X(t))-h(X(0))-\int_0^t\mathscr{L}h(X(s))ds
        \end{equation*}
        is a martingale, i.e.
		$$
		\mathbf{E}^x\left[\left(h\left(X(t)\right)-h\left(X(s)\right)-\int_s^t \mathscr{L} h\left(X(u)\right) d u\right) \mid \mathcal{F}^{Spin}_s \right]=0 ,
		$$
		it suffices by  Lemma 3.1 of Hofmanov{\'a} and Seidler  \cite{hofmanova2012weak} to prove for any $G \in C(\bar{\Omega}_s,[0,1])$, where $\bar{\Omega}_s=C([0, s],\ell_\rho^{2}(\mathbb{Z}^+))$, $s<t$, that 
		\begin{equation}
			\label{SMPProof1}
			\mathbf{E}^x\left[\left(h\left(X(t)\right)-h\left(X(s)\right)-\int_s^t \mathscr{L} h\left(X(u)\right) d u\right) G(\varrho_s (X)))\right]=0,
		\end{equation} 	
		where $\varrho_s:\Omega^{Spin}\rightarrow\bar{\Omega}_s$ denotes the restriction to the interval $[0,s]$, i.e. $\varrho_s(X)=(X(t))_{t\in[0,s]}$. By weak convergence $\tilde{P}_n^x \stackrel{\text{w}}{\rightarrow} \mathbf{P}^x$, the formula in (\ref{SMPProof1}) is a limit of
		\begin{equation}
			\label{limitExn}
			\tilde{E}_n^x\left[\left(h\left(X(t)\right)-h\left(X(s)\right)-\int_s^t\mathscr{L} h\left(X(u)\right) d u\right) G(\varrho_s(X))\right].
		\end{equation}	
		Since $h \in C_c^{2, C y l}(\ell_\rho^{2}(\mathbb{Z}^+))$, the operator $\mathscr{L}$ acting on $h$ in fact reduces to $\mathscr{L}^h$, i.e. the operator
		$$
		\mathscr{L}^hu(x)=\sum_{i\in \Lambda}(a_{i,i-1}x_{i-1}+a_{i,i+1} x_{i+1}+f(x_i))\partial_{x_i}u(x)+ \frac{1}{2} \frac{\partial^2}{\partial_{x_1^2}}u(x),
		$$   	
		  where $\Lambda$ is the finite subset corresponding to $h$ in Definition \ref{clydinderfunctiononell2}. Consider that for $n$ large enough such that $\Lambda\subset \Lambda_n$ and hence $\mathscr{L}^h$ equals to $\mathscr{L}_n$ on $\Lambda$, where $\mathscr{L}_n$ is the operator corresponding to the cutoff SDEs (\ref{CUTSDE}). Consequently, by performing explicit calculations on (\ref{limitExn}), we derive
          \begin{equation*}
              \begin{split}
                  &\tilde{E}_n^x\left[\left(h\left(X(t)\right)-h\left(X(s)\right)-\int_s^t\mathscr{L} h\left(X(u)\right) d u\right) G(\varrho_s(X))\right]\\
                  &=\mathbb{E}\left[\left(h\left(\tilde{X}^{(n),x}(t)\right)-h\left(\tilde{X}^{(n),x}(s)\right)-\int_s^t\mathscr{L}^h h\left(\tilde{X}^{(n),x}\right) d u\right) G(\varrho_s(\tilde{X}^{(n),x}))\right]\\
                  &=\mathbb{E}\left[\left(h\circ\chi_n\left(X^{(n),x}(t)\right)-h\circ\chi_n\left(X^{(n),x}(t)(s)\right)-\int_s^t\mathscr{L}_n h\circ\chi_n\left(X^{(n),x}(u)\right) d u\right)\right.\\ 
                  &\quad\quad\quad \left. G(\varrho_s\circ\psi_n(X^{(n),x}))\right]\\
                  &=E_n^x\left[\left(h\circ\chi_n\left(X_n(t)\right)-h\circ\chi_n\left(X_n(t)(s)\right)-\int_s^t\mathscr{L}_n h\circ\chi_n\left(X_n(u)\right) d u\right) G(\varrho_s\circ\psi_n(X_n)\right]
              \end{split}
          \end{equation*}
		Since we know that $P_n^x$ solves the martingale problem for $\mathscr{L}_n$ on $\Omega_n$, this expression equals to zero and so does  (\ref{SMPProof1}). Our analysis demonstrates that the probability measure $\mathbf{P}^x$ fulfills condition (ii) specified in Definition \ref{solutiontothemartingaleproblemdef}.

        Step 3. With the aid of Portmanteau theorem
		\begin{equation*}
		    \begin{split}
		        \mathbf{P}^x\left(\Vert X(0)-x\Vert_{2\theta,\rho}=0\right)&= \mathbf{P}^x\left(\Vert X(0)-x\Vert_{2,\rho}=0\right)\geq1-\sum_k \mathbf{P}^x\left(\left\Vert X(0)-x\right\Vert_{2,\rho}>\frac{1}{k}\right) \\
			&\geq 1-\sum_k \liminf _n \mathbb{P}\left(\left\Vert \tilde{X}^{n, x}(0)-x\right\Vert_{2,\rho}>\frac{1}{k}\right)
		    \end{split}
		\end{equation*}
		since by construction $\liminf _n \mathbb{P}\left(\left\Vert\tilde{X}^{n, x}(0)-x\right\Vert_{2,\rho}>1 / k\right)=0$ for each $k\in\mathbb{N}$, we see that condition (iii) of Definition \ref{solutiontothemartingaleproblemdef} is indeed satisfied. In summary, $\mathbf{P}^x$ constitutes a solution to the martingale problem associated to (\ref{INFSDE2}) with initial law $\delta_x$. Moreover, by Theorem \ref{eqalweakandmartingale}, the existence of weak solutions to the equation (\ref{INFSDE2}) is guaranteed.
	\end{proof}
	  When the initial law is given by a general probability measure $\mu$ on $\ell_\rho^{2\theta}(\mathbb{Z}^+)$, suitable conditions must be imposed on $\mu$. By modifying our previous analysis accordingly, we obtain analogous existence results for solutions to the martingale problem associated to (\ref{INFSDE2}) with initial law $\mu$ and weak martingale solutions of equation (\ref{INFSDE2}) with the initial condition $\mu$. For conciseness, we present the theorem below while omitting the repetitive technical details.
      \begin{theorem}
          Let $\mu$ be a probability measure on $\ell_\rho^{2\theta}(\mathbb{Z}^+)$ such that $\int_{\ell_\rho^{2\theta}(\mathbb{Z}^+)}\Vert x\Vert_{2\theta,\rho}^{4\theta}\mu(dx)<\infty$. Then there exists at least one solution to the martingale problem defined by (\ref{INFSDE2}) with the initial law $\mu$. Consequently, there exists at least one weak martingale solution to equation (\ref{INFSDE2}) with the initial condition $\mu$.
      \end{theorem}
	%%%%%%%%%%%%%%%%%%%%%%%%%%%%%%%%%%%%%%
	%%%%%%%%%%%%%%%%%%%%%%%%%%%%%%%%%%%%%%%%%%%%%%%%%%%%
	%%%%%%%%%%%%%%%%%%%%%%%%%%%%%%%%%%%%%%%%%%%%%%%%%%%%%%%%%%%%%%%%%%%  
	
	\subsection{Uniqueness of solutions}
	Having demonstrated in Section 3.3 the existence of weak martingale solutions of equation (\ref{INFSDE2}), we presently address the question of uniqueness under suitable hypotheses - specifically, whether solution distributions coincide. The proof of uniqueness extends the method of {\v{Z}}{\'a}k \cite{vzak2016existence}, based on the properties of the finite-volume approximations. Notably, our approach does not require the nonlinear term to be of a linear growth.  To achieve the desired outcome, it is necessary to incorporate the following two assumptions into the solution space and the function $f$.
	\begin{condition}
		\label{uniquenesscondition}
		\begin{description}
			\item[(i)] There exist  constants $ r,M_1,M_2>0$ and $N_1\in \mathbb{Z}^+$ such that for any $n\geq N_1$
			$$
			\rho(n)\geq \frac{M_1}{e^{rn}}, \quad \limsup\limits_{n\rightarrow\infty}\frac{\rho(n)}{\inf\limits_{1\leq i\leq n}\rho(i)}\leq M_2.
			$$ 			
			\item[(ii)] The function $f$ can be decomposed into three parts, $f = f_{1} + f_{2} + f_{3}$, where $f_{2}$ is a globally Lipschitz function with Lipschitz constant $C_f$ (independent of i); $f_{3}$ satisfies that there exists a constant $\xi\in\mathbb{R}$ such that $f_{i,3}(z) + \xi z$ is monotonically decreasing in $z$; and $f_{1}$ is a locally Lipschitz function satisfying 
            \begin{equation*}
				|f_{1}(z)-f_{1}(\tilde{z})|\leq L(N)|z-\tilde{z}|,\quad  -N\leq z,\tilde{z}\leq N,
			\end{equation*}	
            with its local Lipschitz coefficients $L(\cdot):\mathbb{N}\rightarrow\mathbb{R}^+$ satisfying that for any constant $C>0$, there exists $\gamma=\gamma(C)>0$ such that 
             \begin{equation}\label{conditiononflipschitzL}
                \lim\limits_{n\rightarrow \infty}\frac{\exp\{CL^2(n)\}}{n^\gamma L(n)}=0.
            \end{equation}
		\end{description}
	\end{condition}
    \begin{remark}
    \begin{description}
        \item[(i)] It can be verified that Assumption \ref{A2} (ii) and \ref{uniquenesscondition} (i) hold for both categories of special weights,
        \begin{equation*}
            \rho_{\kappa}(i)=e^{-\kappa\vert i\vert},\quad i\in\mathbb{Z}^+,\quad \kappa>0,
        \end{equation*}
        and
        \begin{equation*}
            \rho_{r,\kappa}(i)=\frac{1}{1+\kappa\vert i\vert^r},\quad i\in\mathbb{Z}^+,\quad \kappa>0,\quad r>1.
        \end{equation*}
        \item[(ii)] While Assumption \ref{uniquenesscondition} requires $f$ to satisfy the locally Lipschitz condition, this property does not necessarily extend to $F$ defined in (\ref{AFBdef}). The logarithmic function $L(x)=\sqrt{log(x)}$ serves as a particular case satisfying (\ref{conditiononflipschitzL}), and in this special form we can deduce that $f$ exhibits growth behaviour satisfying $\vert f(x)\vert \leq C(1+\vert x  \sqrt{log(\vert x\vert )}\vert)$.
    \end{description}
    \end{remark}
    For the remainder of this paper, we will assume that Assumption \ref{uniquenesscondition} always holds. If we take some fixed finite set $\Lambda \subset \mathbb{Z}^+$ and two super-sets $\Lambda_n, \Lambda_k \supset \Lambda$, such that we have corresponding solutions $X^{(n),x}, X^{(k),x}$ of finite-dimensional SDEs (\ref{CUTSDE}) on $\Lambda_n$ and $\Lambda_k$ respectively. Due to the dependence of $(X^{(n),x}_i)_{i \in \Lambda}$ and $(X^{(k),x}_i)_{i \in \Lambda}$ on all $X^{(n),x}$ and $X^{(k),x}$ via $A_n$ and $A_k$, respectively, the identity between $(X^{(n),x}_i)_{i \in \Lambda}$ and $(X^{(k),x}_i)_{i \in \Lambda}$ cannot be established. Ideally, as the regions $\Lambda_n$ and $\Lambda_k$ become sufficiently large, the difference between their restrictions of the corresponding solutions to $\Lambda$ becomes arbitrarily small. This is rigorously established in the following lemma.
	\begin{lemma}
		\label{lemmalimitdifference}
        Given $x \in \ell_\rho^{2\theta}(\mathbb{Z}^+)$, let $\{ \Lambda_n \}$ be a sequence of finite sets of $\mathbb{Z}^+$ with $\{ X^{(n),x} \}$ being the solutions to the associated finite-dimensional SDEs (\ref{CUTSDE}).
		For $k\leq n$, we denote by $X^{(n), x}_{(k)}$ the part of $X^{(n), x}$ that lives in the finite set $\Lambda_k$, that is, $X^{(n), x}_{(k)} = (X^{(n),x}_{1}, \ldots, X^{(n),x}_{k})$. Then for any $k>0$, $\epsilon > 0$ and $T > 0$ there exists a constant $V(\epsilon,T,k) > 0$ such that for any $n, m \geq V(\epsilon,T,k)$
		\begin{align} \label{konvergencesamotneposl}
			\mathbb{E}\sup_{t \in [0, T]} \| X^{(n), x}_{(k)}(t) - X^{(m),x}_{(k)}(t) \|^2_{\mathbb{R}^{\Lambda_k}} \leq \epsilon.
		\end{align}
	\end{lemma}
	\begin{proof}
		Let $x \in \ell_\rho^{2\theta}(\mathbb{Z}^+)$ be given. We write $x_{(k)} =\Pi_kx= (x_1, \ldots, x_k)$ for the restriction of $x$ to $\mathbb{R}^{\Lambda_k}$. For any positive integer $d$, the norm of the finite-dimensional space $\mathbb{R}^{\Lambda_d}$ will be denoted by $\Vert\cdot\Vert_d$ for notational convenience.  Given arbitrary $T,N>0$ and $n,m\in\mathbb{N}$, through successive applications of Jensen's inequality, H\"older's inequality, Assumption \ref{uniquenesscondition} and Assumption \ref{A2}, we obtain the following estimate for all $t\in[0,T]$ and $i<n\wedge m$,
		\begin{align*} &\mathbb{E} \vert X^{(n), x}_{i}(t) - X^{(m),x}_{i}(t) \vert^2 
			\\
            &\leq \tilde{C}\tilde{M}t\left(\int_0^t\mathbb{E}\sum\limits_{j=i-1}^{i+1}\left\vert X^{(n),x}_{j}(t_1)-X^{(m),x}_{j}(t_1)\right\vert^2dt_1\right)\\
            &\quad +\tilde{C}t\int_0^t\mathbb{E}\left\vert \mathbf{1}_{\{\vert X^{(n),x}_i(t_1)\vert\vee \vert X^{(m),x}_i(t_1)\vert\leq N\}}\left( f_1(X_{i}^{(n),x}(t_1))-f_1(X^{(m),x}_{i}(t_1))\right) \right.\\
            &\quad \left. \times \left( X_{i}^{(n),x}(t_1)- X^{(m),x}_{i}(t_1)\right)\right\vert dt_1\\
            &\quad +\tilde{C}t\int_0^t\mathbb{E}\left\vert \mathbf{1}_{\{\vert X^{(n),x}_i(t_1)\vert\vee \vert X^{(m),x}_i(t_1)\vert> N\}}\left( f_1(X_{i}^{(n),x}(t_1))-f_1(X^{(m),x}_{i}(t_1))\right)\right.\\
            &\quad \left. \times \left( X_{i}^{(n),x}(t_1)- X^{(m),x}_{i}(t_1)\right)\right\vert dt_1\\
            &\leq \tilde{C}\tilde{M}t\left(\int_0^t\mathbb{E}\sum\limits_{j=i-1}^{i+1}\left\vert X^{(n),x}_{j}(t_1)-X^{(m),x}_{j}(t_1)\right\vert^2dt_1\right)\\
            &\quad +\tilde{C}tL^2(N)\int_0^t\mathbb{E}\left\vert \mathbf{1}_{\{\vert X^{(n),x}_i(t_1)\vert\vee \vert X^{(m),x}_i(t_1)\vert\leq N\}}\left(X_{i}^{(n),x}(t_1)-X^{(m),x}_{i}(t_1)\right)\right\vert^2dt_1\\
            &\quad +\tilde{C}t\int_0^t\mathbb{E}\left\vert \mathbf{1}_{\{\vert X^{(n),x}_i(t_1)\vert\vee \vert X^{(m),x}_i(t_1)\vert> N\}}\left(1+\vert X_{i}^{(n),x}(t_1)\vert^{2\theta}+\vert X^{(m),x}_{i}(t_1)\vert^{2\theta}\right)\right\vert dt_1\\
            &   \leq \tilde{C}\tilde{M}t\left(\int_0^t\mathbb{E}\sum\limits_{j=i-1}^{i+1}\left\vert X^{(n),x}_{j}(t_1)-X^{(m),x}_{j}(t_1)\right\vert^2dt_1\right)\\
            &\quad +\tilde{C}tL^2(N)\int_0^t\mathbb{E}\left\vert X_{i}^{(n),x}(t_1)-X^{(m),x}_{i}(t_1)\right\vert^2dt_1+\frac{\tilde{C}t}{N^{\gamma}}\int_0^t\mathbb{E}\left\vert \mathbf{1}_{\{\vert X_i^{(n),x}(t_1)\vert\vee \vert X_i^{(m),x}(t_1)\vert> N\}}\right.\\
            &\quad \left.\left(1+\vert X_{i}^{(n),x}(t_1)\vert^{2\theta}+\vert X^{(m),x}_{i}(t_1)\vert^{2\theta}\right)\left(\vert X^{(n),x}_i(t_1)\vert^{\gamma}+\vert X^{(m),x}_i(t_1)\vert^{\gamma}\right)\right\vert dt_1\\
            &   \leq \tilde{C}\tilde{M}t\left(\int_0^t\mathbb{E}\sum\limits_{j=i-1}^{i+1}\left\vert X^{(n),x}_{j}(t_1)-X^{(m),x}_{j}(t_1)\right\vert^2dt_1\right)\\
            &\quad +\tilde{C}tL^2(N)\int_0^t\mathbb{E}\left\vert X_{i}^{(n),x}(t_1)-X^{(m),x}_{i}(t_1)\right\vert^2dt_1 \\
            &\quad +\frac{\tilde{C}t}{N^{\gamma}}\int_0^t\mathbb{E}\left\vert 1+\vert X_{i}^{(n),x}(t_1)\vert^{4\theta+2\gamma}+\vert X^{(m),x}_{i}(t_1)\vert^{4\theta+2\gamma}\right\vert dt_1,
		\end{align*}
        where $\tilde{M}=M^2+C_f-\xi$, and $\gamma$ is a positive constant to be determined later in the proof. Here and in what follows, we denote by $\tilde{C}$ a generic positive constant that may change from line to line, incorporating all constant factors without introducing a new notation. We thereby derive
        \begin{align}
            &\mathbb{E} \left[\left\Vert X^{(n), x}_{(k)}(t) - X^{(m),x}_{(k)}(t) \right\Vert_k^2 \right]=\mathbb{E}\left[\sum\limits_{i\in\Lambda_k}\vert X^{(n), x}_{i}(t) - X^{(m),x}_{i}(t) \vert^2 \right]\nonumber\\
           &\leq \tilde{C}\tilde{M}t\int_0^t\mathbb{E}\left[\left\Vert X^{(n),x}_{(k+1)}(t_1)-X^{(m),x}_{(k+1)}(t_1)\right\Vert^2_{k+1}\right]dt_1\nonumber\\
               &\quad +\tilde{C}tL^2(N)\int_0^t\mathbb{E}\left[\left\Vert X_{(k)}^{(n),x}(t_1)-X^{(m),x}_{(k)}(t_1)\right\Vert^2_{k}\right]dt_1\nonumber\\
               &\quad +\frac{\tilde{C}t}{N^{\gamma}}\mathbb{E}\left[\sum\limits_{i\in\Lambda_k}\int_0^t\left\vert 1+\vert X_{i}^{(n),x}(t_1)\vert^{4\theta+2\gamma}+\vert X^{(m),x}_{i}(t_1)\vert^{4\theta+2\gamma} \right\vert dt_1\right]\nonumber\\
               &\leq  \tilde{C}(\tilde{M}+L^2(N))t\int_0^t\mathbb{E}\left[\left\Vert X^{(n),x}_{(k+1)}(t_1)-X^{(m),x}_{(k+1)}(t_1)\right\Vert^2_{k+1}\right]dt_1\nonumber\\
               &\quad +\frac{\tilde{C}t}{N^{\gamma}}\mathbb{E}\left[\sum\limits_{i\in\Lambda_k}\int_0^t\left\vert 1+\vert X_{i}^{(n),x}(t_1)\vert^{4\theta+2\gamma}+\vert X^{(m),x}_{i}(t_1)\vert^{4\theta+2\gamma} \right\vert dt_1\right].\nonumber
        \end{align}
        We claim that there exists a constant $C_1>0$, independent of $k$,  such that 
        \begin{equation}
        \label{rhosumkestimat}
            \rho^{2\theta+\gamma}(k)\mathbb{E}\left[\sup\limits_{0\leq s\leq T}\sum\limits_{i\in\Lambda_k} 1+\vert X_{i}^{(n),x}(s)\vert^{4\theta+2\gamma}+\vert X^{(m),x}_{i}(s)\vert^{4\theta+2\gamma}\right]\leq C_1.
        \end{equation}
        By Assumption \ref{A2} (iii), we have $\theta\geq 1$. Following an argument similar to the proof of (\ref{ineq3}), we first establish the bound that
        \begin{align*}
            &\mathbb{E}\left[\sup\limits_{0\leq s\leq T}\sum\limits_{i\in\Lambda_k}\vert X_i^{(n),x}(s)\vert^{4\theta+2\gamma}\rho^{2\theta+\gamma}(k)\right]\\
            &\leq \mathbb{E}\left[\sup\limits_{0\leq s\leq T}\left(\sum\limits_{i\in\Lambda_k}\vert X_i^{(n),x}(s)\vert^{2}\rho(k)\right)^{2\theta+\gamma}\right]\\
            &\leq \mathbb{E}\left[\sup\limits_{0\leq s\leq T}\left(\sum\limits_{i\in\Lambda_k}\vert X_i^{(n),x}(s)\vert^2\rho(k)+\sum\limits_{i=k+1}^{\infty}\vert X_i^{(n),x}(s)\vert^2\rho(i)\right)^{2\theta+\gamma}\right]\\
            &\leq C(T,2\theta+\gamma)\left(1+\left(\sum\limits_{i\in\Lambda_k}\vert x_i \vert^2\rho(k)+\sum\limits_{i=k+1}^{\infty}\vert x_i\vert^2\rho(i)\right)^{2\theta+\gamma}\right).
        \end{align*}
        Moreover, under Assumption \ref{uniquenesscondition} (i), there exists a constant $C_2$ such that
        \begin{equation*}
            \lim\limits_{k\rightarrow\infty}\sum\limits_{i\in\Lambda_k}\vert x_i\vert^2\rho(k)=\lim\limits_{k\rightarrow\infty}\sum\limits_{i\in\Lambda_k}\vert x_i\vert^2\rho(i)\frac{\rho(k)}{\rho(i)}\leq \lim\limits_{k\rightarrow\infty}\sum\limits_{i\in\Lambda_k}\vert x_i\vert^2\rho(i)\frac{\rho(k)}{\inf\limits_{1\leq i\leq k}\rho(i)}\leq C_2\Vert x\Vert_{2,\rho}^2.
        \end{equation*}
        Combining these results establishes the claimed estimate. Similar calculations show that there exists a constant $C_2$ satisfying for any $k$ and $l$,
        \begin{align}
            &\rho(k+l)\mathbb{E}\left[\sup\limits_{0\leq s\leq T}\sum\limits_{i\in\Lambda_{k+l}}\left(\vert X_i^{(n),x}(s)\vert^2-\vert X_i^{(m),x}(s)\vert^2\right) \right]\nonumber\\
            &\leq \rho(k+l)\mathbb{E}\left[\sup\limits_{0\leq s\leq T}\sum\limits_{i\in\Lambda_{k+l}}\left(\vert X_i^{(n),x}(s)\vert^2+\vert X_i^{(m),x}(s)\vert^2\right) \right]\leq C_2.\label{rhosumkestimat2}
        \end{align}
        As a consequence of (\ref{rhosumkestimat}), we obtain
        \begin{align}
            \mathbb{E} \left\Vert X^{(n), x}_{(k)}(t) - X^{(m),x}_{(k)}(t) \right\Vert_k^2& \leq \tilde{C}(\tilde{M}+L^2(N))t\int_0^t\mathbb{E}\left\Vert X^{(n),x}_{(k+1)}(t_1)-X^{(m),x}_{(k+1)}(t_1)\right\Vert^2_{k+1}dt_1 \nonumber\\
            &\quad +\frac{\tilde{C}t^2}{N^{\gamma}\rho^{2\theta+\gamma}(k)}. \label{oneprocedureestimate}
        \end{align}
        Given any $l\in\mathbb{N}$ such that $k+l<n\wedge m$, iterating (\ref{oneprocedureestimate}) and applying (\ref{rhosumkestimat2}), we derive
        \begin{align}
            & \mathbb{E} \left\Vert X^{(n), x}_{(k)}(t) - X^{(m),x}_{(k)}(t)\right\Vert_{k}^2 \nonumber\\
                &\leq  \tilde{C}(\tilde{M}+L^2(N))t\int_0^t\mathbb{E}\left\Vert X^{(n),x}_{(k+1)}(t_1)-X^{(m),x}_{(k+1)}(t_1)\right\Vert^2_{k+1}dt_1+\frac{\tilde{C}t^2}{N^{\gamma}\rho^{2\theta+\gamma}(k)}\nonumber\\
                &\leq \left[\tilde{C}(\tilde{M}+L^2(N))\right]^2t\int_0^tt_1\int_0^{t_1}\mathbb{E}\left\Vert X^{(n),x}_{(k+2)}(t_2)-X^{(m),x}_{(k+2)}(t_2)\right\Vert^2_{k+2}dt_2dt_1\nonumber\\
                &\quad +\tilde{C}(\tilde{M}+L^2(N))t\frac{\tilde{C}}{N^{\gamma}\rho^{2\theta+\gamma}(k+1)} \int_0^tt_1^2dt_1
                 +\frac{\tilde{C}t^2}{N^{\gamma}\rho^{2\theta+\gamma}(k)}\nonumber\\
                 &\leq \cdots \nonumber\\
                 &\leq  \left[\tilde{C}(\tilde{M}+L^2(N))\right]^l t\int_0^tt_1\int_0^{t_1}t_2\cdots t_{l-1}\int_0^{t_{l-1}}\mathbb{E}\left\Vert X^{(n),x}_{(k+l)}(t_l)-X^{(m),x}_{(k+l)}(t_l)\right\Vert^2_{k+l}dt_{l}\cdots dt_2dt_1\nonumber\\
                &\quad +\frac{\tilde{C}}{N^{\gamma}}\sum\limits_{i=1}^{l} \left[\tilde{C}(\tilde{M}+L^2(N))\right]^{i-1}\frac{1}{\rho^{2\theta+\gamma}(k+i-1)}\frac{t^{2i}}{(2i-1)!!}  \nonumber\\
                & \leq \mathbf{V}_1(N)+\mathbf{V}_2(N),\label{repeatestimationoflamdak}
        \end{align}
		where $(2n-1)!! = (2n-1) \cdot (2n-3) \cdots 3\cdot 1$ denotes the odd factorial and 
        \begin{align*}
            &\mathbf{V}_1(N):=\left[\tilde{C}(\tilde{M}+L^2(N))\right]^l\frac{t^{2l}}{(2l-1)!!}\frac{1}{\rho(k+l)},\\
            &\mathbf{V}_2(N):=\frac{\tilde{C}}{N^{\gamma}}\sum\limits_{i=1}^{l} \left[\tilde{C}(\tilde{M}+L^2(N))\right]^{i-1}\frac{1}{\rho^{2\theta+\gamma}(k+i-1)}\frac{t^{2i}}{(2i-1)!!}.
        \end{align*}
        By (\ref{repeatestimationoflamdak}), to establish the lemma, it suffices to prove that for sufficiently large $l < n\wedge m$ (thereinafter, large $n$ and $m$), there exists $N$ independent of $l$ such that both $\mathbf{V}_1(N)$ and $\mathbf{V_2}(N)$ are sufficiently small.  By Andrews et al. \cite{andrews1999special}, the Taylor series of the error function $\operatorname{erf}(z)=\frac{2}{\sqrt{\pi}}\int_0^{z}e^{-t^2}dt$ is given by
        \begin{equation*}
            \operatorname{erf}(z)=\frac{2}{\sqrt{\pi}}\sum\limits_{n=0}^{\infty}\frac{(-1)^{-n}}{2n+1}\frac{z^{2n+1}}{n!}.
        \end{equation*}
        Using the series expansions of $e^{z^2/2}$ and $\operatorname{erf}(z/\sqrt{2})$, multiplying them term by term, and rearranging the resulting expression, we ultimately obtain
        \begin{equation}\label{ineqdoubleseries}
            \sum\limits_{i=0}^{\infty}\frac{z^{2i+1}}{(2i+1)!!}=\sqrt{\frac{\pi}{2}}e^{\frac{z^2}{2}}\operatorname{erf}(\frac{z}{\sqrt{2}})=e^{\frac{z^2}{2}}\int_0^{z}e^{-\frac{t^2}{2}}dt\leq \sqrt{\frac{\pi}{2}}e^{\frac{z^2}{2}}
        \end{equation}
        From (\ref{ineqdoubleseries}) together with Assumption \ref{uniquenesscondition}, we establish the following estimate for $\mathbf{V}_2(N)$:
        \begin{align}
             \mathbf{V}_2(N)&\leq \frac{\tilde{C}}{N^{\gamma}}\sum\limits_{i=1}^{l} \left[\tilde{C}(\tilde{M}+L^2(N))\right]^{i-1}e^{r(2\theta+\gamma)(k+i-1)}\frac{t^{2i}}{(2i-1)!!}\nonumber\\
            &=\frac{\tilde{C}}{N^{\gamma}}e^{r(2\theta+\gamma)}t^2\sum\limits_{i=1}^{\infty} \left[\tilde{C}(1+L^2(N))e^{r(2\theta+\gamma)}t^2\right]^{i-1}\frac{1}{(2i-1)!!}\nonumber\\
            &\leq \frac{\tilde{C}}{N^{\gamma}}\sum\limits_{i=1}^{\infty} \left[\tilde{C}(1+L^2(N))\right]^{i-1}\frac{1}{(2i-1)!!}\nonumber\\
            &=\frac{\tilde{C}}{\sqrt{\tilde{C}(1+L^2(N))}N^{\gamma}}\sum\limits_{i=0}^{\infty} \sqrt{\tilde{C}(1+L^2(N))}^{2i+1}\frac{1}{(2i+1)!!}\nonumber\\
            &\leq \frac{\tilde{C}}{\sqrt{\tilde{C}(1+L^2(N))}N^{\gamma}}e^{\frac{1}{2}\tilde{C}(1+L^2(N))}\rightarrow0, \quad  \text{as}\quad N\rightarrow\infty,\label{estimateofV2N}
        \end{align}
        where in line 3 we have rescaled $\tilde{C}$ to $\tilde{C}e^{r(2\theta+\gamma)}T^2$ (absorbing the constant factor) and $\gamma$ takes the value specified in Assumption \ref{uniquenesscondition} (ii) for the given constant $\tilde{C}$.
        Therefore, given any $\epsilon>0$, we can find $\tilde{N}=\tilde{N}(\epsilon,\theta,k,T)$ for which the following holds when $N\geq \tilde{N}$: $\mathbf{V}_2(N)\leq \frac{\epsilon}{2}$. With $\tilde{N}$ chosen as above, we can estimate $\mathbf{V}_1(\tilde{N})$ as follows:
        \begin{equation*}
            \mathbf{V}_1(\tilde{N})\leq \left[\tilde{C}(\tilde{M}+L^2(\tilde{N}))\right]^l\frac{t^{2l}}{(2l-1)!!}e^{r(k+l)}\leq \frac{[\tilde{C}(1+L^2(\tilde{N}))]^l}{(2l-1)!!}\rightarrow 0, \quad  \text{as}\quad l\rightarrow\infty.
        \end{equation*}
        Thus, for this fixed $\epsilon>0$, we can find a threshold $V(\epsilon,T,k)>0$ that ensures $\mathbf{V}_1(\tilde{N})\leq \frac{\epsilon}{2}$ for all $l\geq V(\epsilon,T,k)$. We have thus established the conclusion of the lemma.
	\end{proof}

	From Assumptions \ref{uniquenesscondition}, we obtain the property in Lemma \ref{lemmalimitdifference}, which states that when the dimension is sufficiently large, the difference between finite-dimensional equations on a fixed bounded set can be very small. For $x\in \ell_\rho^{2\theta}(\mathbb{Z}^+)$, this property enables us to deduce that in a sequence of probability measures $\{P_n^x\}$, any convergent subsequence weakly converges to a common probability measure. For the uniqueness proof, inspired by Hairer \cite[Proposition 3.31]{hairer2009introduction}, we show in the following lemma that a probability measure on $\Omega_{Spin}=C([0, \infty), \ell_\rho^{2}(\mathbb{Z}^+))$  is uniquely determined by its restriction to finite-dimensional cylindrical sets.
	\begin{lemma}
		\label{finitedimensional determin}
		Let $P$ and $Q$ be two probability measures on $\Omega_{Spin}$ and let $\mathcal{C}$ be the class of the finite-dimensional closed sets, that is to say, the sets of the form
		$$
		C=\left\{z\in C([0, \infty), \ell_\rho^{2}(\mathbb{Z}^+)); (z(t_1),\cdots,z(t_l))\in(q_m^{-1}\Gamma_1,\cdots,q_m^{-1}\Gamma_l) \right\}
		$$ 
		with  $l,m\geq 1$, $T>0$, $t_i\in[0,T]\bigcap \mathbb{Q}$, closed sets $\Gamma_j\subset \mathbb{R}^m$, for $j=1,\cdots,l$ and continuous linear functional $ q_m(x)=\left(x_1,x_2,\cdots,x_m,0_{i\in\mathbb{Z}^+\backslash \Lambda_m}\right)$ $:\ell_\rho^{2}(\mathbb{Z}^+)\rightarrow\mathbb{R}^m$. If $P(C)=Q(C)$ for every $C\in \mathcal{C}$, then $P=Q$.
	\end{lemma}
	\begin{proof}
		Since $P(C)=Q(C)$ for every $C\in \mathcal{C}$, and $\mathcal{C}$ is a $\pi-$class, it follows from the basic uniqueness result in measure theory that $P(C)=Q(C)$ for every $C$ in the $\sigma$-algebra of $\mathcal{E}(\Omega_{Spin})$ generated by $\mathcal{C}$. Thus, it remains to show that $\mathcal{E}(\Omega_{Spin})$ coincides with the Borel $\sigma$-algebra of $\Omega_{Spin}$. As $\Omega_{Spin}$ is a separable metric space under the topology of uniform convergence on compact sets, there exists a countable basis for its topology, which is of the form 
		$$
		B_{\Omega_{Spin}}(\omega^0,\epsilon,n):=\{\omega\in\Omega_{Spin};\sup\limits_{0\leq t\leq n}\|\omega(t)-\omega^0(t)\|_{2,\rho}<\epsilon\}
		$$
		with $n\in\mathbb{N}$, $\omega^0$ belongs to a dense subset of $\Omega_{Spin}$, $\epsilon$ is a positive rational number. In fact, since every finite-dimensional closed set $C$ is a Borel set, it suffices to show that all sets $B_{\Omega_{Spin}}(\omega^0,\epsilon,n)$ (and therefore all open and Borel) sets are contained in $\mathcal{E}(\Omega_{Spin})$. 
		
		For the set $\{x\in\ell_\rho^{2}(\mathbb{Z}^+);\|x\|=1\}$, there exists a dense countable subset $\{x^j\}$ satisfying $x^j=(x^j_1,x^j_2,\cdots,x^j_{m(j)},0_{i\in \mathbb{Z}+ \backslash \Lambda_{m(j)}})$ for some $m(j)\in\mathbb{N}$. Let $\{\zeta_j\}$ be a sequence of continuous linear functions in $\ell_\rho^{2}(\mathbb{Z}^+)$ such that $\zeta_j=(z_1,\cdots,z_{m(j)},0_{i\in \mathbb{Z}+ \backslash \Lambda_{m(j)}})$, $\|\zeta_j\|_{2,\rho}=1$ and $\zeta_j(x^j)=1$ (existence follows from the Hahn-Banach extension theorem, see Brezis \cite{brezis2011functional}). Let $\bar{B}(x^0,a):=\{x\in\ell_\rho^{2}(\mathbb{Z}^+);\|x-x^0\|\leq a\}$ be the closed ball of radius $\epsilon$ centered at $x^0$ in $\ell_\rho^{2}(\mathbb{Z}^+)$. We claim $\bar{B}(0,1)=K:=\bigcap_{j\geq 0}\{x\in\ell_\rho^{2}(\mathbb{Z}^+);|\zeta_j(x)|\leq  1\}$. Since obviously $\bar{B}(0,1)\subset K$, it suffices to show that the reverse inclusion holds. Let $y\in\ell_\rho^{2}(\mathbb{Z}^+)$ with $\|y\|_{2,\rho}>1$ be arbitrary and set $\hat{y}=y/\|y\|_{2,\rho}$. By the density of $x^j$'s, there exists a subsequence $x^{k_j}$ such that $\|x^{k_j}-\hat{y}\|\leq \frac{1}{j}$. Since $1=\zeta_{k_j}(x^{k_j})=\zeta_{k_j}(x^{k_j}-\hat{y})+\zeta_{k_j}(\hat{y})$, we deduce that $\zeta_{k_j}(\hat{y})\geq1-\frac{1}{j}$. By linearity, this implies that $\zeta_{k_j}(y)\geq\|y\|(1-\frac{1}{j})$, so that there exists a sufficiently large $j$ such that $\zeta_{k_j}(y)>1$. This shows that $y\notin K$ and we conclude that $K\subset \bar{B}(0,1)$ as required.

		Let $\{Q_l:=\{t_1,\cdots,t_l\};l\in\mathbb{N}\}$ be a sequence of subsets of $[0,n]\bigcap\mathbb{Q}$ with finite elements such that $\bigcup_{l=1}^{\infty}Q_l=[0,n]\bigcap\mathbb{Q}$. Then we observe that
		\begin{align*}
				&B_{\Omega_{Spin}}(\omega^0,\epsilon,n)\\
                &=\{\omega\in\Omega_{Spin};\sup\limits_{0\leq t\leq n,t\in \mathbb{Q}}\|\omega(t)-\omega^0(t)\|<\epsilon\}\\
				&=\bigcap_{l=1}^{\infty}\{\omega\in\Omega_{Spin};\sup\limits_{0\leq t\leq n,t\in Q_l}\|\omega(t)-\omega^0(t)\|<\epsilon\}\\
				&=\bigcap_{l=1}^{\infty}\{\omega\in\Omega_{Spin};\|\omega(t_1)-\omega^0(t_1)\|<\epsilon,\cdots,\|\omega(t_l)-\omega^0(t_l)\|<\epsilon\}\\
				&=\bigcap_{l=1}^{\infty}\bigcup_{u=1}^{\infty}\{\omega\in\Omega_{Spin};\|\omega(t_1)-\omega^0(t_1)\|\leq\epsilon-\frac{1}{u},\cdots,\|\omega(t_l)-\omega^0(t_l)\|\leq\epsilon-\frac{1}{u}\}\\
				&=\bigcap_{l=1}^{\infty}\bigcup_{u=1}^{\infty}\{\omega\in\Omega_{Spin};\omega(t_1)\in \bar{B}(\omega^0(t_1),\epsilon-\frac{1}{u}),\cdots,\omega(t_l)\in \bar{B}(\omega^0(t_l),\epsilon-\frac{1}{u})\}\\
				&=\bigcap_{l=1}^{\infty}\bigcup_{u=1}^{\infty}\bigcap_{j=1}^{\infty}\{\omega\in\Omega_{Spin};\omega(t_1)\in K_{1,j},\cdots,\omega(t_l)\in K_{l,j}\}
		\end{align*}
	where $K_{i,j}=\{y=(\epsilon-\frac{1}{u})(x+\omega^0(t_i));x\in\ell_\rho^{2}(\mathbb{Z}^+),|\zeta_j(x)|\leq  1\}$ for $i=1,\cdots,l$. Therefore, it is clear that $B_{\Omega_{Spin}}(\omega^0,\epsilon,n)\in\mathcal{E}(\Omega_{Spin})$ and the lemma is established.
	\end{proof}
    \begin{definition}
    	\label{deffinitedimensionalclyinderfunction}
    	A function $h:\Omega_{Spin}\rightarrow\mathbb{R}$ is called a finite-dimensional cylinder function, if there exists $m,l\in \mathbb{N}$, a finite subset $\Lambda_m\subset\mathbb{Z}^+$ and $g:\mathbb{R}^{l\Lambda_m}\rightarrow\mathbb{R}$ such that $h(\omega)=g(\Pi_{\Lambda_m}\omega(t_1),\cdots,\Pi_{\Lambda_m}\omega(t_l))$. Moreover, if $g$ is bounded and Lipschitz i.e. there is a  constant $L>0$ such that 
    	$$
    	|g(z) - g(\tilde{z})| \leq L\|z - \tilde{z}\|_{\mathbb{R}^{l\Lambda_m}} \ \forall z, \tilde{z} \in \mathbb{R}^{l\Lambda_m},
    	$$
    	we say that $h$ is cylindrically bounded and  Lipschitz, such functions $h$ form a function space denoted as $C_{b,lip}^{Cly}(\Omega_{Spin})$.
    \end{definition}

	\begin{corollary}
		\label{corollarypequalq}
		Let $P$ and $Q$ be two probability measures on $\Omega_{Spin}$.  Then $P=Q$ if and only if $Ph=Qh$ for every $h\in C_{b,lip}^{Cly}(\Omega_{Spin})$.
	\end{corollary}
    \begin{proof}
    	The necessity is evident, thus we only need to prove the sufficiency. According to Lemma \ref{finitedimensional determin}, it suffices to demonstrate that $P(C)=Q(C)$ for every $C\in \mathcal{C}$.  For the form of set $C\in \mathcal{C}$ given in Lemma \ref{finitedimensional determin}, we define a sequence of functions on $\Omega_{Spin}$:
    	 \begin{align}
    	 	h_n(\omega)&=[1-n\inf\limits_{z\in C}\sum\limits_{j=1}^l\|\omega(t_j)-z(t_j)\|_{2,\rho}]\vee 0 \nonumber\\
    	 	&=[1-n\inf\limits_{z\in C}\sum\limits_{j=1}^l(\sum\limits_{i=1}^{\infty}|\omega_i(t_j)-z_i(t_j)|^2\rho(i))^{1/2}]\vee 0 \nonumber\\
    	 	&=[1-n\inf\limits_{z\in C}\sum\limits_{j=1}^l(\sum\limits_{i=1}^{m}|\omega_i(t_j)-z_i(t_j)|^2\rho(i))^{1/2}]\vee 0 \label{compactlipschitzfunctionhn}
    	 \end{align}
    	By (\ref{compactlipschitzfunctionhn}), we observe that $h_n\in C_{b,lip}^{Cly}(\Omega_{Spin})$. Hence, the assumptions yield $Ph_n=Qh_n$. We notice that each $h_n$ is bounded and $\lim\limits_{n\rightarrow\infty}h_n(\omega)=\mathbf{1}_{C}(\omega)$, so, by the dominated convergence theorem, $P\mathbf{1}_{C}=Q\mathbf{1}_{C}$ holds for the $C$ mentioned above.
     \end{proof}
	\begin{theorem}
		\label{uniqunessof l2sspace}
		For $x \in \ell_\rho^{2\theta}(\mathbb{Z}^+)$, let $\tilde{X}^{(m), x}, \tilde{X}^{(n),x}$ be two different sequences of approximating processes on $\Omega_{Spin}$. Then there exists a probability measure $\mathbf{P}^x$ on $\Omega_{Spin}$ such that $$\lim_{m \rightarrow \infty} \mathbb{P} \circ (\tilde{X}^{(m), x})^{-1} = \lim_{n \rightarrow \infty} \mathbb{P} \circ (\tilde{X}^{(n), x})^{-1} = \mathbf{P}^x.$$ 
	\end{theorem}
	\begin{proof}
		By Corollary \ref{tightness} we know that any two such sequences have weakly convergent subsequence. So, it remains to show that the limit point is the same for any two weakly convergent subsequences (to simplify the notation, we again call the convergent subsequences
$m$ and $n$) $\lbrace \mathbb{P} \circ (\tilde{X}^{(m), x})^{-1} \rbrace, \lbrace \mathbb{P} \circ (\tilde{X}^{(n), x})^{-1} \rbrace$. We denote the limit point by $\mathbf{P}^x$ and $\tilde{\mathbf{P}}^x$, respectively. By Corollary \ref{corollarypequalq}, to prove $\mathbf{P}^x=\tilde{\mathbf{P}}^x$, it suffices to show that for any $h \in C_{b,lip}^{Cly}(\Omega_{Spin})$ 
		\begin{align} \label{equalcondition1}
			E^{\mathbf{P}^x}h=E^{\tilde{\mathbf{P}}^x}h.
		\end{align}
	    By Portmanteau theorem, if 
	    \begin{align}
	    	\label{equalconditon2}
	    	\lim_m \mathbb{E} h(\tilde{X}^{(m), x} (\cdot))=\lim_n \mathbb{E} h(\tilde{X}^{(n), x}(\cdot)),
	    \end{align}
		then (\ref{equalcondition1}) holds. Let $g$ be the corresponding function that satisfies the conditions of Definition \ref{deffinitedimensionalclyinderfunction}. Then we get for $m, n$ large enough
		\begin{align*} 
			&|\mathbb{E} h(\tilde{X}^{(n), x}(\cdot)) - \mathbb{E} h(\tilde{X}^{(m), x} (\cdot))|^2\\
            &= |\mathbb{E} g(X^{(n), x}_{(k)}(t_1),\cdots,X^{(n), x}_{(k)}(t_l)) -  \mathbb{E} g(X_{(k)}^{(m), x} (t_1),\cdots,X_{(k)}^{(m), x} (t_l))|^2\\
			&\leq \mathbb{E}|g(X^{(n), x}_{(k)}(t_1),\cdots,X^{(n), x}_{(k)}(t_l)) -  g(X_{(k)}^{(m), x} (t_1),\cdots,X_{(k)}^{(m), x} (t_l))|^2 \\
			&\leq L\sup\limits_{j=1,\cdots,l} \mathbb{E}\|X^{(n), x}_{(k)}(t_j) -  X_{(k)}^{(m), x} (t_j)\|^2_{\mathbb{R}^{\Lambda^k}}.
		\end{align*} 
		With the aid of lemma \ref{lemmalimitdifference}, we can derive
		\begin{align*}
			&\lim\limits_{n,m}\sup\limits_{j=1,\cdots,l}\mathbb{E}\|X^{(n), x}_{(k)}(t_j) -  X_{(k)}^{(m), x} (t_j)\|^2_{\mathbb{R}^{\Lambda_k}}=0.
		\end{align*}
		Hence (\ref{equalconditon2}) holds for $h \in C_{b,lip}^{Cly}(\Omega_{Spin})$. 
	\end{proof}

	%%%%%%%%%%%%%%%%%%%%%%%%%%%%%%%%%%%
	%%%%%%%%%%%%%%%%%%%%%%%%%%%%%%%%%%%%%%%%%%%%
	%%%%%%%%%%%%%%%%%%%%%%%%%%%%%%%%%%%%%%%%%%%%%%%%%%%%%
	
	\subsection{Solution as a Markov process}
	This section is devoted to proving the Markov property for the probability measure $\{\mathbf{P}^x\}_{x\in\ell^{2\theta}_{\rho}(\mathbb{Z^+})}$ solving the associated martingale problem. As preliminary results for the main theorem, we first present two fundamental lemmas that will support our arguments. The first of these lemmas provides a rigorous characterization of how the solutions continuously depend on their initial values. 
    \begin{lemma}\label{contionuouswithinitialconditon}
        Let $k\in\mathbb{N}$, $x\in\ell^{2\theta}_{\rho}(\mathbb{Z}^+)$ and $T\geq0$ be given.  Then for all $\epsilon>0$, there exist $\delta>0$ and $V\in\mathbb{N}$ satisfying: whenever $n\geq V$ and $y\in\{z\in\ell^{2\theta}_{\rho}(\mathbb{Z}^+);\Vert z-x\Vert_{2,\rho}\leq\delta\}$, we have 
        \begin{equation*}
            \mathbb{E}\sup\limits_{t\in[0,T]}\Vert X_{(k)}^{(n),x}(t)-X_{(k)}^{(n),y}(t)\Vert_{\mathbb{R}^{\Lambda_k}}^2\leq \epsilon.
        \end{equation*}
    \end{lemma}
    \begin{proof}
        By employing the same notation and estimation techniques as in Lemma \ref{lemmalimitdifference}, we can show that for arbitrary but fixed $k\in\mathbb{N}$, $x,y\in\ell^{2\theta}_{\rho}(\mathbb{Z}^+)$, $0\leq t\leq T$ and sufficiently large $n$, the following estimate holds:
        \begin{align*}
            &\mathbb{E}\left\Vert X^{(n),x}_{(k)}(t)-X^{(n),y}\right\Vert_{k}^2\\
            &\leq \tilde{C}\Vert x_{(k)}-y_{(k)}\Vert^2_{k}+\frac{\tilde{C}t^2}{N^{\gamma}\rho^{2\theta+\gamma}(k)}+\tilde{C}(\tilde{M}+L^2(N))t\int_0^t\mathbb{E}\left\Vert X^{(n),x}_{(k+1)}(t_1)-X^{(n),y}_{(k+1)}(t_1)\right\Vert_{k+1}^2dt_1\\
            &\leq \cdots\\
            &\leq \tilde{\mathbf{V}}_1(N)+\tilde{\mathbf{V}}_2(N)+\tilde{\mathbf{V}}_3(N).
        \end{align*}
        Here
        \begin{align*}
            &\tilde{\mathbf{V}}_1(N)=\frac{\left[\tilde{C}(\tilde{M}+L^2(N))\right]^{l}t^{2l}}{(2l-1)!!\rho(k+l)},\\
            &\tilde{\mathbf{V}}_2(N)=\sum\limits_{i=1}^{l}\frac{\left[\tilde{C}(\tilde{M}+L^2(N))\right]^{i-1}t^{2i}}{(2i-1)!!N^{\gamma}\rho^{2\theta+\gamma}(k+i-1)},\\
            &\tilde{\mathbf{V}}_3(N)=\sum\limits_{i=1}^{l}\tilde{C}\left[\tilde{C}(\tilde{M}+L^2(N))\right]^{i-1}\Vert x_{(k+i-1)}-y_{(k+i-1)}\Vert^2_{k+i-1}\frac{t^{2(i-1)}}{(2i-3)!!},
        \end{align*}
        where $\gamma$ is a positive constant to be determined. Following arguments analogous to those in Lemma \ref{lemmalimitdifference} and using Assumption \ref{uniquenesscondition}, we find that for any $\epsilon>0$, there exists $\tilde{N}$ such that $\tilde{\mathbf{V}}_2(N)\leq \frac{\epsilon}{3}$ holds for all $N\geq \tilde{N}$. Note that 
        \begin{align*}
            \rho(k+i-1)\Vert x_{(k+i-1)}-y_{(k+i-1)}\Vert_{k+i-1}^2
            &\leq \sum\limits_{j=1}^{k+i-1}\vert x_j-y_j\vert^2\rho(j)\frac{\rho(k+i-1)}{\inf\limits_{1\leq p\leq k+i-1}\rho(p)} \\
            &\quad +\sum\limits_{j=k+i}^{\infty}\vert x_j-y_j\vert^2\rho(j)\leq\tilde{C}\Vert x-y\Vert_{2,\rho}^2.
        \end{align*}
        Consequently, for this particular $\tilde{N}$, through computations and rescaling $\tilde{c}$ analogous to (\ref{estimateofV2N}), we obtain the estimate
        \begin{align*}
            \tilde{\mathbf{V}}_3(\tilde{N})&\leq\tilde{C}\sum\limits_{i=1}^{\infty}\frac{\left[\tilde{C}(\tilde{M}+L^2(\tilde{N}))\right]^{i-1}t^{2(i-1)}}{(2i-3)!!\rho(k+i-1)}\Vert x-y\Vert^2_{2,\rho}\\
            & \leq\tilde{C}\sum\limits_{i=1}^{\infty}\frac{\left[\tilde{C}(\tilde{M}+L^2(\tilde{N}))\right]^{i-1}t^{2(i-1)}}{(2i-3)!!}e^{r(k+i-1)}\Vert x-y\Vert^2_{2,\rho}\\
            &\leq \tilde{C}\sum\limits_{i=1}^{\infty}\frac{\left[\tilde{C}(1+L^2(\tilde{N}))\right]^{i-1}}{(2i-3)!!}\Vert x-y\Vert^2_{2,\rho}\\
            &\leq \tilde{C}\Vert x-y\Vert_{2,\rho}^2\left(1+\sqrt{\tilde{C}(1+L^2(\tilde{N}))}\sum\limits_{i=0}^{\infty}\frac{\sqrt{\tilde{C}(1+L^2(\tilde{N}))}^{2i+1}}{(2i+1)!!}\right)\\
            &\leq \tilde{C}\Vert x-y\Vert_{2,\rho}^2\left(1+\sqrt{\tilde{C}(1+L^2(\tilde{N}))}\exp\{\frac{\tilde{C}(1+L^2(\tilde{N}))}{2}\}\right).
        \end{align*}
        With 
        \begin{equation*}
            \delta=\frac{\epsilon}{3}\left[\tilde{C}\left(1+\sqrt{\tilde{C}(1+L^2(\tilde{N}))}\exp\{\frac{\tilde{C}(1+L^2(\tilde{N}))}{2}\}\right)\right]^{-1},
        \end{equation*}
          we have $\tilde{\mathbf{V}}_3(\tilde{N})\leq\frac{\epsilon}{3}$, $\forall y\in\ell^{2\theta}_{\rho}(\mathbb{Z}^+)$ with $\Vert x-y\Vert_{2,\rho}^2\leq \delta$. Following arguments analogous to those in the proof of Lemma \ref{lemmalimitdifference}, we conclude that there exists $V$ such that $\tilde{\mathbf{V}}_1(\tilde{N})\leq \frac{\epsilon}{3}$ holds for all $l\geq V$. The assertion of the lemma follows by synthesizing the above analysis.
    \end{proof}
    \begin{remark}\label{remarkoncontionuouswithinitialconditon}
        While Lemma \ref{contionuouswithinitialconditon} establishes the result for deterministic initial values, the same methodology applies to random initial variables. More precisely, for given $k\in\mathbb{N}$, $x\in\ell^{2\theta}_{\rho}(\mathbb{Z}^+)$ and $T\geq0$ and any $\epsilon>0$, there exist $\delta>0$ and $V\in\mathbb{N}$ such that for all $n\geq V$ and any $\ell^2_{\rho}(\mathbb{Z}^+) $-valued random variable $\xi$ satisfying $\mathbb{E}\Vert \xi-x\Vert_{2,\rho}\leq \delta$, we have 
        \begin{equation*}
            \mathbb{E}\sup\limits_{t\in[0,T]}\Vert X_{(k)}^{(n),x}(t)-X_{(k)}^{(n),\xi}(t)\Vert_{\mathbb{R}^{\Lambda_k}}^2\leq \epsilon.
        \end{equation*}
    \end{remark}
    The second lemma concerns the measurable modification of the canonical process. Following arguments analogous to Lemma 2.2 of Flandoli and Romito  \cite{FlandoliRomito2008Markovselections}, we omit its proof here.
    \begin{lemma}\label{pmodificationcondition}
        Let $P\in Pr(\Omega_{Spin})$ be such that 
        \begin{equation*}
            P(C([0,\infty);\ell^{2\theta}_{\rho,\sigma}(\mathbb{Z^+}))\bigcap \Omega_{Spin})=1.
        \end{equation*}
        Then, for any given $t\geq 0$, the canonical process $X(t,\omega)=\omega(t)$ has a $P$-modification on $\mathscr{B}_t$ which is $\mathscr{B}_t$-measurable with values in $(\ell^{2\theta}_{\rho}(\mathbb{Z^+}),\mathscr{B}(\ell^{2\theta}_{\rho}(\mathbb{Z^+})))$, where $\mathscr{B}(\ell^{2\theta}_{\rho}(\mathbb{Z^+}))$ is the Borel $\sigma$-field of $\ell^{2\theta}_{\rho}(\mathbb{Z^+})$.
    \end{lemma}
    The following is our main result on the Markov property of solutions to (\ref{INFSDE2}).
	\begin{theorem}
		\label{Markov property}
		Let $\mathbf{P}^x$ be the unique solution to the martingale problem. Then $x\mapsto\mathbf{P}^x$ is measurable on $\ell^{2\theta}_{\rho}(\mathbb{Z}^+)$ and $\{\mathbf{P}^x\}_{x\in\ell^{2\theta}_{\rho}(\mathbb{Z}^+)}$ is Markov.
	\end{theorem}
	\begin{proof}
        Through reasoning similar to Exercise 6.7.4 in Stroock and Varadhan  \cite{StroockVaradhan1979Multidimensional}, the well-posedness of the martingale problem guarantees that for each $\Gamma\in\mathscr{B}^{Spin}$, the mapping $x\mapsto\mathbf{P}^x(\Gamma)$ is  $\mathscr{B}(\ell^{2\theta}_{\rho}(\mathbb{Z}^+))$-measurable. As for the Markov property of $\{\mathbf{P}^x\}_{x\in\ell^{2\theta}_{\rho}(\mathbb{Z}^+)}$, we need only demonstrate that
        \begin{equation}\label{Markovc2}
            \ \mathbf{P}^x\left(X(\cdot+t) \in C \mid \mathscr{B}^{Spin}_t\right)=\phi\left(X(t)\right),\quad \mathbf{P}^x-a.s. \quad  \forall C \in \mathscr{B}^{Spin},\, t\geq 0,
        \end{equation}
        where $\phi(y)=\mathbf{P}^{y}\left(X(\cdot) \in C\right)$. By combining Lemma \ref{borelmeasurablesetinomega}, Lemma \ref{pmodificationcondition} with the definition of solutions to the martingale problem, we conclude that for arbitrary $t\geq 0$, the canonical process $X(t)$ admits a $\mathbf{P}^x$-modification on $\mathscr{B}^{Spin}_t$ which is $\mathscr{B}^{Spin}_t$-measurable with values in $(\ell^{2\theta}_{\rho}(\mathbb{Z^+}),\mathscr{B}(\ell^{2\theta}_{\rho}(\mathbb{Z^+})))$. We shall denote the modified process again by $X$, and all subsequent occurrences of $X$ in what follows refer to this  modified process. Hence, $\phi(X(t))$ in (\ref{Markovc2}) is $\mathbf{P}^x$-a.s. well-defined.  
		
		It is obvious that (\ref{Markovc2}) is the same as the statement that for any $B\in \mathscr{B}^{Spin}_t$,
		\begin{align*}
			\mathbf{E}^x(\mathbf{1}_C(X(\cdot+t))\mathbf{1}_B)=\mathbf{E}^x(\phi\left(X(t)\right)\mathbf{1}_B).
		\end{align*}
		Since the collection of sets $C\in \mathscr{B}(\ell^{2\theta}_{\rho}(\mathbb{Z^+})$ for which (\ref{Markovc2}) holds forms a Dynkin system, to prove (\ref{Markovc2}), it suffices to prove that for any $C\in \mathcal{C}$ (defined in Lemma \ref{finitedimensional determin}),
		\begin{align} 
			\label{MarkovDynkinsystemset}
			\mathbf{E}^x[\mathbf{1}_{C}(X(\cdot+t))|\mathscr{B}^{Spin}_t] = \mathbf{E}^{X(t)} [\mathbf{1}_C(X(\cdot))]. 
		\end{align} 
	    Owing to the approximation of $\mathbf{1}_{C}$ by $h_n\in C_{b,lip}^{f,Cyl}(\Omega_{Spin})$, the equation (\ref{MarkovDynkinsystemset}) can be equivalent to  that for any  $h\in C_{b,lip}^{f,Cyl}(\Omega_{Spin})$,
	    \begin{align}
	    	\label{MarkovClipfunction}
	    	\mathbf{E}^x[h(X(\cdot+t))|\mathscr{B}^{Spin}_t] = \mathbf{E}^{X(t)} h(X(\cdot)). 
	    \end{align}
		If we denote $\varphi(x) = \mathbf{E}^x[h(X(\cdot))]$, it is obvious that (\ref{MarkovClipfunction}) is the same as the statement that for any $B\in \mathscr{B}_t^{Spin}$,
		\begin{align}
			\label{integral equation of condition expetation}
			\int_B h(X(\cdot+t)) d\mathbf{P}^x = \int_B \varphi(X(t)) d\mathbf{P}^x.
		\end{align}
        As the indicator function $\mathbf{1}_B$ admits approximation by continuous bounded measurable functions, it suffices to establish that
        \begin{equation}\label{Markovequalcondition}
            \mathbf{E}^x[h(X(\cdot+t,\omega))Y(\omega|_{[0,t]})]=\mathbf{E}^x [\varphi(X(t,\omega))Y(\omega|_{[0,t]})],
        \end{equation}
		where $Y:C([0,t],\ell^{2}_{\rho}(\mathbb{Z}^+))\rightarrow\mathbb{R}$ is arbitrary, but fixed continuous bounded and $\mathscr{B}^{Spin}_t/\mathscr{B}(\mathbb{R})$-measurable function. Recall the definitions of $x_n$ and $\psi$ in (\ref{chidefinition}). Given a function $h\in C_{b,lip}^{f,Cyl}(\Omega_{Spin})$ that fulfills the requirements of Definition \ref{deffinitedimensionalclyinderfunction}, the weak convergence $\tilde{P}^x_{n}\stackrel{\text{w}}{\rightarrow} \mathbf{P}^x$ together with the Markov property for the finite-dimensional system yields:
        \begin{align}
            &\mathbf{E}^x[h(X(\cdot+t,\omega))Y(\omega|_{[0,t]})]\nonumber\\
            &=\mathbf{E}^x[h(\omega(\cdot+t))Y(\omega|_{[0,t]})]\nonumber\\
            &=\lim\limits_{n\rightarrow\infty}\tilde{E}^x_n[h(\omega(\cdot+t))Y(\omega|_{[0,t]})]\nonumber\\
            &=\lim\limits_{n\rightarrow\infty}\tilde{E}^x_n[g(\Pi_{\Lambda_m}\omega(t_1+t),\cdots,\Pi_{\Lambda_m}\omega(t_l+t))Y(\omega|_{[0,t]})]\nonumber\\
            &=\lim\limits_{n\rightarrow\infty}E^x_n[g(\Pi_{\Lambda_m}\chi_n \omega_n(t_1+t),\cdots,\Pi_{\Lambda_m}\chi_nX(t_l+t))Y((\psi_n\omega_n)|_{[0,t]})]\nonumber\\
            &=\lim\limits_{n\rightarrow\infty}E^x_n[E^{\omega_n(t)}_n[g(\Pi_{\Lambda_m}\chi_n \omega_n(t_1),\cdots,\Pi_{\Lambda_m}\chi_n\omega_n(t_l))]Y((\psi_n\omega_n)|_{[0,t]})]\nonumber\\
            &=\lim\limits_{n\rightarrow\infty}\tilde{E}^x_n[\tilde{E}^{\omega(t)}_n[g(\Pi_{\Lambda_m} \omega(t_1),\cdots,\Pi_{\Lambda_m}\omega(t_l))]Y(\omega|_{[0,t]})]\nonumber\\
            &=\lim\limits_{n\rightarrow\infty}\tilde{E}^x_n[\tilde{E}^{\omega(t)}_n[h(\omega(\cdot))]Y(\omega|_{[0,t]})]\label{Markovpropertylimt1}
        \end{align}
        Let $\varphi_n(x)=\tilde{E}^x_n[h(\omega(\cdot))]$. In order to establish (\ref{Markovequalcondition}), we observe from (\ref{Markovpropertylimt1}) that it suffices to prove
        \begin{equation*}
            \lim\limits_{n\rightarrow\infty}\tilde{E}^x_n[\varphi_n(\omega(t))Y(\omega|_{[0,t]})]=\mathbf{E}^x [\varphi(X(t,\omega))Y(\omega|_{[0,t]})].
        \end{equation*}
        By the Skorohod representation theorem (see Da Prato and Zabczyk \cite[Theorem 2.4]{stochasticequationsininfinitedimensionsdaza2014}), since $\tilde{P}^x_{n}\stackrel{\text{w}}{\rightarrow} \mathbf{P}^x$, we can obtain that on some probability space $(\bar{\Omega},\bar{\mathscr{F}},\bar{\mathbb{P}})$, there exist $\Omega_{Spin}$-valued random variables $\bar{X}_n$ and $\bar{X}$ such that
        \begin{itemize}
            \item[(i)] $\bar{X}_n\rightarrow\bar{X}$, in $\Omega_{Spin}$, $\bar{\mathbb{P}}$-a.s.,
            \item[(ii)] $\tilde{P}^x_{n}=\bar{\mathbb{P}}\circ\bar{X}_n^{-1}$, $\mathbf{P}^x=\bar{\mathbb{P}}\circ\bar{X}^{-1}$.
        \end{itemize}
         We claim that
         \begin{equation*}
             x_n,x\in\ell^{2\theta}_{\rho}(\mathbb{Z}^+)\quad\text{and}\quad  x_n\rightarrow x, \quad \text{in}\quad \ell^2_{\rho}(\mathbb{Z}^+) \quad \text{implies}\quad \varphi_n(x_n)\rightarrow\varphi(x),\quad \text{as}\quad n\rightarrow \infty.
         \end{equation*}
        By this claim, since $\bar{X}_n(t),\bar{X}(t)\in\ell^{2\theta}_{\rho}(\mathbb{Z}^+)$ and $\bar{X}_n(t)\rightarrow\bar{X}(t)$ in $\ell^2_{\rho}(\mathbb{Z}^+)$, $\bar{\mathbb{P}}$-a.s., we obtain
        \begin{equation*}
            \varphi_n(\bar{X}_n(t))\rightarrow\varphi(\bar{X}(t)),\quad \bar{\mathbb{P}}-a.s.
        \end{equation*}
        Consequently, an application of the dominated convergence theorem gives
        \begin{align*}
            \lim\limits_{n\rightarrow\infty}\tilde{E}^x_n[\varphi_n(\omega(t))Y(\omega|_{[0,t]})]&=\lim\limits_{n\rightarrow\infty}\bar{\mathbb{E}}[\varphi_n(\bar{X}_n(t))Y(\bar{X}_n|_{[0,t]})]\\
            &=\bar{\mathbb{E}}[\varphi(\bar{X}(t))Y(\bar{X}|_{[0,t]})]=\mathbf{E}^x [\varphi(X(t,\omega))Y(\omega|_{[0,t]})].
        \end{align*}
        To complete the proof of the theorem, it suffices to verify the above claim.  By Lemma \ref{contionuouswithinitialconditon}, for any $\epsilon>0$, there exist $\delta>0$ and $N_1\in\mathbb{N}$ such that when $n>N_1$ and $y\in\{z\in\ell^{2\theta}_{\rho}(\mathbb{Z}^+);\Vert z-x\Vert_{2,\rho}\leq\delta\}$, 
        \begin{equation}\label{limitphiineq1}
            \mathbb{E}\sup\limits_{t\in[0,T]}\Vert X_{(k)}^{(n),x}(t)-X_{(k)}^{(n),y}(t)\Vert_{\mathbb{R}^{\Lambda_k}}^2\leq \frac{\epsilon}{4L^2},
        \end{equation}
        where $L$ is the Lipschitz constant of $g$ in Definition \ref{deffinitedimensionalclyinderfunction}. 
        Since $x_n\rightarrow x$ in $\ell^{2}_{\rho}(\mathbb{Z}^+)$, for the given $\delta>0$, there exists $N_2\in\mathbb{N}$ such that when $n>N_2$, $\Vert x_n-x\Vert_{2,\rho}\leq \delta$. 
        On the other hand, the weak convergence implies that for this $\epsilon$, there exists $N_3\in\mathbb{N}$ such that whenever $n>N_3$,
        \begin{equation}\label{limitphiineq2}
            \left\vert \tilde{E}^x_n[h(\omega(\cdot))]-\mathbf{E}^x[h(X(\cdot))] \right\vert\leq\frac{\epsilon}{4}.
        \end{equation}
        From (\ref{limitphiineq1}) and (\ref{limitphiineq2}), it follows that for this $\epsilon$, whenever $n>N_1\vee N_2\vee N_3$,
        \begin{align*}
            &\vert \varphi_n(x_n)-\varphi(x)\vert^2\\
            &\leq 2\vert \varphi_n(x_n)-\varphi_n(x)\vert^2+2\vert \varphi_n(x)-\varphi(x)\vert^2\\
            &=2\vert \tilde{E}^{x_n}_n[h(\omega(\cdot))]- \tilde{E}^{x}_n[h(\omega(\cdot))]\vert^2+2\vert \tilde{E}^{x}_n[h(\omega(\cdot))]-\mathbf{E}^x[h(X(\cdot))]\vert^2\\
            &\leq 2\mathbb{E}[\vert g(\Pi_{\Lambda_m} \tilde{X}^{(n),x_n}(t_1),\cdots,\Pi_{\Lambda_m}\tilde{X}^{(n),x_n}(t_l))-g(\Pi_{\Lambda_m} \tilde{X}^{(n),x}(t_1),\cdots,\Pi_{\Lambda_m}\tilde{X}^{(n),x}(t_l))\vert^2]\\
            &\quad+2\vert \tilde{E}^{x}_n[h(\omega(\cdot))]-\mathbf{E}^x[h(X(\cdot))]\vert^2\\
            &\leq 2L^2\mathbb{E}[\sup\limits_{t\in[0,T]}\Vert X_{(k)}^{(n),x}(t)-X_{(k)}^{(n),x_n}(t)\Vert_{\mathbb{R}^{\Lambda_k}}^2]+2\vert \tilde{E}^{x}_n[h(\omega(\cdot))]-\mathbf{E}^x[h(X(\cdot))]\vert^2\\
            &\leq \epsilon.
        \end{align*} 	
	\end{proof}
	%%%%%%%%%%%%%%%%%%%
	%%%%%%%%%%%%%%%%%%%%%%%%%%%%%%
	%%%%%%%%%%%%%%%%%%%%%%%%%%%%%%%%%%%%%%%%%%%%%%  
	
 	\subsection{Existence of invariant measure for the transition semi-group}
	From Theorem \ref{Markov property},  it follows that $\{\mathbf{P}^x\}_{x\in\ell^{2\theta}_{\rho}(\mathbb{Z}^+)}$ and the modified process $X=\{X(t),\mathscr{B}^{Spin}_t;$ $t\geq 0\}$ constitutes a Markov process on $(\Omega_{Spin},\mathscr{B}^{Spin})$. We denote by $\mathcal{P}_t$ and $\mathcal{P}_t(x,\Gamma)$, $t\geq  0$, $x\in \ell^{2\theta}_{\rho}(\mathbb{Z}^+)$, $\Gamma\in\mathscr{B}(\ell^{2\theta}_{\rho}(\mathbb{Z}^+))$  the associated transition semi-group on $B_b(\ell^{2\theta}_{\rho}(\mathbb{Z}^+))$ the space of bounded measurable functions on $(\ell^{2\theta}_{\rho}(\mathbb{Z}^+),\mathscr{B}(\ell^{2\theta}_{\rho}(\mathbb{Z}^+))$, and the transition probability kernel, which are defined as  
    \begin{equation*}
        (\mathcal{P}_t\phi)(x)=\mathbf{E}^x[\phi(X(t))],\quad t\geq 0,x\in\ell^{2\theta}_{\rho}(\mathbb{Z}^+), \phi\in B_b(\ell^{2\theta}_{\rho}(\mathbb{Z}^+)),
    \end{equation*}
    and
    \begin{equation*}
        \mathcal{P}_t(x,\Gamma)=\mathcal{P}_t\mathbf{1}_{\Gamma}(x),\quad t\geq 0,x\in \ell^{2\theta}_{\rho}(\mathbb{Z}^+), \Gamma\in\mathscr{B}(\ell^{2\theta}_{\rho}(\mathbb{Z}^+)).
    \end{equation*}
    For any $t\geq 0$, we define the adjoint operator $\mathcal{P}^*_t$ acting on $\mathscr{P}(\ell^{2\theta}_{\rho}(\mathbb{Z}^+))$, the space of probability measures on $(\ell^{2\theta}_{\rho}(\mathbb{Z}^+),\mathscr{B}(\ell^{2\theta}_{\rho}(\mathbb{Z}^+))$  by
    \begin{equation*}
        \mathcal{P}^*_t\mu(\Gamma)=\int_{\ell^{2\theta}_{\rho}(\mathbb{Z}^+)}\mathcal{P}_t(x,\Gamma)\mu(dx),\quad \Gamma\in\mathscr{B}(\ell^{2\theta}_{\rho}(\mathbb{Z}^+)).
    \end{equation*}
    This section focusses on the existence problem of invariant measures $\nu\in\mathscr{P}(\ell^{2\theta}_{\rho}(\mathbb{Z}^+))$ for the transition semi-group $\mathcal{P}_t$, which are defined by the requirement $\mathcal{P}^*_t\nu=\nu$ for each $t\geq 0$. We derive that the tightness of measures $\left\{\nu_n\right\}=\left\{\mu_n \circ \chi^{-1}_n\right\}$ ($\chi^{-1}_n$ is given by (\ref{chidefinition}) on $\ell^{2\theta}_{\rho}(\mathbb{Z}^+)$,  where $\mu_n$ is the invariant measure for the finite-dimensional process $X^{(n)}$ on $(\Omega_n,\mathscr{B}(\Omega_n))$, and consequently show that any limit point is an invariant measure for the transition semi-group $\mathcal{P}_t$.
	To establish the tightness of finite-dimensional invariant measures, we impose the following assumptions on the space weights.
    \begin{condition}\label{invariantmeasuretightness}
    \begin{itemize}
        \item[(i)] Suppose that 
        $$\lambda>\frac{2MN+2(2\theta-1)M}{2\theta}+\theta+\eta+\frac{1}{2\theta}-\frac{1}{2},$$
        where $\theta,\lambda,M,N,\eta$ are given in Assumption \ref{A2};
        \item[(ii)] There exists $v:\mathbb{Z^+}\rightarrow\mathbb{R}^+ \quad  \text{such that}\quad  v(i)>0, i \in \mathbb{Z}^+, \sum_i v(i)<+\infty, \sum_i \frac{\rho(i)}{v(i)}<+\infty.$
    \end{itemize}   
    \end{condition}
	In this subsection, let us work under Assumptions \ref{A2}, \ref{uniquenesscondition}, and \ref{invariantmeasuretightness}.
	\begin{lemma}\label{tightnessinvariantmeasure}
		The sequence of measures $\left\{\nu_n\right\}$ is tight in $\ell^{2\theta}_{\rho}(\mathbb{Z}^+)$.
	\end{lemma}
	\begin{proof}
        We want to show that for a given $\epsilon>0$ there is a compact set $K_\epsilon$ in $\ell_\rho^{2\theta}(\mathbb{Z}^+)$ such that for any $ n \in \mathbb{N}$ one has $\nu_n\left(K_\epsilon\right) \geq 1-\epsilon$. 
	
		Set
        \begin{equation*}
            V_n(x^{(n)}):=\sum\limits_{i=1}^{n}(x_i^{(n)})^{2\theta}v(i), \quad  x^{(n)}\in\mathbb{R}^n.
        \end{equation*}
        An application of Assumptions \ref{A2} (ii), (iii) and \ref{invariantmeasuretightness} (i) combined with computational techniques parallel to (\ref{Lyapucon}) yields
        \begin{align}
            &\mathscr{L}_nV_n(x^{(n)})\nonumber\\
            &=2\theta(a_{1,2}x^{(n)}_2+f(x_1^{(n)}))(x^{(n)}_1)^{2\theta-1}v(1)+2\theta(a_{2,1}x^{(n)}_1+a_{2,3}x^{(n)}_3+f(x^{(n)}_2))(x^{(n)}_2)^{2\theta-1}v(2)+\cdots\nonumber\\
            &\quad +2\theta(a_{n,n-1}x^{(n)}_{n-1}+f(x_n))(x^{(n)}_n)^{2\theta-1}v(n)+\frac{2\theta(2\theta-1)}{2}(x^{(n)}_1)^{2\theta-2}v(1)\nonumber\\
            & \leq 2\theta M\left(\frac{N(x_2^{(n)})^{2\theta}}{2\theta}v(2)+\frac{(2\theta-1)(x_1^{(n)})^{2\theta}}{2\theta}v(1)+\frac{N(x_1^{(n)})^{2\theta}}{2\theta}v(1)+\frac{(2\theta-1)(x_2^{(n)})^{2\theta}}{2\theta}v(2)\right.\nonumber\\
            &\quad \left.+\cdots +\frac{N(x_{n-1}^{(n)})^{2\theta}}{2\theta}v(n-1)+\frac{(2\theta-1)(x_n^{(n)})^{2\theta}}{2\theta}v(n)\right)+2\theta\left(\eta\sum\limits_{i=1}^{n}(x_i^{(n)})^{2\theta-2}v(i)\right.\nonumber\\
            &\quad \left.-\lambda\sum\limits_{i=1}^n(x_i^{(n)})^{2\theta}v(i)\right)+\frac{2\theta(2\theta-1)}{2}(x_1^{(n)})^{2\theta-2}v(1) \nonumber\\
            &\leq (2MN+2(2\theta-1)M)V_n(x^{(n)})-2\theta\lambda V_n(x^{(n)})+(\theta(2\theta-1)+2\theta\eta)(V_n(x^{(n)})+\sum\limits_{i=1}^{n}v(i))\nonumber\\
            &\quad +\theta(2\theta-1)(x^{(n)}_1)^{2\theta-2}v(1)\nonumber\\
            &\leq -V_n(x^{(n)})+2\theta\eta\sum\limits_{i=1}^{\infty}v(i),\label{invarianttight1}
        \end{align}
	    where $M$ and $N$ are constants in Assumption \ref{A2}. Let $C=2\theta\eta\sum\limits_{i=1}^{\infty}v(i)$ for simplicity of notation. From Theorem 3.37 of Liggett \cite{liggett2010continuoustimemarkov}, we have for the invariant measure $\mu_n$ of the finite-dimensional processes,
		$$
		\mu_n\left(\mathscr{L}_n V_n\right)=0.
		$$	
		Clearly
		$$
		\mu_n\left(\mathscr{L}_n V_n\right)=\mu_n\left(\mathscr{L}_n V_n \mathbf{1}_{\mathscr{L}_n V_n>0}\right)+\mu_n\left(\mathscr{L}_n V_n \mathbf{1}_{\mathscr{L}_n V_n \leq 0}\right),
		$$
		and  it follows from (\ref{invarianttight1}) that $\mu_n\left(\mathscr{L}_n V_n \mathbf{1}_{\mathcal{L}_n V_n>0}\right) \leq C$, so 
		\begin{equation}
			\label{contradiction}
			\mu_n\left(\mathscr{L}_n V_n \mathbf{1}_{\mathscr{L}_n V_n \leq 0}\right) \geq-C.
		\end{equation}	
		For a given $\epsilon>0$, we define the set $K_\epsilon$ as
		$$
		K_\epsilon=\left\{x \in \ell_\rho^{2\theta}(\mathbb{Z}^+);\forall i \in \mathbb{Z}^+,\left|x_i\right|^{2\theta}  \leq \left(\frac{C+1}{ \epsilon v(i)}+\frac{C}{ v(i)}\right)\right\} .
		$$	
		We can verify that $K_\epsilon$ is compact in $\ell^{2\theta}_{\rho}(\mathbb{Z}^+)$. Indeed, Assumption \ref{invariantmeasuretightness} (ii) yields: (1) for any $x\in K_\epsilon$, we have 
        \begin{equation*}
            \Vert x\Vert_{2\theta,\rho}^{2\theta}=\sum\limits_{i=1}^{\infty}\vert x_i\vert^{2\theta}\rho(i)\leq \sum\limits_{i=1}^{\infty}\left(\frac{C+1}{ \epsilon v(i)}+\frac{C}{ v(i)}\right)\rho(i),
        \end{equation*}
        hence $K_\epsilon$ is bounded; (2) for any $\epsilon_1>0$, there exists $N\in\mathbb{N}$ such that for $n_0>N$, 
        \begin{equation*}
            \sum\limits_{i=n_0}^{\infty}\vert x_i\vert^{2\theta}\rho(i)\leq \sum\limits_{i=n_0}^{\infty}\left(\frac{C+1}{ \epsilon v(i)}+\frac{C}{ v(i)}\right)\rho(i)\leq \epsilon_1.
        \end{equation*}
        Moreover, $K_\epsilon$ is clearly closed, and thus Lemma \ref{Precompactcondition} establishes its compactness. 
        
        For the particular set $K_\epsilon$ under consideration, we have
        \begin{align*}
            \nu_n\left(K_\epsilon^c\right)&=\mu_n\left(\chi^{-1}_n\left(K_\epsilon^c\right)\right)=\mu_n\left(x^{(n)}\in\mathbb{R}^{\Lambda_n};(x^{(n)},0_{i\in\mathbb{Z}^+\backslash\Lambda_n})\in K_\epsilon^c\right)\\
            &=\mu_n\left(x^{(n)} \in\mathbb{R}^{\Lambda_n}; \exists i \in \Lambda_n,|x^{(n)}_i|^{2\theta}>\frac{C+1}{ \epsilon v(i)}+\frac{C}{ v(i)}\right).
        \end{align*}
		Hence, for $x^{(n)} \in \chi^{-1}_n\left(K_\epsilon^c\right)=\left\{x^{(n)} \in\mathbb{R}^{\Lambda_n}; \exists i \in \Lambda_n,|x^{(n)}_i|^{2\theta}>\frac{C+1}{ \epsilon v(i)}+\frac{C}{ v(i)}\right\}$, we have
		$$
		\mathscr{L}_n V_n \left(x^{(n)}\right) \leq-\left(\frac{C+1}{ \epsilon v(i)}+\frac{C}{ v(i)}\right) v(i)+C \leq-\frac{C+1}{\epsilon} .
		$$	
		Therefore, if $\nu_n\left(K_\epsilon^c\right)>\epsilon$ hold, we would get a contradiction to (\ref{contradiction}).
	\end{proof}
    Following arguments similar to those in the proof of Lemma \ref{corollarypequalq}, we obtain without a proof a determining class of probability measures on the measurable space $(\ell^{2\theta}_\rho(\mathbb{Z}^+),\mathscr{B}(\ell^{2\theta}_\rho(\mathbb{Z}^+)))$, which we state as the following lemma.
    \begin{lemma}\label{determingclassell2s}
        Let $P$ and $Q$ be two probability measures on $(\ell^{2\theta}_\rho(\mathbb{Z}^+),\mathscr{B}(\ell^{2\theta}_\rho(\mathbb{Z}^+)))$. Then $P=Q$ if and only if $Ph=Qh$ for every $h\in C_{b,  lip }^{C y l}(\ell_\rho^{2\theta}(\mathbb{Z}^+))$. 
    \end{lemma}
    Applying Lemmas \ref{tightnessinvariantmeasure} and \ref{determingclassell2s}, we can now establish the following main theorem of this section concerning the existence of invariant measures.
	\begin{theorem}\label{invariantmessureforinfinitsystem}
		There exists an invariant measure for the Markov transition 
        semi-group $\mathcal{P}_t$.
	\end{theorem}
	\begin{proof}
    From Lemma \ref{tightnessinvariantmeasure}, we know that the sequence of measures $\{\nu_n\}$ is tight. Consequently, there exists a weakly convergent subsequence, which we still denote by $\{\nu_n\}$ without loss of generality, such that $\nu_n\stackrel{\text{w}}{\rightarrow} \nu$.  To show that $\nu$ is invariant, by Lemma \ref{determingclassell2s} it suffices to show that for any $h \in C_{b,  lip }^{C y l}(\ell_\rho^{2\theta}(\mathbb{Z}^+))$,
		\begin{equation}
			\label{invariantmeasurecondition}
			\int_{\ell_\rho^{2\theta}(\mathbb{Z}^+)} \mathcal{P}_t h(x) d \nu(x)=\int_{\ell_\rho^{2\theta}(\mathbb{Z}^+)} h(x) d \nu(x).
		\end{equation}	
		 Recall that $\nu_n=\mu_n \circ \chi^{-1}_n$ and that $\mu_n$ is the invariant measure for process $X^{(n)}$ on $\mathbb{R}^{\Lambda_n}$, so that the equality
		\begin{equation}
		    \label{finitedimensioninvariantequality}
			\int_{\mathbb{R}^{\Lambda_n}} E_n^{x_n} h\circ\chi_n\left(\omega_n(t)\right) d \mu_n\left(x_n\right)=\int_{\mathbb{R}^{\Lambda_n}} h\circ\chi_n\left(x_n\right) d \mu_n\left(x_n\right), \quad \forall h \in C_b\left(\mathbb{R}^{\Lambda_n}\right)
		\end{equation}
		holds. Therefore, we have
		\begin{align}
			\int_{\ell_\rho^{2\theta}(\mathbb{Z}^+)} h(x) d \nu(x)&=\lim _n \int_{\ell_\rho^{2\theta}(\mathbb{Z}^+)} h(x) d \nu_n(x)=\lim _n \int_{\mathbb{R}^{\Lambda_n}}\left(h \circ \chi_n\right)\left(x_n\right) d \mu_n\left(x_n\right) \nonumber\\
			& \overset{(\ref{finitedimensioninvariantequality})}{=}\lim _n \int_{\mathbb{R}^{\Lambda_n}} E_n^{x_n}[\left(h \circ \chi_n\right)\left(\omega_n(t)\right)] d \mu_n\left(x_n\right)\nonumber\\
            &=\lim _n \int_{\ell_\rho^{2\theta}(\mathbb{Z}^+)} \tilde{E}_n^x[\left(h \circ \chi_n\right)\left(\omega_n(t)\right)] d \nu_n(x) \nonumber\\
			& =\lim _n \int_{\ell_\rho^{2\theta}(\mathbb{Z}^+)} \tilde{E}_n^x h\left(\omega(t)\right) d \nu_n(x).\label{limiteqaulity1}
		\end{align}
		 By the Skorohod representation theorem (see Da Prato and Zabczyk \cite[Theorem 2.4]{stochasticequationsininfinitedimensionsdaza2014}), since $\nu_n\stackrel{\text{w}}{\rightarrow} \nu$, we can obtain that on some probability space $(\bar{\Omega},\bar{\mathscr{F}},\bar{\mathbb{P}})$, there exist $\ell_\rho^{2\theta}(\mathbb{Z}^+)$-valued random variables $Z_n$ and $Z$ such that
        \begin{itemize}
            \item[(i)] $Z_n\rightarrow Z$, in $\ell_\rho^{2\theta}(\mathbb{Z}^+)$, $\bar{\mathbb{P}}$-a.s.,
            \item[(ii)] $\nu_{n}=\bar{\mathbb{P}}\circ Z_n^{-1}$, $\nu=\bar{\mathbb{P}}\circ Z^{-1}$.
        \end{itemize}
        From the proof of Theorem \ref{Markov property}, we know $\tilde{E}_n^{Z_n} h\left(\omega(t)\right)\rightarrow \mathbf{E}^{Z} h\left(X(t)\right)$, $\bar{\mathbb{P}}$-a.s., hence the dominated convergence theorem gives
        \begin{align}
            \lim\limits_n \int_{\ell_\rho^{2\theta}(\mathbb{Z}^+)} \tilde{E}_n^x h\left(\omega(t)\right) d \nu_n(x)&=\lim\limits_n \int_{\bar{\Omega}} \tilde{E}_n^{Z_n(\bar{\omega})} h\left(\omega(t)\right) d \bar{\mathbb{P}}_n(\bar{\omega})=\int_{\bar{\Omega}} \mathbf{E}^{Z(\bar{\omega})} h\left(\omega(t)\right) d \bar{\mathbb{P}}(\bar{\omega})\nonumber\\
            &=\int_{\ell_\rho^{2\theta}(\mathbb{Z}^+)} E^x h\left(X(t)\right) \nu(x)=\int_{\ell_\rho^{2\theta}(\mathbb{Z}^+)} \mathcal{P}_t h(x) d \nu(x).\label{limiteqaulity2}
        \end{align}
        A combination of (\ref{limiteqaulity1}) with (\ref{limiteqaulity2}) yields (\ref{invariantmeasurecondition}), thus establishing the theorem.
	\end{proof}
	
	%%%%%%%%%%%%%%%%%%%%%%
	%%%%%%%%%%%%%%%%%%%%%%%%%%%%%%%%%%%%%%%
	%%%%%%%%%%%%%%%%%%%%%%%%%%%%%%%%%%%%%%%%%%%%%%%%%%%%%%%%%%%
    \section{Time-inhomogeneous systems}
    In this section, we address the case where the coefficients of the equation depend on time in a periodic manner. Under such circumstances, the resulting solutions become time-inhomogeneous, rendering traditional methods based on invariant measures inadequate for studying long-time behaviour. To characterize the asymptotic dynamics of the system, we thus employ the concept of periodic measures (Feng and Zhao \cite{fengzhao2020randomperiodic}). This necessitates an investigation into the existence of periodic measures for the equation.
    \subsection{Time-inhomogeneous finite-dimensional systems: preliminary result}
    In Section 2, we investigated the existence and uniqueness of invariant measures for time-homogeneous finite-dimensional SDEs. The current section focusses on time-inhomogeneous finite-dimensional SDEs with time-periodic coefficients, examining the existence and uniqueness of periodic measures. Let us first introduce some notation. Given a locally compact separable metric space $(E, \mathscr{B}(E))$, consider a time-inhomogeneous SDE with Markovian solution $\{X^{s,x}(t); t\geq s\}$ starting at time $s \geq  0$ from an initial position $x \in E$. We define
    \begin{description}
        \item[(i)] The two-parameter transition semi-group $\{\mathcal{P}(s,t);0\leq s \leq t\}$ acting on $B_b(E)$ by
        \begin{equation*}
            \mathcal{P}(s,t)\phi(x)=\mathbb{E}[\phi(X^{s,x}(t))],\quad 0\leq s \leq t,x\in E,\phi \in B_b(E);
        \end{equation*}
        \item[(ii)]  The transition probability kernel $\mathcal{P}(s, t, x, dy)$ on $E$ satisfying
        \begin{equation*}
            \mathcal{P}(s,t,x,\Gamma)=\mathcal{P}(s,t)\mathbf{1}_{\Gamma}(x),\quad 0\leq s \leq t,x\in E,\Gamma\in \mathscr{B}(E);
        \end{equation*}
        \item[(iii)] The adjoint semi-group $\{\mathcal{P}^*(s,t);0\leq s \leq t\}$ acting on $\mathscr{P}(E)$, the space of all probability measures on $E$, by
        \begin{equation*}
            \mathcal{P}^*(s,t)\mu(\Gamma)=\int_E \mathcal{P}(s,t,x,\Gamma)\mu(dx),\quad 0\leq s \leq t, \Gamma\in \mathscr{B}(E).
        \end{equation*}
    \end{description} 
    The concept of periodic measures, originally defined by Feng and Zhao  \cite{fengzhao2020randomperiodic}, is formally stated below.
    \begin{definition}
        Let $(E,\mathscr{B}(E))$ be a locally compact separable metric space. A measure-valued function $\mu:\mathbb{T}\rightarrow\mathscr{P}(E)$ is called a T-periodic measure with respect to the transition semi-group $\mathcal{P}(\cdot,\cdot)$ if for all $0\leq s\leq t$
        \begin{equation*}
            \mu_{s+T}=\mu_s,\quad \mu_{t}=\mathcal{P}^*(s,t)\mu_s.
        \end{equation*}
    \end{definition}
    We investigate the following $\mathbb{R}^n$-valued SDEs
    \begin{equation}
		\label{periodicFSDE}
		\begin{cases}
			dX^{(n)}(t)=\left(A_n(t)X^{(n)}(t)+F_n(t,X^{(n)}(t))\right)dt+B_ndW^{(n)}(t) ,  \quad t\geq s,\\
			X^{(n)}(s)=x^{(n)} \in \mathbb{R}^n,
		\end{cases}
	\end{equation}
	where $W^{(n)}(t)$ is an $n$-dimensional Wiener process on a probability space $(\Omega,\mathcal{F},\mathbb{P})$ with identity covariance and $B_n$ is given in (\ref{AnFnBn}). The time-dependent operator $A_n$ and nonlinear function $F_n$ are given by 
	\begin{equation}
		\label{periodicAnFnBn}
			A_n(t)=\begin{pmatrix}
				0 & a_{1,2}(t) & & & & \\
				a_{2,1}(t) & 0 & a_{2,3}(t)& & & \\
				& \ddots &\ddots& \ddots& \\
				& & \ddots& \ddots& a_{n-1,n}(t) \\
				& & & a_{n,n-1}(t)&0&\\
		
			\end{pmatrix},
			F_n(t,x^{(n)})=\begin{pmatrix}
				f(t,x^{(n)}_1)\\
				f(t,x^{(n)}_2)\\
				\vdots \\
				f(t,x^{(n)}_n)\\
			\end{pmatrix}, 
	\end{equation}
	where $f$ is a function that maps $\mathbb{R}^+\times\mathbb{R}$ to $\mathbb{R}$.

    After appropriate modifications to our assumptions, the same methodology yields the existence and uniqueness of solutions. We state our conclusions without proof below.
    \begin{condition}
		\label{periodicA1}
		\begin{description}
        \item[(i)] The coefficients $\{a_{i,i-1}(t),a_{i-1,i}(t);2\leq i\leq n\}$ and $f(t,z)$ are T-periodic in $t$ for any fixed $z\in\mathbb{R}$.
			\item [(ii)] Matrix $A_n$ has elements satisfying: $a_{i,i-1}(t)\neq 0,2\leq i\leq n,t\geq 0$ and there exists a constant $M>0$ such that
			$$
			\max \left\{\vert a_{i-1,i}(t)\vert ,\vert a_{i,i-1}(t) \vert; 2\leq i\leq n \right\}\leq M,\quad \forall t\geq 0. 
			$$
			\item [(iii)] The function $f\in C^\infty (\mathbb{R}^+\times\mathbb{R} , \mathbb{R})$ is weakly dissipative, i.e. there exist constants $\eta >0$ and
			$2M< \lambda-\frac{1}{2}$ such that
			$$
			f(t,z)z \leq \eta -\lambda \vert z \vert^2 ,\quad \forall(t,z)\in \mathbb{R}^+\times\mathbb{R}.
			$$
		\end{description}
	\end{condition} 
     Based on the existence and uniqueness of strong solutions $X^{(n),s,x}$ with initial time $s$ and initial value $x$, and by exploiting the periodicity of the equation coefficients, we conclude that the transition probability kernel $\mathcal{P}_n(s,t,x,\Gamma)=\mathbb{E}[X^{(n),s,x}(t)\in \Gamma]$ is periodic (see Feng et al.  \cite[Definition 2.1]{fengzhaozhong2023existence}). This leads to the following theorem. 
     \begin{theorem}
         \label{periodicmeasureexistence}
         Under Assumption \ref{periodicA1}, there exists a unique T-periodic measure $\mu_{n,\cdot}$.
     \end{theorem}
     \begin{proof}
         The proof follows similar arguments to those in Section 2, so we outline only the key steps while omitting technical details:
     \begin{itemize}
         \item [1.]Following analogous reasoning to Section 2 establishes both the irreducibility of $\mathcal{P}_n(s,t,x,\cdot)$ and the validity of the H\"ormander condition;
         \item[2.] By Theorem 3.4 of \cite{fengzhaozhong2023existence}, the local Doeblin condition holds;
         \item[3.] Verification of the Lyapunov function confirms that all conditions of Theorem 3.6 in \cite{fengzhaozhong2023existence} are satisfied.
     \end{itemize}
     The conclusion then follows.
     \end{proof}

     First, adopting estimation methods analogous to Lemma 2.5 of Feng et al. \cite{fengquzhao2021randomquasiperiodic}, we derive distance estimates for finite-dimensional systems over infinite time horizons.
    \begin{lemma}\label{infinithrizonestimate}
        For any given initial condition $x^{(n)}\in\mathbb{R}^n$, let $X^{(n),s,x^{(n)}}$ denote the solution to equation (\ref{periodicFSDE}) with the initial value $X^{(n),s,x^{(n)}}(s)=x^{(n)}$. Under Assumption \ref{periodicA2}, for any $p\geq 2$,  there exists a constant $C=C(p,\eta,\lambda,M)$ such that for all $t>s$, the following estimate holds:
        $$\mathbb{E}\Vert \tilde{X}^{(n),s,x^{(n)}}(t)\Vert^{p}_{\mathbb{R}^n}\leq C(1+\Vert x^{(n)}\Vert^{p}_{\mathbb{R}^n}).$$
    \end{lemma}
    \begin{proof}
        While the case $p = 2$ admits a similar proof, we shall restrict our attention to the situation where $p > 2$. For any fixed $\beta$ and $p> 2$, applying It$\hat{\text{o}}$'s formula to $e^{\beta t}\Vert X^{(n),s,x^{(n)}}(t)\Vert^p_{\mathbb{R}^n}$ and invoking the conditions in Assumption \ref{periodicA2}, we obtain
        \begin{align*}
            &e^{\beta t}\Vert X^{(n),s,x^{(n)}}(t)\Vert^p_{\mathbb{R}^n}\\
            &\leq e^{\beta s}\Vert x^{(n)}\Vert_{\mathbb{R}^n}^p+\int_s^te^{\beta r}\Vert X^{(n),s,x^{(n)}}(r)\Vert_{\mathbb{R}^n}^{p-2}\left(\beta \Vert X^{(n),s,x^{(n)}}(r)\Vert^2_{\mathbb{R}^n}\right.\\
            &\left.\quad +pX^{(n),s,x^{(n)}}(r)\cdot \left(A_n(r)X^{(n),s,x^{(n)}}(r) +F_n(r,X^{(n),s,x^{(n)}}(r))\right)+\frac{p(p-1)}{2}tr(B_nB_n^*)\right)dr\\
            &\quad +\int_s^t pe^{\beta r}\Vert X^{(n),s,x^{(n)}}(r)\Vert^{p-2}X^{(n),s,x^{(n)}}(r)B_ndW^{(n)}(r)\\
            &\leq e^{\beta s}\Vert x^{(n)}\Vert_{\mathbb{R}^n}^p+\int_s^te^{\beta r}\Vert X^{(n),s,x^{(n)}}(r)\Vert_{\mathbb{R}^n}^{p-2}\left((\beta+p(2M-\lambda) \Vert X^{(n),s,x^{(n)}}(r)\Vert^2_{\mathbb{R}^n}+p\eta \right.\\
            &\left.\quad +\frac{p(p-1)}{2}\right)dr+\int_s^t pe^{\beta r}\Vert X^{(n),s,x^{(n)}}(r)\Vert^{p-2}X^{(n),s,x^{(n)}}(r)B_ndW^{(n)}(r).
        \end{align*}
        Since the condition $\lambda>2M+\frac{1}{2}$ holds, we first set $2\epsilon:=p(\lambda-2M)>0$ and choose $\beta=p(\lambda-2M)-\epsilon>0$, yielding 
        \begin{align*}
            &e^{\beta t}\Vert X^{(n),s,x^{(n)}}(t)\Vert^p_{\mathbb{R}^n}\\
            &\leq e^{\beta s}\Vert x^{(n)}\Vert_{\mathbb{R}^n}^p+\int_s^te^{\beta r}\left(-\epsilon\Vert X^{(n),s,x^{(n)}}(r)\Vert_{\mathbb{R}^n}^{p}+C(p,\eta)\Vert X^{(n),s,x^{(n)}}(r)\Vert_{\mathbb{R}^n}^{p-2}\right)dr\\
            & \quad +\int_s^t pe^{\beta r}\Vert X^{(n),s,x^{(n)}}(r)\Vert^{p-2}X^{(n),s,x^{(n)}}(r)B_ndW^{(n)}(r),
        \end{align*}
        where $C(p,\eta)=p\eta+\frac{p(p-1)}{2}$. By an application of Young's inequality $ab\leq \frac{\epsilon^{\prime}a^{p^\prime}}{p^\prime}+\frac{b^{p^\prime}}{q^\prime(\epsilon^{\prime})^{q^\prime/p^\prime}}$, $\frac{1}{p^\prime}+\frac{1}{q^\prime}=1$, we establish
        \begin{align*}
            C(p,\eta)\Vert X^{(n),s,x^{(n)}}(r)\Vert_{\mathbb{R}^n}^{p-2}&\leq \epsilon \Vert X^{(n),s,x^{(n)}}(r)\Vert_{\mathbb{R}^n}^{p}+\frac{2}{p}C^{\frac{p}{2}}(p.\eta)(\frac{\epsilon p}{p-2})^{-\frac{p-2}{2}}\\
            &=\epsilon \Vert X^{(n),s,x^{(n)}}(r)\Vert_{\mathbb{R}^n}^{p}+C(p,\eta,\lambda,M).
        \end{align*}
        This immediately gives
        \begin{align*}
            &e^{\beta t}\Vert X^{(n),s,x^{(n)}}(t)\Vert^p_{\mathbb{R}^n}\\
            &\leq e^{\beta s}\Vert x^{(n)}\Vert_{\mathbb{R}^n}^p+C(p,\eta,\lambda,M)\int_s^te^{\beta r}dr +\int_s^t pe^{\beta r}\Vert X^{(n),s,x^{(n)}}(r)\Vert^{p-2}X^{(n),s,x^{(n)}}(r)B_ndW^{(n)}(r)\\
            &\leq e^{\beta s}\Vert x^{(n)}\Vert_{\mathbb{R}^n}^p+C(p,\eta,\lambda,M)e^{\beta t} +\int_s^t pe^{\beta r}\Vert X^{(n),s,x^{(n)}}(r)\Vert^{p-2}X^{(n),s,x^{(n)}}(r)B_ndW^{(n)}(r),
        \end{align*}
        where $C(p,\eta,\lambda,M)$ denotes a generic positive constant whose value may change from line to line. Taking expectations on both sides, we obtain
        \begin{equation*}
            e^{\beta t}\mathbb{E}\Vert X^{(n),s,x^{(n)}}(t)\Vert^p_{\mathbb{R}^n}\leq e^{\beta s}\Vert x^{(n)}\Vert_{\mathbb{R}^n}^p+C(p,\eta,\lambda,M)e^{\beta t}.
        \end{equation*}
        Then
        \begin{equation*}
            \mathbb{E}\Vert X^{(n),s,x^{(n)}}(t)\Vert^p_{\mathbb{R}^n}\leq \Vert x^{(n)}\Vert_{\mathbb{R}^n}^p+C(p,\eta,\lambda,M),
        \end{equation*}
        thus proving the lemma.
    \end{proof}
     The finite-dimensional periodic measures are proved to be continuous in time under the weak topology, as precisely formulated in the following proposition.
    \begin{proposition}
        For any bounded cylindrically Lipschitz function $h\in C_{b,Lip}(\mathbb{R}^n)$ and arbitrary $s\geq 0$, there holds
         \begin{equation*}
             \lim\limits_{t\rightarrow0} \mu_{n,s+t}h=\mu_{n,s}h.
         \end{equation*}
    \end{proposition}
    \begin{proof}
        Theorem \ref{periodicmeasureexistence} guarantees that the transition probability kernel $\mathcal{P}_n(s,t)$ of equation (\ref{periodicFSDE}) admits a unique periodic measure $\mu_{n,\cdot}$. By Theorem 3.6 of Feng, Zhao and Zhong  \cite{fengzhaozhong2023existence}, $\mathcal{P}_n(s,s+mT,0,\cdot)$ converges weakly to $\mu_{n,s}$ as $m\rightarrow \infty$. Moreover, since $\mathcal{P}_n(s,t)$ is strong Feller, it follows that for any bounded and  Lipschitz continuous function $h\in C_{b,Lip}(\mathbb{R}^n)$ and any $s \geq 0$, $t > 0$, there exists a constant $C$ such that
         \begin{align*}
             &\vert \mu_{n,s+t}h-\mu_{n,s}h \vert^2\\
             &=\left\vert \int_{\mathbb{R}^n}h(x^{(n)})\mu_{n,s+t}(dx^{(n)})-\int_{\mathbb{R}^n}h(x^{(n)})\mu_{n,s}(dx^{(n)})\right\vert^2\\
             &=\left\vert \int_{\mathbb{R}^n}h(x^{(n)})\mathcal{P}^*(s,s+t)\mu_{n,s}(dx^{(n)})-\int_{\mathbb{R}^n}h(x^{(n)})\mu_{n,s}(dx^{(n)})\right\vert^2\\
             &=\left\vert \int_{\mathbb{R}^n} \int_{\mathbb{R}^n} h(y^{(n)})\mathcal{P}(s,s+t,x^{(n)},dy^{(n)})\mu_{n,s}(dx^{(n)})-\int_{\mathbb{R}^n}h(x^{(n)})\mu_{n,s}(dx^{(n)})\right\vert^2\\
             &=\left\vert \int_{\mathbb{R}^n} \left[\int_{\mathbb{R}^n} h(y^{(n)})\mathcal{P}(s,s+t,x^{(n)},dy^{(n)})-h(x^{(n)})\right]\mu_{n,s}(dx^{(n)})\right\vert^2\\
             &= \lim\limits_{m\rightarrow\infty}\left\vert \int_{\mathbb{R}^n} \left[\int_{\mathbb{R}^n} h(y^{(n)})\mathcal{P}(s,s+t,x^{(n)},dy^{(n)})-h(x^{(n)})\right]\mathcal{P}_n(s,s+mT,0,dx^{(n)})\right\vert^2\\
             &\leq  \varliminf\limits_{m\rightarrow\infty}\int_{\mathbb{R}^n} \mathbb{E}\left[ \left\vert h(X^{(n),s,x^{(n)}}(s+t))-h(x^{(n)})\right\vert^2\right] \mathcal{P}_n(s,s+mT,0,dx^{(n)})\\
             &\leq  \varliminf\limits_{m\rightarrow\infty}\int_{\mathbb{R}^n} C\mathbb{E}\left[ \left\Vert X^{(n),s,x^{(n)}}(s+t)-x^{(n)}\right\Vert^2_{\mathbb{R}^n}\right]\mathcal{P}_n(s,s+mT,0,dx^{(n)})
         \end{align*}
         Following an analogous computation to that for (\ref{ineq1}), we obtain the estimate
         \begin{equation}\label{periodicineq1}
             \forall 0\leq s \leq t\leq T_1: \mathbb{E}\Vert X^{(n),s,x^{(n)}}(t)-x^{(n)} \Vert_{\mathbb{R}^n}^2 \leq C(T_1)(1+ \Vert x^{(n)} \Vert_{\mathbb{R}^n}^{2s})\vert t-s \vert.
         \end{equation}
         Then, under the condition $s + t < T_1$ for some fixed $T_1 > 0$, the combined application of (\ref{periodicineq1}) and Lemma \ref{infinithrizonestimate} yields 
         \begin{align*}
             \vert \mu_{n,s+t}h-\mu_{n,s}h \vert^2&\leq \varliminf\limits_{m\rightarrow\infty}\int_{\mathbb{R}^n} C(1+ \Vert x^{(n)} \Vert_{\mathbb{R}^n}^{2\theta})\vert t-s \vert \mathcal{P}_n(s,s+mT,0,dx^{(n)})\\
             &=\varliminf\limits_{m\rightarrow\infty}C\vert t-s \vert \mathbb{E}(1+\Vert X^{(n),s,0}(s+mT)\Vert^{2\theta}_{\mathbb{R}^n})\\
             &\leq C \vert t-s \vert,
         \end{align*}
         which proves the right-continuity of $\mu_{n,\cdot}$. The left-continuity can be verified through analogous arguments by making appropriate modifications to the preceding analysis.
    \end{proof}

    %%%%%%%%%%%%%%%%%%%%%%%%%
    \subsection{Time-inhomogeneous infinite-dimensional systems: periodic measures}
	In this section, we consider infinite-dimensional equations with time-periodic coefficients. Under modified assumptions, arguments analogous to those in Section 3 yield the existence and uniqueness of weak solutions. Furthermore, we establish the periodicity of transition semi-group, from which we deduce the existence of periodic measures.   The equations considered are:
    \begin{equation}
		\label{periodicINFSDE2}
		\begin{cases}
			dX(t)=\left(A(t)X(t)+F(t,X(t))\right)dt+BdW(t) ,  \ t\geq s,\\
			X(s)=x \in \ell_\rho^{2s}(\mathbb{Z}^+),
		\end{cases}
	\end{equation}
	where $W(t)$ is a Wiener process on a probability space $(\Omega,\mathcal{F},\mathbb{P})$ taking values in $\ell^2 (\mathbb{Z}^+)$ and with identity covariance, $B$ is defined as in (\ref{AFBdef}) and the operators $A(t)$, $F(t,x)$ are given by the following formula 
	$$
	\begin{cases}
		A(t)x=\left(a_{\gamma,\gamma-1}(t)x_{\gamma-1}+a_{\gamma,\gamma+1}(t)x_{\gamma+1}\right)_\gamma,x=\left(x_\gamma\right)_\gamma,\\
		F(t,x)=\left(f(t,x_\gamma)\right)_\gamma,x=\left(x_\gamma\right)_\gamma.
	\end{cases}
	$$ 
    We impose the following assumptions on the coefficients of (\ref{periodicINFSDE2}) and the solution space:
    \begin{condition}
		\label{periodicA2}
		\begin{description}
			\item [(i)]  The coefficients $\{a_{i,i-1}(t),a_{i-1,i}(t); i\geq 2 \}$ and $f(t,z)$ are T-periodic in $t$ for any fixed $z\in\mathbb{R}$.
            \item[(ii)]  All the elements of the operator $A(t)$ satisfy:  $a_{i,i-1}(t)\neq  0,i\geq 2,t\geq 0$ and there is a constant $M>0$ such that
			$$
			\max \left\{\vert a_{i-1,i}(t)\vert ,\vert a_{i,i-1}(t) \vert\right\}\leq M,\quad    i\geq 2,\quad t\geq 0.
			$$
			\item[(iii)] There exist constants $R,N>0$ such that
			$$
			\left\vert \frac{\rho(\gamma)}{\rho(j)} \right\vert\leq N \ \text{if} \ \vert\gamma-j\vert \leq R,\quad \sum\limits_{i=1}^\infty \rho(i)< +\infty,\quad \rho(i)>0. 
			$$
			\item [(iii)] The function $f\in C^\infty (\mathbb{R}^+\times\mathbb{R} , \mathbb{R})$ is weakly dissipative, i.e. there exists constants $\eta >0$ and
			$$
			2M< \lambda-\frac{1}{2}
			$$ 
			such that
			$$
			  f(t,z)z \leq \eta -\lambda \vert z \vert^2 ,\quad (t,z)\in\mathbb{R}^+\times \mathbb{R}.
			$$
			Moreover there exists some constants $\theta\geq 1$, and $\eta_0\ge 0$ such that
			\begin{equation*}
				\left|f(t,z)\right|\leq \eta_0\left(1+\left|z\right|^\theta \right), \quad 
                (t,z)\in\mathbb{R}^+ \times\mathbb{R}.
			\end{equation*}
		\end{description}
	\end{condition} 
While Assumption \ref{periodicA2} guarantees the existence of weak solutions, the following supplementary conditions ensure the uniqueness.
   \begin{condition}
		\label{periodicuniquenesscondition}
		\begin{description}
			\item[(i)] There exist  constants $ r,M_1,M_2>0$ and $N_1\in \mathbb{Z}^+$ such that for any $n\geq N_1$
			$$
			\rho(n)\geq \frac{M_1}{e^{rn}}, \quad \limsup\limits_{n\rightarrow\infty}\frac{\rho(n)}{\inf\limits_{1\leq i\leq n}\rho(i)}\leq M_2.
			$$ 			
			\item[(ii)] There exists a function $L:\mathbb{N}\rightarrow\mathbb{R}^+$ such that for all $N>0$, the function $f$ satisfies 
            \begin{equation*}
				|f(t,z)-f(t,\tilde{z})|\leq L(N)|z-\tilde{z}|, -N\leq z,\tilde{z}\leq N,t\geq0.
			\end{equation*}	
             where $L$ further satisfies: for any constant $C>0$, there exists $\gamma=\gamma(C)>0$ such that 
             \begin{equation}
                \lim\limits_{n\rightarrow \infty}\frac{\exp\{CL^2(n)\}}{n^\gamma L(n)}=0.
            \end{equation}
		\end{description}
	\end{condition}
    Under Assumptions \ref{periodicA2} and \ref{periodicuniquenesscondition}, we obtain the existence and uniqueness of weak solutions. Following arguments similar to those in Section 3.5, one can verify that the solutions possess the Markov property. This allows us to define the transition probability kernel $\mathcal{P}(s,t,x,\Gamma)=\mathbf{P}^{s,x}[X^s(t)\in\Gamma]$, where $X^s(\cdot)$ is the modified canonical process on $\Omega_{Spin}^s=C([s,\infty),\ell^2_{\rho}(\mathbb{Z}^+))$, $\Gamma\in \mathscr{B}(\ell^{2\theta}_{\rho}(\mathbb{Z}^+))$ and $\mathbf{P}^{s,x}$ is the solution to the martingale problem associated with (\ref{periodicINFSDE2}) with the initial condition $x\in\ell^{2\theta}_{\rho}(\mathbb{Z}^+)$. The following lemma establishes the periodicity of $\mathcal{P}(s,t)$.
    \begin{lemma}
        Assuming the validity of Assumptions \ref{periodicA2} and \ref{periodicuniquenesscondition}, the transition probability kernel $\mathcal{P}(\cdot,\cdot,\cdot,\cdot)$ possesses T-periodicity, namely $\mathcal{P}(s,t,x,\cdot)=\mathcal{P}(s+T,t+T,x,\cdot)$, for all $x\in\ell^{2\theta}_{\rho}(\mathbb{Z}^+),s\leq t$.
    \end{lemma}
    \begin{proof}
        In order to prove the result of this lemma, application of Lemma \ref{determingclassell2s} suggests we only need to verify that for any $h \in C_{b,  lip }^{C y l}(\ell_\rho^{2\theta}(\mathbb{Z}^+))$,
		\begin{equation}
			\label{tperiodictransitionkernel}
			\int_{\ell_\rho^{2\theta}(\mathbb{Z}^+)} h(a)\mathcal{P}(s,t,x,da)=\int_{\ell_\rho^{2\theta}(\mathbb{Z}^+)} h(a) \mathcal{P}(s+T,t+T,x,da).
		\end{equation}	
        From the weak convergence property combined with the periodicity of transition probability kernels for finite-dimensional SDEs, we derive
        \begin{align*}
            \int_{\ell_\rho^{2\theta}(\mathbb{Z}^+)} h(a)\mathcal{P}(s,t,x,da)&=\int_{\ell_\rho^{2\theta}(\mathbb{Z}^+)} h(a) \mathbf{P}^{s,x}\circ X^s(t)=\int_{\Omega^s_{Spin}} h(X^s(t)) d\mathbf{P}^{s,x}\\
            &=\lim\limits_{n} \int_{\Omega^s_{Spin}}h(\omega(t))d\tilde{P}^{s,x}_n=\lim\limits_{n} \int_{\Omega}h(\tilde{X}^{(n),s,x}(t))d\mathbb{P}\\
            &=\lim\limits_{n} \int_{\Omega}h(\tilde{X}^{(n),s+T,x}(t+T))d\mathbb{P}=\lim\limits_{n} \int_{\Omega^s_{Spin}}h(\omega(t+T))d\tilde{P}^{s+T,x}_n\\
            &=\int_{\Omega^s_{Spin}} h(\omega(t+T)) d\mathbf{P}^{s+T,x}=\int_{\ell_\rho^{2\theta}(\mathbb{Z}^+)} h(a) \mathcal{P}(s+T,t+T,x,da),
        \end{align*}
        which yields (\ref{tperiodictransitionkernel}).
    \end{proof}
    
    Before establishing the existence of periodic measures for $\mathcal{P}(s,t)$, we present an auxiliary lemma. 
    \begin{lemma}\label{inttighttotntight}
    Let $(E,\mathscr{B}(E))$ be a complete separable metric space. Given a family $\{\mu_{n,s}:0\leq s\leq T\}_n$ of the $\mathscr{P}(E)$-valued functions that are measurable with respect to $s$. If a family of average measures $\mu_n:=\frac{1}{T}\int_0^T \mu_{n,s} ds, n\in\mathbb{N}$, forms a tight sequence in $\mathscr{P}(E)$,  then there exists a sequence $s_n\in [0,T]$ such that $\{\mu_{n,s_n}\}_n$ is tight.
    \end{lemma}
    \begin{proof}
        Since $\{\mu_n\}_n$ is tight, for each $k\in\mathbb{N}$, there exists a compact set $K_k$ such that for all $n\in\mathbb{N}$, $\mu_n(K_k^c) \leq \frac{1}{k^4}$. For fixed $n\in\mathbb{N}$, define the sets $O_{n,k}=\{s\in[0,T];\mu_{n,s}(K_k^c)<\frac{1}{k^2}\}$. By the definition of $\mu_n:=\frac{1}{T}\int_0^T \mu_{n,s} ds$ and Fubini's theorem, we obtain
        \begin{equation*}
            \frac{1}{k^4}>\mu_n(K_k^c)=\frac{1}{T}\int_0^T \mu_{n,s}(K_k^c) ds\geq \frac{1}{Tk^2}m(O_{n,k}^c),
        \end{equation*}
        where $m$ denotes the Lebesgue measure. For each $n\in\mathbb{N}$, define $U_n=\bigcap\limits_{k=2}^{\infty}O_{n,k}$. Then we have the estimate
        \begin{equation*}
            m(U_n)=m([0,T]\backslash U_n^c)=T-m(\bigcup\limits_{k=2}^{\infty}O_{n,k}^c)
            \geq T-\sum\limits_{k=2}^{\infty}m(O_{n,k}^c)\geq T-T\sum\limits_{k=2}^{\infty}\frac{1}{k^2}>0,
        \end{equation*}
        which implies $U_n\neq \emptyset$. Choosing $s_n\in U_n$, we now prove that the sequence $\{\mu_{n,s_n}\}, n\geq 0$, is tight. In fact, for any $\epsilon>0$, select $k\in \mathbb {N}$ such that $\frac{1}{k^2}<\epsilon$. Then for the compact set $K_k$, we have for all $n\in\mathbb{N}$, $\mu_{n,s_n}(K_k^c)<\frac{1}{k^2}<\epsilon$, completing the proof.
    \end{proof} 
    
    We now investigate the existence of periodic measures for the transition semi-group $\mathcal{P}(s,t)$.
    \begin{theorem}
        Let Assumptions \ref{periodicA2}, \ref{periodicuniquenesscondition} and \ref{invariantmeasuretightness} hold. Then there exists a T-periodic measure $(\nu_{s})_{s\geq 0}$ for the transition semi-group $\mathcal{P}(s,t)$.
    \end{theorem}
    \begin{proof}
    We begin by lifting the semi-flow of the finite-dimensional SDEs (\ref{periodicFSDE}) to $\tilde{\mathbb{X}}=[0,T)\times \ell^{2\theta}_{\rho}(\mathbb{Z}^+)$:  
     \begin{equation}
		\label{liftperiodicFDSDE}
		\begin{cases}
            dZ_0^{(n)}(t)=dt \quad \text{mod}\quad T,\\
			dZ_i^{(n)}(t)=\left(A_n(Z_0^{(n)}(t))Z^{(n)}(t)+F_n(Z_0^{(n)}(t),Z^{(n)}(t))\right)dt+B_ndW^{(n)}(t) ,  \\
               t\geq 0, \quad 1\leq i\leq n, \\
			Z_0^{(n)}(0)=0,\quad (Z_i^{(n)}(0))_{1\leq i\leq n}=\Pi_nx \in \mathbb{R}^n.
		\end{cases}
	\end{equation}
    The corresponding transition probability kernel is defined by $\tilde{\mathcal{P}}_n(t,(s,x),\tilde{\Gamma}):=\mathbb{P}[(Z^{(n)}_0(t),$ $Z^{(n)}(t))\in \tilde{\Gamma}]$, $\tilde{\Gamma}\in \mathscr{B}([s,T])\times\mathscr{B}(\mathbb{R}^n)$. By adapting the proof methodology of Theorem 3.13 in Feng, Qu and Zhao \cite{fengquzhao2021randomquasiperiodic}, we derive that $\tilde{\mu}_{n,s}(\tilde{\Gamma}):=[\delta_{s}\times\mu_{n,s}](\tilde{\Gamma})$ is a periodic measure for $\tilde{\mathcal{P}}_n$. It follows from Theorem 3.6 of Feng and Zhao \cite{fengzhao2020randomperiodic} that $\bar{\tilde{\mu}}_{n}:=\frac{1}{T}\int_0^T \tilde{\mu}_{n,s} ds$ is an invariant measure for $\tilde{\mathcal{P}}_n$.

    Noting that the generator of equation (\ref{liftperiodicFDSDE}) is 
    \begin{equation*}
        \begin{split}
            \tilde{\mathscr{L}}_n\tilde{h}(z^{(n)})=&\partial_t\tilde{h}(z^{(n)})+(a_{1,2}(z^{(n)}_0)z_2+f(z^{(n)}_0,z^{(n)}_1))\partial_{z^{(n)}_1}\tilde{h}(z^{(n)})\\
            &+\cdots+(a_{n,n-1}(z^{(n)}_0)z^{(n)}_{n-1}+f(z^{(n)}_0,z^{(n)}_n))\partial_{z^{(n)}_n}\tilde{h}(z^{(n)}),\quad z^{(n)}\in [0,T]\times\mathbb{R}^n,
        \end{split}
    \end{equation*} 
    we adapt methods similar to Lemma \ref{tightnessinvariantmeasure} to show $\bar{\tilde{\nu}}_n:=\bar{\tilde{\mu}}_{n}\circ\tilde{\chi}_n^{-1}$ is tight in $[0,T]\times\ell^{2\theta}_{\rho}(\mathbb{Z}^+)$, where $\tilde{\chi}_nz^{(n)}=(z^{(n)}, 0_{i\in \mathbb{Z}^+\backslash\Lambda_n})$.   By Lemma \ref{inttighttotntight}, there exists a sequence $\{t_n\}_n$ such that $\{\tilde{\mu}_{n,t_n}\circ\tilde{\chi}_n^{-1}\}$ is tight in space $[0,T]\times\ell^{2\theta}_{\rho}(\mathbb{Z}^+)$. Consequently, the projected measures $\mu_{n,t_n}\circ\tilde{\chi}_n^{-1}(\cdot)=\tilde{\mu}_{n,t_n}\circ\tilde{\chi}_n^{-1}([0,T]\times \cdot)$ form a tight family in the space $\ell^{2\theta}_{\rho}(\mathbb{Z}^+)$. Since $[0,T]$ is compact, we may extract a convergent subsequence (still denoted by $\{t_n\}_n$) with $t_n\rightarrow s_*\in[0,T]$. Furthermore, by Prokhorov's theorem, there exists a further subsequence of $\{\mu_{n,t_n}\}_n$ (which we still denote by $\{\mu_{n,t_n}\}_n$ for simplicity) that converges weakly to some probability measure $\nu_{s_*}$. We now verify that $\nu_{s_*}$ is an invariant measure for the one-step transition probability kernel $\mathcal{P}(s_*,s_*+T)$. It suffices to show that for any $h \in C_{b,  lip }^{C y l}(\ell_\rho^{2\theta}(\mathbb{Z}^+))$,  
    \begin{equation}
        \label{invriantforonestep}
        \mathcal{P}^*(s_*,s_*+T)\nu_{s_*} h=\int_{\ell^{2\theta}_{\rho}(\mathbb{Z}^+)}\int_{\ell^{2\theta}_{\rho}(\mathbb{Z}^+)}h(y)\mathcal{P}(s_*,s_*+T,x,dy)\nu_{s_*}(dx)=\nu_{s_*} h.
    \end{equation}
    Observing the periodicity of $\mu_{n,\cdot}$, we have $\mathcal{P}_n^*(s,s+T)\mu_{n,s}(h\circ\chi_n)=\mu_{n,s}(h\circ\chi_n)$ for $s\geq 0$, which yields 
    \begin{align}
        &\int_{\ell^{2\theta}_{\rho}(\mathbb{Z}^+)}\int_{\ell^{2\theta}_{\rho}(\mathbb{Z}^+)}h(y)\mathcal{P}_n(s_n,s_n+T,x,dy)\mu_{n,s_n}\circ\chi_n^{-1}(dx) \nonumber\\
        &=\int_{\mathbb{R}^n}\int_{\mathbb{R}^n}h\circ\chi_n(y^{(n)})\mathcal{P}_n(s_n,s_n+T,x^{(n)},dy^{(n)})\mu_{n,s_n}(dx^{(n)})\nonumber\\
        &=\int_{\mathbb{R}^n}h\circ\chi_n(x^{(n)})\mu_{n,s_n}(dx^{(n)}). \label{finiteonestepinvariant}
    \end{align}
    Note that the limit of the right-hand side of Equation (\ref{finiteonestepinvariant}) is $\nu_{s_*} h$. To establish (\ref{invriantforonestep}), we only need to prove that the limit of the left-hand side of (\ref{finiteonestepinvariant}) is $\int_{\ell^{2\theta}_{\rho}(\mathbb{Z}^+)}\mathcal{P}^*(s_*,s_*+T)\nu_{s_*} h$. For this purpose, we adopt the proof technique of Theorem \ref{Markov property} to demonstrate
    \begin{equation*}
             x_n,x\in\ell^{2\theta}_{\rho}(\mathbb{Z}^+)\quad\text{and}\quad  x_n\rightarrow x, \quad \text{in}\quad \ell^2_{\rho}(\mathbb{Z}^+) \quad \text{implies}\quad H_n(x_n)\rightarrow H(x),\quad \text{as}\quad n\rightarrow \infty,
    \end{equation*}
    where $H_n(x):=\int_{\ell^{2\theta}_{\rho}(\mathbb{Z}^+)}h(y)\tilde{\mathcal{P}}_n(s_n,s_n+T,x,dy)$ and $H(x):=\int_{\ell^{2\theta}_{\rho}(\mathbb{Z}^+)}h(y)\mathcal{P}(s_*,s_*+T,x,dy)$. We observe that by an argument analogous to Lemma \ref{contionuouswithinitialconditon}, for any $\epsilon>0$, there exist $\delta_1>0$ and $N_0\in \mathbb{N}$  such that for all $n>N_0$ and $y\in\{z\in\ell^{2\theta}_{\rho}(\mathbb{Z}^+);\Vert z-x\Vert_{2,\rho}\leq \delta_1\}$, the following estimate holds: 
    \begin{equation}
             \label{infiniteperiodicinq1}
             \mathbb{E}\Vert X_{(k)}^{(n),s_n,x}(s_n+T)-X_{(k)}^{(n),s_n,y}(s_n+T)\Vert_{\mathbb{R}^{\Lambda_k}}^2\leq \frac{\epsilon}{64L^2},
    \end{equation}
    where $L$ is the Lipschitz constant for $g$, the finite-dimensional mapping corresponding to $h$. Furthermore, since $x_n\rightarrow x$, there exists $N_1\in\mathbb{N}$ such that $\vert x_n-x\vert <\delta_1$ for all $n>N_1$. Using a method similar to Lemma \ref{ineqXn}, we obtain the following estimate,
    \begin{align}
             \mathbb{E}\Vert\tilde{X}_{(k)}^{(n),s_n,x}(t)- \tilde{X}_{(k)}^{(n),s_n,x}(s)\Vert^2_{\mathbb{R}^k}\nonumber&\leq \mathbb{E}\Vert \tilde{X}^{(n),s_n,x}(t)- \tilde{X}^{(n),s_n,x}(s)\Vert^{2}_{2,\rho} \frac{1}{\inf\limits_{1\leq i\leq k}\rho(i)}\nonumber\\
             &\leq C(1+\Vert x\Vert^{2\theta}_{2\theta,\rho})\vert t-s\vert.\label{inftnitperiodict-sinq}
    \end{align}
    Consequently, there exists $\delta_2=\frac{\epsilon}{64LC(1+\Vert x\Vert^{2\theta}_{2\theta,\rho})}$ such that whenever $\vert t-s\vert <\delta_2$, we have 
    \begin{equation}
              \label{infiniteperiodicinq2}
              \mathbb{E}\Vert\tilde{X}_{(k)}^{(n),s_n,x}(t)- \tilde{X}_{(k)}^{(n),s_n,x}(s)\Vert^2_{\mathbb{R}^k}\leq \frac{\epsilon}{64L^2}.
    \end{equation}
    Finally, given that $s_n\rightarrow s_*$, we can find $N_2\in\mathbb{N}$ for which $\vert s_n-s_*\vert <\delta_2$ holds for all $n>N_2$. Moreover, in the case where $s_n < s_*$ (the situation $s_n > s_*$ can be handled similarly), we obtain that
    \begin{equation*}
              \mathbb{E}\Vert\tilde{X}_{(k)}^{(n),s_n,x}(s_*+T)- \tilde{X}_{(k)}^{(n),s_*,x}(s_*+T)\Vert^2_{\mathbb{R}^k}= \mathbb{E}\Vert\tilde{X}_{(k)}^{(n),s_*,\tilde{X}_{(k)}^{(n),s_n,x}(s_*)}(s_*+T)- \tilde{X}_{(k)}^{(n),s_*,x}(s_*+T)\Vert^2_{\mathbb{R}^k}.
    \end{equation*}
    Combining this observation with Remark \ref{remarkoncontionuouswithinitialconditon} and (\ref{infiniteperiodicinq2}), we conclude that there exists a constant $N_3\in\mathbb{N}$ such that 
    \begin{equation}\label{infiniteperiodicinq3}
        \mathbb{E}\Vert\tilde{X}_{(k)}^{(n),s_n,x}(s_*+T)- \tilde{X}_{(k)}^{(n),s_*,x}(s_*+T)\Vert^2_{\mathbb{R}^k}\leq \frac{\epsilon}{64L^2},
    \end{equation}
    holds for all $n>N_3$. Furthermore, the weak convergence property provides another threshold $N_4\in\mathbb{N}$ that ensures
    \begin{equation}
        \label{infiniteperiodicinq4}
        \vert\mathbb{E} h(\tilde{X}^{(n),s_*,x}(s_*+T))- \mathbf{E}^{s_*,x}h(X(s_*+T))\vert^2\leq \frac{\epsilon}{64},
    \end{equation}
    when $n>N_4$.
    Synthesizing results (\ref{infiniteperiodicinq1}), (\ref{infiniteperiodicinq2}), (\ref{infiniteperiodicinq3}), and (\ref{infiniteperiodicinq4}), we conclude that for the given $\epsilon>0$, we have whenever $n>N_0\vee N_1\vee N_2\vee N_3\vee N_4$
         \begin{align*}
             \vert H_n(x_n)-H(x) \vert^2&=\vert \mathbb{E}h(\tilde{X}^{(n),s_n,x_n}(s_n+T))-\mathbf{E}^{s_*,x}h(X(s_*+T))\vert^2\\
             &\leq 16\left( L^2\mathbb{E}\Vert\tilde{X}_{(k)}^{(n),s_n,x_n}(s_n+T)- \tilde{X}_{(k)}^{(n),s_n,x}(s_n+T)\Vert^2_{\mathbb{R}^k}\right.\\
             &\quad+L^2\mathbb{E}\Vert\tilde{X}_{(k)}^{(n),s_n,x}(s_n+T)- \tilde{X}_{(k)}^{(n),s_n,x}(s_*+T)\Vert^2_{\mathbb{R}^k}\\
             &\quad+L^2\mathbb{E}\Vert\tilde{X}_{(k)}^{(n),s_n,x}(s_*+T)- \tilde{X}_{(k)}^{(n),s_*,x}(s_*+T)\Vert^2_{\mathbb{R}^k}\\
             &\quad +\left. \vert\mathbb{E} h(\tilde{X}^{(n),s_*,x}(s_*+T))- \mathbf{E}^{s_*,x}h(X(s_*+T))\vert^2 \right)\\
             &\leq \epsilon.
         \end{align*}
         Following the proof of Theorem \ref{Markov property}, we establish (\ref{invriantforonestep}). By Lemma 3.1 of Feng, Zhao and Zhong \cite{fengzhaozhong2023existence}, there exists a $T$-periodic measure for $\mathcal{P}(s,t)$, which completes the proof.
    \end{proof}
    \section{Proof of Theorem \ref{finitdimensionalergodicresult}}
    In this section, our aim is to utilize the method proposed by Meyn and Tweedie to establish the geometric ergodicity of the SDE (\ref{FSDE}). Their theorem is stated as follows (see  Meyn and Tweedie \cite{meyn2012markov} and Mattingly, Stuart and Higham \cite[Theorem 2.5]{mattingly2002ergodicity}).
	\begin{theorem}
		\label{MeynT}
		(Harris-Meyn-Tweedie). Let $X(t)$ be a time-homogeneous Markov process on $\mathbb{R}^{n}$ with transition probability kernel $\mathcal{P}_t$ and infinitesimal generator $\mathscr{L}$. Suppose that following hypotheses are satisfied
		\item[(i)] The Markov process is irreducible, that is, there exists $t_0$(and then for all $t > t_0 $ ) such that
		$$
		\mathcal{P}_{t_0}(x,\Gamma)> 0,
		$$
		for all $x \in \mathbb{R}^n$ and nonempty open sets $\Gamma\subset\mathbb{R}^n$.
		\item[(ii)] For any $t> 0$ the Markov semi-group $\mathcal{P}_t(x,dy)$ has a density $p_t(x,y)$, which is a continuous function of $(x,y)$.\\
		Assume there exists a Lyapunov function
		$$
		V:\mathbb{R}^n \rightarrow \left[ 0,\infty\right) \xrightarrow{\Vert x \Vert_{\mathbb{R}^n} \rightarrow +\infty}+\infty
		$$
		and constants $C,c> 0$ such that 
		$$
		\mathscr{L}V+cV\leq C.
		$$
		Then there exists unique invariant measure $\mu$ for the process $X(t)$ and there exist constants $K,\alpha > 0$ such that
		$$
		\sup\limits_{\{f:\vert f(x) \vert \leq V(x)\}}\vert \mathcal{P}_tf(a)-\mu (f) \vert \leq KV(a)e^{-\alpha t}
		$$
		for any $a\in \mathbb{R}^n$.
	\end{theorem}
	
	In the remainder part of this section, we will show that the transition probability kernel $\mathcal{P}^{(n)}_t(x,\cdot)$ of the SDE (\ref{FSDE}) indeed satisfies the conditions of the above theorem. Thus, the ergodicity of the SDE in the finite-dimension case is obtained.
	
	\subsection{Strong Feller property} 
In this subsection, we prove that the transition probability kernel $\mathcal{P}^{(n)}_t(x,\cdot)$ associated with SDE (\ref{FSDE}) admits a smooth density function, which is the immediate consequence of the H\"ormander theorem in probabilistic settings (see Ichihara and Kunita \cite{ichihara1974classification}). 
	
	\begin{theorem}
		\label{hormandertheorem}
		(H\"ormander probabilistic setting, Ichihara - Kunita). Assume $X(t)$ is the strong unique  solution to the Stratonovich SDE
		$$
		dX_i(t) = b_i(t,X(t))dt + \sum_{j=1}^{d}\sigma_{i,j}(t,X(t))\circ dW_j(t),\quad t\geq 0,\quad i=1,\cdots, n
		$$
		where $b_i,\sigma_{i,j}\in C^{\infty}(\mathbb{R}^+\times\mathbb{R}^n,\mathbb{R}), 1\leq i,j \leq n,$. Suppose that following (H\"ormander) condition is satisfied
		$$
		dim(Lie\{A_0,A_1...,A_n\})=n+1, \ \forall x\in \mathbb{R}^n,
		$$
        where $A_0(t,x)=(1,b_1(t,x),\cdots,b_n(t,x))^T$ and $A_j(t,x)=(0,\sigma_{1,j},\cdots,\sigma_{n,j})^T$ for $1\leq j\leq n$.
		Then the transition probability kernel $\mathcal{P}_t(x,\cdot)$ associated with $X(t)$ admits a density function $p_t(x, y)$ such that 
        \begin{equation*}
            p_t(x, y)\in C^\infty (\mathbb{R}^+\times\mathbb{R}^n\times\mathbb{R}^n,\mathbb{R}).
        \end{equation*}
	\end{theorem}
	Since SDE (\ref{FSDE}) in the Stratonovich sense is actually the same as in the It$\hat{\text{o}}$ form, we can calculate that the corresponding $A_0,A_i,1\leq i \leq d$ in Theorem \ref{hormandertheorem} can be represented as

	$$
	A_0(x)=\begin{pmatrix}
	    1\\
        A_nx
	\end{pmatrix}+\begin{pmatrix}
		0\\
        f(x_1)\\
		f(x_2)\\
		\vdots \\
		f(x_n)
	\end{pmatrix}, \quad
	A_1(x)=\begin{pmatrix}
		0\\
        1\\
		0\\
		\vdots \\
		0
	\end{pmatrix}.
	$$
	Let 
	$$
	[V,W]\equiv DV(x)W(x)-DW(x)V(x)
	$$
	denote the Lie bracket of two $C^1$ vector fields $V$ and $W$. We can calculate that
	$$
	[A_0,A_1]=DA_0(x)A_1(x)-DA_1(x)A_0(x)=\begin{pmatrix}
		0\\
        b_{1,2} \\
		a_{2,1} \\
		0\\
		\vdots \\
		0\\
	\end{pmatrix},
	[A_0 ,  [A_0,A_1] ]= \begin{pmatrix}
		0\\
        b_{1,3}\\
		b_{2,3} \\
		a_{2,1}a_{3,2}\\
		0 \\
		\vdots \\
		0\\
	\end{pmatrix},
	$$
	where 
    \begin{equation*}
        b_{1,2}=f^\prime(x_1),b_{1,3}=\left(f^\prime(x_1)\right)^2+a_{1,2}a_{2,1}-f^{\prime \prime}(x_1)\left(a_{1,2}x_2+f (x_1)\right),b_{2,3}=a_{2,1}\left(f^\prime(x_1)+f^\prime(x_2)\right).
    \end{equation*} By iterating, we can get
	$$
	\underbrace{[A_0,[A_0,\cdots[A_0,[A_0,A_1]]]]}_{i}=\begin{pmatrix}
		0\\
        b_{1,i}\\
		b_{2,i}\\
		\vdots \\
		b_{i-1,i}\\
		a_{2,1}a_{3,2}\cdots a_{i,i-1}\\
		0\\
		\vdots \\
		0\\
	\end{pmatrix},1< i\leq n.
	$$
	The form of $b_{k,i}$ has no effect on verifying the H\"ormander condition, so we will not give specific formulas here. According to Assumptions \ref{A1} $(i)$, the above calculations show that SDE (\ref{FSDE}) satisfies the H\"ormander condition, i.e.
	$$
	dim(Lie\{A_0(x),A_1(x)\})=n+1, \ \forall x\in \mathbb{R}^n.
	$$
	As a consequence of Ichihara-Kunita's result, we conclude that the density function of the transition probability kernel corresponding to Equation (\ref{FSDE}) is smooth.
	
	\subsection{Irreducibility}
	To investigate the irreducibility of $X^{(n),x}(t)$, we would like to use ``approximate controllability" method. The question is whether we can solve the corresponding control problem,
	\begin{equation}
		\label{CP}
		\begin{cases}
			y^\prime (t) = A_ny(t) +F_n(y(t))+B_nu(t), \\
			y(0)=x\in \mathbb{R}^n,\\
		\end{cases}  
	\end{equation}
	where $A_n,F_n,B_n$ is given by (\ref{AnFnBn}), and $u$ is the control. We denote by $y(\cdot , x;u)$ the solution of (\ref{CP}). We recall that the system (\ref{CP}) is approximately controllable in time $T> 0$ if, for arbitrary $x, z \in \mathbb{R}^n$ and $\epsilon > 0$, there exists $u \in L^2(0,T;\mathbb{R}^n)$ such that $\vert y(T,x;u)-z \vert \leq \epsilon$. To prove the controllability of (\ref{CP}), we need the following lemma.
	\begin{lemma}
		\label{lemmaonderivative}
		For arbitrary constants $q_1,q_2,a_1,b_1,a_2,b_2$, there exists a function $g\in$  $C^\infty([0,T];\mathbb{R})$ such that
		\begin{equation}
			\label{lemmacondition}
			g(0)=q_1, g(T)=q_2,g^\prime(0)=a_1,g^\prime(T)=a_2,g^{\prime\prime}(0)=b_1,g^{\prime\prime}(T)=b_2.
		\end{equation}
	\end{lemma}
	\begin{proof}
		According to the setting of this lemma, we construct $\tilde{g}$ piecewise in $0<t_1<t_2<T$ as follows
		\begin{equation*}
			\tilde{g}(t)=\begin{cases}
				k_1t^2+k_2t+k_3,t\in [0,t_1],\\
				l_1t^2+l_2t+l_3,t\in[t_2,T],\\
				m_1+\frac{t-t_1}{t_2-t_1}(m_2-m_1),t\in[t_1,t_2],\\
			\end{cases}
		\end{equation*}
		where $m_1=k_1t_1^2+k_2t_1+k_3,m_2=l_1t_2^2+l_2t_2+l_3$ and $k_1,k_2,k_3,l_1,l_2,l_3$ are coefficients to be determined. In order for $\tilde{g}$ to satisfy (\ref{lemmacondition}), we have
		\begin{equation}
			\label{linearequation}
			\begin{cases}
				\tilde{g}(0)=k_3=q_1,\\
				\tilde{g}(T)=l_1T^2+l_2T+l_3=q_2,\\
				\tilde{g}^\prime(0)=k_2=a_1,\\
				\tilde{g}^\prime(T)=2l_1T+l_2=a_2,\\
				\tilde{g}^{\prime\prime}(0)=2k_1=b_1,\\
				\tilde{g}^{\prime\prime}(T)=2l_1=b_2.\\
			\end{cases}
		\end{equation}
		Solving the linear equation (\ref{linearequation}), we get
		\begin{equation*}
			k_1=\frac{b_1}{2},k_2=a_1,k_3=q_1,l_1=\frac{b_2}{2},l_2=a_2-b_2T,l_3=q_2-a_2T-\frac{b_2T^2}{2}+b_2T^2.
		\end{equation*}
		In the following, the function $g$ is obtained by the mollifier method. We can split the interval $[0,T]$ into three parts, $V_1=[0,t_1^\prime),H_2=(t_2^\prime,t_3^\prime),V_3=(t_4^\prime,T]$, where $0<t_2^\prime<t_1^\prime<t_1<t_2<t_4^\prime<t_3^\prime<T$. The mollification of $\tilde{g}$ in $V_2=(t_2^\prime+\epsilon,t_3^\prime-\epsilon)$ is defined by $\tilde{g}^\epsilon\coloneqq \eta_\epsilon\star\tilde{g}$. By the property of mollifiers, we have $\tilde{g}^\epsilon\in C^\infty(V_2)$. We can choose $\epsilon$ small enough such that $t_2^\prime+\epsilon<t_1^\prime, t_4^\prime<t_3^\prime-\epsilon$. Now let $\{\zeta_i\}_{i=1}^3$ be a smooth partition of unity subordinate to the sets $V_1,V_2,V_3$; that is, suppose
		\begin{equation*}
			\begin{cases}
				0\leq\zeta_i\leq1,\ \zeta_i\in C_c^\infty(V_i)\\
				\zeta_1+\zeta_2+\zeta_3=1, \ \  \text{on}\ [0,T].
			\end{cases}
		\end{equation*}
		We define $g=\zeta_1\tilde{g}\mathbf{1}_{V_1}+\zeta_2\tilde{g}^\epsilon\mathbf{1}_{V_2}+\zeta_3\tilde{g}\mathbf{1}_{V_3}$. Then clearly $g\in C^\infty([0,T];\mathbb{R})$ and $g$ satisfies (\ref{lemmacondition}).
	\end{proof}

	\begin{lemma}
    \label{lemmaapproximatelycontrolable}
		Let Assumption \ref{A1} hold. For every time $T> 0$,  the system (\ref{CP}) is approximately controllable in time $T$. 
	\end{lemma}
	\begin{proof}
		For simplicity, we only prove the three-dimensional case. For arbitrary $T> 0$, $x=(x_1 , x_2 , x_3)^\top \in \mathbb{R}^3$, $z=(z_1 , z_2 , z_3)^\top \in \mathbb{R}^3$ and $u=(u_1 , u_2, u_3)^\top \in L^2(0,T;\mathbb{R}^3)$ , we can write $y(T , x;u)$ as
		\begin{equation}
			\label{lemmaequation3dimensional}
			y(T,x;u)=x+ \begin{pmatrix}
				\int_0^T a_{1,2} y_2(t)+f(y_1(t))dt \\
				\int_0^T a_{2,1}y_1(t)+a_{2,3}y_3(t)+f(y_2(t))dt\\
				\int_0^T a_{3,2}y_2(t)+f(y_3(t))dt\\
			\end{pmatrix} + \begin{pmatrix}
				\int_{0}^{T}u_1(t)dt\\
				0\\
				0\\
			\end{pmatrix}.	
		\end{equation}
		We will construct $\bar{y}(t)=(\bar{y}_1(t) ,\bar{y}_2(t) , \bar{y}_3(t))^\top\in C^\infty([0,T];\mathbb{R}^3)$ to satisfy the following conditions:
		\begin{equation}
			\label{conditiony}
			\begin{cases}
				\bar{y}_1(0)=x_1, \bar{y}_1(T)=z_1, \bar{y}_2(0)=x_2, \bar{y}_2(T)=z_2, \bar{y}_3(0)=x_3, \bar{y}_3(T)=z_3,	\\
				a_{2,1}\bar{y}_1(t)=\bar{y}^{\prime}_2(t)-a_{2,3}\bar{y}_3(t)-f(\bar{y}_2(t)),\\
				a_{3,2}\bar{y}_2(t)=\bar{y}^{\prime}_3(t)-f(\bar{y}_3(t)).\\	
			\end{cases}		
		\end{equation}	
		Observe that we only need to construct $\bar{y}_3(t)$. According to the above conditions we deduce that
		\begin{equation}
			\label{conditiony2}
			\begin{cases}
				a_{2,1}x_1=\bar{y}^{\prime}_2(0)-a_{2,3}x_3-f(x_2),\\
				a_{2,1}z_1=\bar{y}^{\prime}_2(T)-a_{2,3}z_3-f(z_2).\\  		
			\end{cases}
		\end{equation}
		Therefore, in order to construct $\bar{y}(t)$ we just need construct $\bar{y}_3(t)$ satisfying
		\begin{equation}
			\label{conditiony3}
			\begin{cases}
				a_{3,2}\bar{y}^{\prime}_2(0)=\bar{y}^{\prime\prime}_3(0)-f^\prime(x_3)\bar{y}^{\prime}_3(0),\\
				a_{3,2}\bar{y}^{\prime}_2(T)=\bar{y}^{\prime\prime}_3(T)-f^\prime(z_3)\bar{y}^{\prime}_3(T),\\
				a_{3,2}x_2=\bar{y}^{\prime}_3(0)-f(x_3),\\
				a_{3,2}z_2=\bar{y}^{\prime}_3(T)-f(z_3).\\
			\end{cases}
		\end{equation}
		By (\ref{conditiony}), (\ref{conditiony2}) and (\ref{conditiony3}), we have
		\begin{equation}
			\label{ybar3condition}
			\begin{cases}
				\bar{y}_3(0)=x_3, \bar{y}_3(T)=z_3,	\\
				\bar{y}_3^{\prime}(0)=a_{3,2}x_2+f(x_3),\\
				\bar{y}_3^{\prime}(T)=a_{3,2}z_2+f(z_3),\\		
                \bar{y}^{\prime\prime}_3(0)=a_{3,2}\left(a_{2,1}x_1+a_{2,3}x_3+f(x_2)\right)+f^\prime(x_3)\left(a_{3,2}x_2+f(x_3)\right),\\
				\bar{y}^{\prime\prime}_3(T)=a_{3,2}\left(a_{2,1}z_1+a_{2,3}z_3+f(z_2)\right)+f^\prime(z_3)\left(a_{3,2}z_2+f(z_3)\right).\\
			\end{cases}
		\end{equation}
		By lemma \ref{lemmaonderivative}, there exists a function $\bar{y}_3\in C^\infty([0,T];\mathbb{R})$ satisfying (\ref{ybar3condition}). Then, let 
        \begin{equation*}
			\begin{cases*}
				\bar{y}_2(t)=\frac{1}{a_{3,2}}\left(\bar{y}^{\prime}_3(t)-f(\bar{y}_3(t))\right),\\
				\bar{y}_1(t)=\frac{1}{a_{2,1}}\left(\bar{y}^{\prime}_2(t)-a_{2,3}\bar{y}_3(t)-f(\bar{y}_2(t))\right).\\
			\end{cases*}
		\end{equation*}
		Thus, we have constructed the desired $\bar{y}(t)$. Next, if we set
		\begin{equation*}
			u_1(t)=\bar{y}^{\prime}_1(t)-a_{1,2}\bar{y}_2(t)-f(\bar{y}_1(t))
		\end{equation*}
		Notice that $\bar{y}$ is a solution to Equation (\ref{lemmaequation3dimensional}). By the properties of $\bar{y}$ , we deduce the approximate controllability of Equation (\ref{lemmaequation3dimensional}).	
	\end{proof}
    \begin{remark}
        In the proof of Lemma \ref{lemmaapproximatelycontrolable}, we can derive a stronger result than that stated in the lemma itself: for any given $x,z\in \mathbb{R}^n$, there exists a control $u \in L^2(0,T;\mathbb{R}^n)$ such that the system (\ref{CP}) satisfies $y(0)=x$, $y(T)=z$.
    \end{remark}
    According to Theorem 7.4.1 of Da Prato and Zabczyk   \cite{da1996ergodicity}, we can conclude that the transition probability kernel $\mathcal{P}^{(n)}_t(x,\cdot)$ corresponding to (\ref{FSDE}) is irreducible.
	\subsection{Lyapunov condition}
	The existence of the Lyapunov function for the corresponding infinitesimal generator of the SDE (\ref{FSDE}) can be proved by structural calculation. 
	\begin{lemma}
		Let $\mathscr{L}_n$ be the infinitesimal generator of  SDE (\ref{FSDE}) under Assumption \ref{A1}. For the function $V_n(x)=1+x_1^2+\cdots +x_n^2$, there exist constants $C,c> 0$ such that 
		\begin{equation}
			\label{Lyapucon}
			\mathscr{L}_nV_n(x)+cV_n(x)\leq C , \forall x\in \mathbb{R}^n.
		\end{equation}   	
	\end{lemma}
	% the result to V of the form $V(x)=1+\sum_{i=1}^{n}\rho(i)\vert x_i \vert^2$
	\begin{proof}
		The corresponding infinitesimal generator of SDE (\ref{FSDE}) is of the following form
		\begin{align*}
		    \mathscr{L}_nh(x)&=(a_{1,2} x_2+f(x_1))\partial_{x_1}h(x)+(a_{2,1}x_1+a_{2,3} x_3+f(x_2))\partial_{x_2}h(x)+\cdots \\
            &\quad +(a_{n,n-1} x_{n-1} +f(x_n))\partial_{x_n}h(x) + \frac{1}{2} \frac{\partial^2}{\partial_{x_1^2}}h(x), \quad h\in C^2(\mathbb{R}^n),\quad x\in\mathbb{R}^n.
		\end{align*}
		We compute that 
		\begin{align*}
		    \mathscr{L}_nV_n(x)&=2(a_{1,2} x_2+f(x_1))x_1+\cdots+2(a_{n,n-1} x_{n-1} +f(x_n))x_n+1\\
			&\leq 4M\left(\vert x_1\vert \vert x_2\vert + \cdots+\vert x_{n-1}\vert \vert x_n\vert \right)  +2f(x_1)x_1+\cdots +2f(x_n)x_n+1,
		\end{align*}
		where $M=\max \left\{\vert a_{i-1,i}\vert ,\vert a_{i,i+1} \vert; 2\leq i\leq n \right\}$.
		According to Assumption \ref{A1} (ii), we get
        \begin{equation*}
            \begin{split}
                &\mathscr{L}_nV_n(x)+V_n(x)\\
                &\leq 4M\left(\vert x_1\vert \vert x_2\vert + \cdots+\vert x_{n-1}\vert \vert x_n\vert \right)  -(2\lambda-1)\left(\vert x_1\vert^2 +\cdots+\vert x_n\vert^2 \right)+2n\eta+2\\
			&\leq4M\left(\vert x_1\vert^2+  \cdots+ \vert x_n\vert^2 \right)  -(2\lambda-1)\left(\vert x_1\vert^2 +\cdots+\vert x_n\vert^2 \right)+2n\eta +2\\
			&\leq2n\eta +2.
            \end{split}
        \end{equation*}
		We take $c=1$ and $C=2n\eta+2$, then the result of the lemma follows.
	\end{proof}
	\begin{remark}
		We can get a more general form of the Lyapunov function $ V_n(x)=1+\sum_{i=1}^{n} \rho(i)\vert x_i \vert^{2\theta}$, where $\rho(i)>0,\forall i$, and $\theta\in \mathbb{N}$. 
	\end{remark}
    \begin{proof}[Proof of Theorem \ref{finitdimensionalergodicresult}]
        Combining the discussions in Subsections 6.1, 6.2 and 6.3, Theorem \ref{MeynT} (the Meyn-Tweedie theory) applies to give that the transition probability kernel $\mathcal{P}^{(n)}_t(x,\cdot)$ of SDE (\ref{FSDE}) exponentially converges to a unique invariant measure. The demonstration of Theorem \ref{finitdimensionalergodicresult} is hereby concluded.
    \end{proof}
	
	\appendix
\renewcommand\thesection{\normalsize Acknowledgements}
\section{}
We acknowledge the financial supports of EPSRC grant ref EP/S005293/2 and Royal Society Newton Fund grant (ref. NIF\textbackslash R1\textbackslash 221003).
\newpage
	%%%%%%%%%%%%%%%%%%%%%%%%%%%
    \bibliographystyle{siam}

    \addcontentsline{toc}{chapter}{\\ {\bf References} \hskip13.4cm}
    
%\newpage
\footnotesize
%\bibliography{ref}
\bibliography{bmyref}

\end{document}